\documentclass[11pt]{article}
\usepackage{latexsym,amsmath,amscd,amssymb,framed,graphics,mathrsfs}
\usepackage{float} 
\usepackage{enumerate}
\usepackage{graphicx}
\usepackage{subcaption}
\usepackage{color}
\usepackage{framed}
\usepackage{amsmath}
\usepackage{multirow}
\usepackage[colorlinks]{hyperref}
\usepackage{url}
\usepackage[all]{xy}
\usepackage{tikz}
\usetikzlibrary{snakes}

\newcommand{\rem}[1]{}
\makeatletter

\newcommand\x{\times}

\newcommand\pa{\partial}

\newcommand\dd{{\bf d}}

\newcommand\F{\mathcal F}

\newenvironment{proof}[1][Proof]{\noindent\textbf{#1.} }{\ \rule{0.5em}{0.5em}}

\def\XXint#1#2#3{{\setbox0=\hbox{$#1{#2#3}{\int}$ }
\vcenter{\hbox{$#2#3$ }}\kern-.5\wd0}}

\definecolor{bgd}{RGB}{153,0,51}      

\newcommand{\todo}[1]{\vspace{5 mm}\par \noindent
\framebox{\begin{minipage}[c]{0.95 \textwidth}
\tt #1 \end{minipage}}\vspace{5 mm}\par}

\begin{document}

\newtheorem{theorem}{Theorem}[section]
\newtheorem{definition}[theorem]{Definition}
\newtheorem{lemma}[theorem]{Lemma}
\newtheorem{remark}[theorem]{Remark}
\newtheorem{proposition}[theorem]{Proposition}
\newtheorem{corollary}[theorem]{Corollary}
\newtheorem{example}[theorem]{Example}



\title{Variational integrators for stochastic Hamiltonian systems on Lie groups: properties and convergence}
\author{Fran\c{c}ois Gay--Balmaz$^{1}$ and Meng Wu$^{1,2}$}

\addtocounter{footnote}{1}

\footnotetext{Division of Mathematical Sciences, Nanyang Technological University, 21 Nanyang Link, Singapore 637371.
\texttt{francois.gb@ntu.edu.sg}
\addtocounter{footnote}{1} }

\footnotetext{LMD, \'Ecole Normale Sup\'erieure, Paris, France.
\texttt{meng.wu@lmd.ipsl.fr}
\addtocounter{footnote}{1}}

\date{ }
\maketitle
\makeatother



\begin{abstract}
We derive variational integrators for stochastic Hamiltonian systems on Lie groups using a discrete version of the stochastic Hamiltonian phase space principle. The structure-preserving properties of the resulting scheme, such as symplecticity, preservation of the Lie-Poisson structure, preservation of the coadjoint orbits, and conservation of Casimir functions, are discussed, along with a discrete Noether theorem for subgroup symmetries. We also consider in detail the case of stochastic Hamiltonian systems with advected quantities, studying the associated structure-preserving properties in relation to semidirect product Lie groups. A full convergence proof for the scheme is provided for the case of the Lie group of rotations. Several numerical examples are presented, including simulations of the free rigid body and the heavy top.
\end{abstract}

\section{Introduction}

This work is motivated by recent advances in the use of Lie group variational methods for both stochastic Hamiltonian modeling \cite{Ho2015,GBHo2018} and structure-preserving discretization \cite{GaGB2020,GaGB2021} in fluid dynamics. These developments are founded on  either stochastic or discretized versions of the classical Hamilton principle of Lagrangian mechanics.
In this context, an efficient way to derive structure preserving numerical methods for stochastic Hamiltonian models is to rely on a discretization of the underlying stochastic variational formulation. This is the approach we develop in this paper for the case of finite dimensional Lie groups.
 
Methods based on discretized variational principles, known as variational integrators \cite{MaWe2001}, are well-established for providing a constructive approach to the derivation of geometry-preserving schemes \cite{HaLuWa2006}. These methods guarantee, among other things, symplecticity and the satisfaction of a discrete version of Noether’s theorem in the presence of symmetries.

\medskip

In this paper, we develop these methods for stochastic Hamiltonian systems on Lie groups. In addition to the usual Hamiltonian function, such systems also involve a collection of Hamiltonian functions driven by independent Wiener processes. In general, given a Poisson manifold $(P,\{\cdot, \cdot\})$ and smooth Hamiltonians $H,H_i: P\rightarrow\mathbb{R}$, for $i=1,...,N$, the associated stochastic Hamiltonian system is given by
\begin{equation}\label{SHS}
{\rm d}x = X_H(x) {\rm d}t +\sum_{i=1}^N X_{H_i}(x)\circ{\rm d}W_i(t),
\end{equation}
where $W_i(t)$, for $i = 1, ..., N$, are independent Wiener processes, \cite{Bi1982,LCOr2008}. The differential equation is understood in the Stratonovich sense. In equation \eqref{SHS}, $X_H$ denotes the Hamiltonian vector field associated with the function $H$, defined by ${\rm d}F\cdot X_H=\{F,H\}$, for all $F:P\rightarrow\mathbb{R}$, and similarly for $X_{H_i}$.

In this paper, we first focus on the case where the Poisson manifold $P$ is the cotangent bundle $T^*G$ of a Lie group $G$, endowed with the canonical Poisson bracket $\{F,H\}=\Omega(X_F,X_H)$, where $\Omega$ is the canonical symplectic form on $T^*G$. For discretization purposes, it is advantageous to rewrite the stochastic dynamics on the trivialized cotangent bundle $T^*G\simeq G\times \mathfrak{g}^*$, yielding the stochastic Hamiltonian system \eqref{SHS} in the equivalent form
\begin{equation}\label{Phasespace_Lie_intro}
\begin{aligned}
   & g ^{-1} {\rm d} g = \frac{\pa h}{\pa \mu }(g,\mu)  {\rm d} t + \sum^N_{i=1} \frac{\pa h_i}{\pa \mu }(g,\mu) \circ {\rm d} W_i(t), \\ & {\rm d} \mu - \operatorname{ad}^*_{ g ^{-1} {\rm d} g} \mu = - g ^{-1} \frac{\pa h}{\pa g}(g,\mu)  {\rm d} t - \sum^N_{i=1} g ^{-1} \frac{\pa h_i}{\pa g}(g,\mu) \circ {\rm d} W _i (t),
\end{aligned}
\end{equation}
for $(g,\mu)\in G\times\mathfrak{g}^*$, where $H, H_i:G \times\mathfrak{g}^*\rightarrow\mathbb{R}$, and $\mathfrak{g}^*$ is the dual of the Lie algebra $\mathfrak{g}$ of $G$. The notation will be explained in details later. This type of stochastic dynamics encompasses the models used in \cite{Ho2015,GBHo2018}.

Next, we focus on the case where the Hamiltonians are invariant under the action of a subgroup $K\subset G$, which, by Poisson reduction, leads to noncanonical stochastic Hamiltonian systems on the reduced Poisson manifold $P=(T^*G)/K$. For example, if $K=G$, then $(T^*G)/G$ is identified with $\mathfrak{g}^*$, and \eqref{Phasespace_Lie_intro} reduces to the stochastic Lie-Poisson system
\begin{equation}\label{LPS_intro}
{\rm d} \mu = \operatorname{ad}^*_{ \frac{\partial h}{\partial \mu}} \mu {\rm d}t +  \sum^N_{i=1}\operatorname{ad}^*_{ \frac{\partial h_i}{\partial \mu}} \mu\circ {\rm d} W _i (t)
\end{equation}
on $P=\mathfrak{g}^*$. If the symmetry subgroup is the isotropy group $K=G_{\alpha_0}$ of a representation of $G$ in a vector space $V$, \eqref{Phasespace_Lie_intro} reduces to a stochastic Lie-Poisson system on a semidirect product associated with this representation, written as:
\begin{equation}\label{SDP_intro}
({\rm d} \mu, {\rm d}\alpha) = \operatorname{ad}^*_{ \left(\frac{\partial h}{\partial \mu},  \frac{\partial h}{\partial \alpha}\right)} (\mu,\alpha)  {\rm d}t +  \sum^N_{i=1}\operatorname{ad}^*_{ \left(\frac{\partial h_i}{\partial \mu},  \frac{\partial h_i}{\partial \alpha}\right)} (\mu,\alpha)\circ {\rm d} W _i (t),
\end{equation}
with the notations to be explained later.

The main geometric properties of the general system \eqref{Phasespace_Lie_intro} mirror those of its deterministic counterpart and are as follows: (i) the flow preserves the symplectic form, and (ii) if the Hamiltonians are invariant under a symmetry, Noether's theorem holds.
For systems \eqref{LPS_intro} and \eqref{SDP_intro}, the main geometric properties are: (i) the flow preserves the Lie-Poisson structure, (ii) the flow restricts to coadjoint orbits, thereby preserving the Casimir functions, and (iii) on each coadjoint orbits, the flow preserves the orbit's symplectic form (the Kirillov-Kostant-Souriau form).

In this paper, we demonstrate that all of these properties hold at the discrete level for our stochastic variational integrator. In addition to its constructive nature, the variational approach provides a systematic framework for proving the structure-preserving properties of the integrator, such as symplecticity and satisfaction of the Noether theorem. The Poisson integrator property, along with the preservation of coadjoint orbits and conservation of the Casimir function, then follows by applying concepts from symmetry reduction.

In addition to its constructive nature, the variational approach also provides a systematic way to demonstrate the structure-preserving properties of the integrators, such as symplecticity and the satisfaction of Noether's theorem. In the Lie-Poisson case, \eqref{LPS_intro} and \eqref{SDP_intro} the Poisson integrator property of our scheme, along with the preservation of coadjoint orbits and conservation of the Casimir function, directly follows from the applying general ideas from reduction by symmetry. This approach offers a more unified framework for deriving these properties compared to previous works on Poisson integrators, such as \cite{HoRuSuWa2021}.

\paragraph{Previous works.}
There is a substantial body of work on structure-preserving discretization of stochastic Hamiltonian systems. For example, symplectic integrators for the simulation of stochastic Hamiltonian systems, associated with the canonical symplectic structure on vector spaces, have been derived and studied in \cite{MiReTr2002a,MiReTr2002,Mi2010,MaDiDi2012,DeAnSW2014,MaDi2015,WaHoXu2017,ZhZhHoSo2017}. An extension to non-canonical Poisson brackets were given in \cite{HoRuSuWa2021}, where Poisson integrators were derived. These integrators preserve the Casimir functions and ensure that their flow is a Poisson map. Further development and analysis of such integrators, based on a splitting strategy and with a focus on Lie-Poisson systems, were given in \cite{BrCoJa2023}.
In the area of variational discretization, stochastic variational integrators were first introduced in \cite{BROw2009} for a specific class of stochastic Hamiltonians, including systems on Lie groups. In \cite{HoTy2018}, higher-order symplectic integrators were obtained through a stochastic variational approach based on type II generating functions, specifically for stochastic canonical Hamiltonian systems on vector spaces.
In the deterministic case, variational integrators on Lie groups are well-established, with foundational works in \cite{BoSu1999a,BoSu1999b,MaPeSh1999,MaPeSh2000}, and have been widely applied to various problems, as demonstrated for example in \cite{Le2008,KoMa2011,DeGBKoRa2014,DeGBLeOBRaWe2015,CoGB2020}.

\paragraph{Content of the paper.} In \S\ref{vector_spaces}, we consider stochastic canonical Hamiltonian systems on vector spaces and provide a variational derivation of the stochastic midpoint method. While this method is already known to be symplectic \cite{MiReTr2002}, the advantage of a variational derivation is that it naturally extends to systems on Lie groups. In \S\ref{sec_3}, we extend this method by deriving a variational integrator for general stochastic Hamiltonian systems on Lie groups (see \eqref{Phasespace_Lie_intro}), without assuming any symmetries. We also provide a variational proof of the symplecticity of the scheme. In \S\ref{sec_4}, we examine the case where the Hamiltonians are invariant under a Lie subgroup $K\subset G$ and show that a discrete version of Noether's theorem holds. When $K=G$, we obtain stochastic Lie-Poisson systems and demonstrate that our scheme is a Poisson integrator, preserving the Lie-Poisson bracket, coadjoint orbits, the Kirillov-Kostant-Souriau symplectic form, and Casimir functions. In \S\ref{Section_4}, we consider the case where the symmetry subgroup is the isotropy subgroup of a representation of $G$, which is particularly relevant in applications such as heavy top dynamics and compressible fluid models. We show that all the properties derived in \S\ref{sec_4} hold in this more general setting, when expressed in a semidirect product Lie group framework. In \S\ref{sec_convergence}, we provide a full convergence proof of our scheme for the case $G=SO(3)$, applied to the free rigid body. Finally, \S\S\ref{rigid_body} and \ref{heavy_top} illustrate the results from \S\ref{sec_4} and \S\ref{Section_4}, respectively, using the free rigid body and the heavy top, including stochastic modeling of gyroscopic precession.

\section{Stochastic variational integrators on vector spaces}\label{vector_spaces}

In this section we present a variational derivation of the midpoint rule for stochastic canonical Hamiltonian systems on vector spaces. We show that it emerges as the critical point condition for an appropriate  discretization of the stochastic Hamilton phase space principle.
We then show how this discrete variational setting can be used to obtain the (well-known) symplecticity of the stochastic midpoint rule.
This discrete setting will serve as a guide for the variational discretization of stochastic Hamiltonian systems on Lie groups that we propose in \S\ref{sec_3}.

\subsection{Stochastic phase space principle on vector spaces}

Given Hamiltonian functions $H, H_i:T^*V \rightarrow \mathbb{R}$, $i=1,...,N$, with $V$ a vector space and $T^*V=V\times V^*$, we consider the associated stochastic Hamiltonian 
system
\begin{equation}\label{SHS} 
{\rm d} q = \frac{\partial H}{\partial p}  {\rm d} t + \sum^N_{i=1} \frac{\partial H_i}{\partial p} \circ {\rm d} W_i(t), \qquad {\rm d} p = -\frac{\partial H}{\partial q}  {\rm d} t - \sum^N_{i=1} \frac{\partial H_i}{\partial q} \circ {\rm d} W_i(t).
\end{equation} 
We assume that $H$ and $H^i$ are sufficiently smooth functions and satisfy the conditions ensuring that the stochastic differential equation \eqref{SHS} admits a pathwise unique solution on $[0,T]$ almost surely, for any given initial condition $(q_0,p_0) \in T^*V$, see \cite{KoPl1992}. The solution is a stochastic flow in the space 
\[
C([0,T]) = \{(q,p):\Omega \times [0,T] \rightarrow T^*V\,|\, (q,p) \; \text{are almost surely continuous}\}
\]
where $\Omega$ is the domain of probability. 



In a similar way with their deterministic counterpart, the stochastic Hamilton equations can be derived from the Hamilton phase space principle.  Consider the action functional $\mathcal{G} : C([0,T]) \rightarrow \mathbb{R}$ given by
\begin{equation}\label{CAF}
 \mathcal{G} (q,p) = \int_0^T \left\langle p, \circ \,{\rm d} q \right\rangle - H(q, p) {\rm d} t - \sum^N_{i=1} H_i(q, p) \circ {\rm d} W_i(t).
\end{equation}
Then, the critical point condition
\begin{equation}\label{SPS} 
\delta \mathcal{G} (q,p) =0,
\end{equation} 
where $q$ has fixed endpoints, characterizes the stochastic flow that solves the stochastic partial differential equations \eqref{SHS}, see \cite{LCOr2008,BROw2009,HoTy2018}.


For fixed endpoint values $(q_0, q_T) \in V \times V$, let $(\bar q(q_0, q_T),\bar p(q_0, q_T)) \in C([0,T])$ be the solution of \eqref{SHS} such that at the end points $q$ takes the values $q_0$ and $q_T$. Define
\[
\mathcal{S} (q_0, q_T) = \mathcal{G} \big(\bar q(q_0, q_T),\bar p(q_0, q_T)\big).
\]
According to Hamilton's phase space principle, this definition is equivalent to
\[
\mathcal{S} (q_0, q_T) = \underset{q(0)=q_0, q(T)=q_T}{\underset{(q,p) \in C([0,T])}{\textrm{ext}}}  \mathcal{G} (q,p)
\]
giving the extremum value of $\mathcal{G}$ over all the flows that satisfy $q(0)=q_0$ and $ q(T)=q_T$. The function
$\mathcal{S}(q_0, q_T)$ is called the Type-I generating function for the stochastic flow because of the generating property:
\begin{equation}\label{gf}
    p_0 = - \frac{\partial \mathcal{S}(q_0, q_T)}{\partial q_0} , \quad p_T = \frac{\partial \mathcal{S}(q_0, q_T)}{\partial q_T}.
\end{equation}

This is an analogue to the result given in \cite{HoTy2018}, where the stochastic Type-II generating property has been shown. This particular 
property of $\mathcal{S}(q_0, q_T)$ will provide us with the idea of how we can properly define the covectors $p$ at the end points for a discrete stochastic flow.

\subsection{Discrete stochastic phase space principle on vector spaces}
In order to obtain the stochastic midpoint integrator, we propose the following discrete version of the stochastic phase space principle \eqref{SPS}: 
\begin{equation}\label{dS0} 
{\rm d} \mathcal{G} _d( c_d) \cdot \delta c_d=0,
\end{equation} 
where the discrete curve takes the form:
\begin{equation}\label{discete_curve} 
c_d=\big(q_0, (\tilde q_1, \tilde p^1_1, \tilde p^2_1), q_1, .., q_{K-1}, (\tilde q_K, \tilde p^1_K, \tilde p^2_K), q_{K}\big) \in V \times \big(V \times (V^* \oplus V^*) \times V \big)^{K} =: C_d(K)
\end{equation}
and the discrete action functional is defined as 
\begin{equation}\label{discrete_PS2}
\begin{aligned} 
\mathcal{G} _d(c_d)&=\sum_{k=0}^{K-1} \Delta t \Big[ \Big\langle \tilde p^2_{k+1}, \frac{ q_{k+1}-\tilde q_{k+1}}{ \Delta t} \Big\rangle  + \Big\langle \tilde p^1_{k+1}, \frac{\tilde q_{k+1} -q_k}{\Delta t} \Big\rangle \\
&\quad - H \Big( \tilde q_{k+1}, \frac{ \tilde p^1_{k+1}+ \tilde p^2_{k+1}}{2} \Big)  - \sum^N_{i=1}H_i \Big( \tilde q_{k+1}, \frac{ \tilde p^1_{k+1}+ \tilde p^2_{k+1}}{2} \Big) \frac{ \Delta W_{k+1}^i}{ \Delta t}  \Big],
\end{aligned} 
\end{equation}
where $ \Delta W_{k+1}^i$ are the increments of the Wiener processes, i.e. $\Delta W_{k+1}^i = W_i(t_{k+1}) - W_i(t_{k})$.

\begin{remark}\rm 
The image below illustrates the midpoint scheme and the discrete action functional. The value of $q$ at the midpoint is denoted as $\tilde q_k$, and the covector $p$ coupled with the difference of $q$ in the first mid-step is denoted as $\tilde p^1_k$ (and $\tilde p^2_k$ coupled with the difference of $q$ in the second mid-step); both are covectors attached to $\tilde q_k$. In other words, $(\tilde q_k,\tilde p^1_k ) \in T_{\tilde q_k}^*V$ and $(\tilde q_k,\tilde p^2_k ) \in T_{\tilde q_k}^*V$. Here we insist on expressing $(\tilde q_k,\tilde p^1_k )$ and $(\tilde q_k,\tilde p^2_k )$ as pairs, clearly indicating onto which vector $q$ the covector $p$ is attached to in each case. This seems an irrelevant issue in the case of vector spaces, where the tangent and cotangent spaces of each vector element are equivalent and the transformations among them are trivial. This is not the case, however, in the case of Lie groups which is the center of this study.

The values of the covector $p$ at the end points do not appear explicitly in the definition of the discrete action functional. Thus no information of $p$ at the end points can be derived by simply applying the variational principle and taking the extremum of $\mathcal{G} _d(c_d)$. Such an aim will be achieved with the generating property of the action functional, now as a definition. See the next proposition for the details. The values of $p$ at other time steps ($t = 1, ... , K-1$) can also be defined intrinsically, and we will cover this in the proposition below and in the following sections.
\end{remark}

\begin{center}
\begin{tikzpicture}[snake=zigzag, line before snake = 5mm, line after snake = 5mm]
    \draw (0,0) -- (4,0);
    \draw[snake] (4,0) -- (7,0);
    \draw (7,0) -- (11,0);

    \foreach \x in {0,4,7,11}
      \draw (\x cm,5pt) -- (\x cm,-5pt);

    \foreach \x in {2,9}
      \draw (\x cm,3pt) -- (\x cm,-3pt);

    \draw (0,0) node[below=3pt] {$ t_0 =0 $} node[above=3pt] {$ (q_0, \textcolor{lightgray}{p_0}) $};
    \draw (2,0) node[below=3pt] {$ t_{\frac{1}{2}} $} node[above=19pt] {$ (\tilde q_1, \tilde p^1_1) $} node[above=3pt] {$ (\tilde q_1, \tilde p^2_1) $};
    \draw (4,0) node[below=3pt] {$ t_1 $} node[above=3pt] {$ (q_1, \textcolor{lightgray}{p_1})$};
    \draw (7,0) node[below=3pt] {$ t_{K-1} $} node[above=3pt] {$ (q_{K-1}, \textcolor{lightgray}{p_{K-1}}) $};
    \draw (9,0) node[below=3pt] {$ t_{K-\frac{1}{2}} $} node[above=19pt] {$ (\tilde q_K, \tilde p^1_K) $} node[above=3pt] {$ (\tilde q_K, \tilde p^2_K) $};
    \draw (11,0) node[below=3pt] {$ t_K = T$} node[above=3pt] {$ (q_K, \textcolor{lightgray}{p_K})  $};
  \end{tikzpicture}
\end{center}

\begin{proposition} Consider the critical condition for the discrete Hamilton phase space principle
\begin{equation}\label{psp2}
    {\rm d} \mathcal{G} _d(c_d) \cdot \delta c_d = 0,
\end{equation}
where $c_d$ is the discrete curve in the form \eqref{discete_curve} with the end points $q_0$ and $q_K$ fixed and $\mathcal{G} _d$ the discrete action functional defined in \eqref{discrete_PS2}. 
Denote with $\mathcal{S}_d $ the extremum value obtained from the variational principle for discrete curves with fixed $(q_0,q_K)$:
\begin{equation*}
   \mathcal{S}_d  (q_0,q_K) =  \underset{q_0, q_K \, \text{fixed}}{\underset{c_d \in C_d (K)}{ \operatorname{ext}}}  \mathcal{G}_d (c_d).
\end{equation*} 
Define $p_0$ and $p_K$ as:
\begin{equation}\label{GCn2}
   p_0 = -\frac{\partial \mathcal{S}_d (q_0,q_K)}{\partial q_0 },  \quad p_K = \frac{\partial \mathcal{S}_d (q_0,q_K)}{\partial q_K }.
\end{equation}
These are the discrete analogue of the generating property of $\mathcal{S}_d$ as in \eqref{gf}. Under such definition, the variational principle \eqref{psp2} yields a discrete stochastic flow $F_{0,T}: T^*V \rightarrow T^*V $, which satisfies $F_{0,T} (q_0,p_0) = (q_K,p_K) $. The flow is characterized by the midpoint method:
\begin{equation} \label{Stoch_Midpoint_VS}
\left\{ 
\begin{array}{l} 
\displaystyle\vspace{0.2cm}
\frac{ q_k - q_{k-1}}{ \Delta t} =  \frac{\partial H}{\partial p} \left( \frac{q_k+q_{k-1}}{2}, \frac{ p _{k}+p_{k-1}}{2} \right)+ \sum^N_{i=1}\frac{\partial H_i}{\partial p} \left( \frac{q_k+q_{k-1}}{2}, \frac{ p _{k}+p_{k-1}}{2}\right)\frac{ \Delta W_{k}^i}{ \Delta t}\\
\displaystyle \frac{p_k - p_{k-1}}{ \Delta t} =- \frac{\partial H}{\partial q} \left( \frac{q_k+q_{k-1}}{2}, \frac{ p _{k}+p_{k-1}}{2}\right) - \sum^N_{i=1}\frac{\partial H_i}{\partial q} \left( \frac{q_k+q_{k-1}}{2},\frac{ p _{k}+p_{k-1}}{2} \right)\frac{ \Delta W_{k}^i}{ \Delta t},
\end{array}
\right.
\end{equation}
where we use the definition $p_k = \tilde p^2_k$ for $k = 1, ..., K-1$.
\end{proposition}
\begin{proof}
We compute the derivative of $\mathcal{G} _d( c_d)$:
\begin{equation}\label{dS_d_proof} 
\begin{aligned}
&\frac{1}{\Delta t} {\rm d}\mathcal{G} _d( c_d) \cdot \delta c_d =\\
&\sum_{k=1}^{K}\Bigg[\Big\langle \frac{\tilde p^1_k-\tilde p^2_k}{\Delta t} - \frac{\pa \mathsf{h}_k}{\pa q}, \delta \tilde q_k \Big\rangle + \Big\langle \frac{\tilde q_k-q_{k-1}}{\Delta t}- \frac{1}{2} \frac{\pa \mathsf{h}_k}{\pa p}, \delta \tilde p^1_k  \Big\rangle + \Big\langle \frac{q_k - \tilde q_k}{\Delta t}- \frac{1}{2} \frac{\pa \mathsf{h}_k}{\pa p }, \delta \tilde p^2_k  \Big\rangle \Bigg]\\
& +\sum_{k=1}^{K-1}\Big\langle \frac{\tilde p^2_{k}-\tilde p^1_{k+1}}{\Delta t} , \delta q_k  \Big\rangle +\Big\langle \frac{\tilde p^2_K}{\Delta t} , \delta q_K \Big\rangle- \Big\langle\frac{\tilde p^1_1}{\Delta t} ,\delta q _0  \Big\rangle,
\end{aligned}
\end{equation}  
in which we use the notation
\[
\mathsf{h}_k := H \Big( \tilde q_k, \frac{ \tilde p^1_{k}+ \tilde p^2_{k}}{2} \Big)  + \sum^N_{i=1} H_i \Big( \tilde q_k, \frac{ \tilde p^1_{k}+ \tilde p^2_{k}}{2} \Big) \frac{ \Delta W_{k}^i}{ \Delta t}.
\]
$\delta \mathsf{h}_k  / \delta q$ is the partial derivative with respect to $\tilde q_k$ and $\delta \mathsf{h}_k  / \delta p$  is the partial derivative with respect to $(\tilde p^1_{k}+ \tilde p^2_{k})/2$.

The critical condition \eqref{psp2} thus yields the following equations:
\begin{equation}\label{Midpoint}
\left\{
\begin{array}{ll} 
\displaystyle\vspace{0.2cm}
\frac{\tilde q_k-q_{k-1}}{\Delta t} = \frac{1}{2} \frac{\pa \mathsf{h}_k}{\pa p}  \, , \qquad &\text{for $k = 1,..,K$}\\
\displaystyle\vspace{0.2cm}\frac{q_k - \tilde q_k}{\Delta t} = \frac{1}{2} \frac{\pa \mathsf{h}_k}{\pa p } \, , \qquad &\text{for $k = 1,..,K$}\\ 
\displaystyle \vspace{0.2cm}\frac{\tilde p^2_k-\tilde p^1_k}{\Delta t} = - \frac{\pa \mathsf{h}_k}{\pa q}\, , \qquad &\text{for $k = 1,..,K$} \\
\displaystyle \tilde p^1_{k+1}=\tilde p^2_k , \qquad &\text{for $k = 1,..,K-1$} .
\end{array}
\right.   
\end{equation}

Subtracting the first two equations of \eqref{Midpoint}, we get $\tilde q_k= \frac{1}{2}(q_k+q_{k-1})$, which we use to substitute $\tilde q_k$. Adding up the first two equations and using the third one, we have
\begin{equation}\label{pre_Stoch_Midpoint_VS}
\begin{aligned}
&\frac{q_k-q_{k-1}}{\Delta t} = \frac{\partial H}{\partial p} \Big( \frac{q_k+q_{k-1}}{2}, \frac{ \tilde p^1_{k}+ \tilde p^2_{k}}{2} \Big) + \sum^N_{i=1} \frac{\partial H_i}{\partial p} \Big( \frac{q_k+q_{k-1}}{2}, \frac{ \tilde p^1_{k}+ \tilde p^2_{k}}{2} \Big),\\
&\frac{\tilde p_k^2-\tilde p_k^1}{\Delta t} = - \frac{\partial H}{\partial q} \Big( \frac{q_k+q_{k-1}}{2}, \frac{ \tilde p^1_{k}+ \tilde p^2_{k}}{2} \Big) - \sum^N_{i=1} \frac{\partial H_i}{\partial q} \Big( \frac{q_k+q_{k-1}}{2}, \frac{ \tilde p^1_{k}+ \tilde p^2_{k}}{2} \Big),
\end{aligned}
\end{equation}
for $k=1,...,K-1$.

The last equation of \eqref{Midpoint} gives the identification of $\tilde p^1_{k+1}$ with $\tilde p^2_k$ from the previous time step. Furthermore, the definition of $p_0$ and $p_K$ at the end points \eqref{GCn2} can be calculated from \eqref{dS_d_proof} as
\begin{equation}\label{partial_S_d}
p_0 = - \frac{\partial \mathcal{S}_d (q_0,q_K)}{\partial q_0 } = \tilde p^1_1,  \quad p_K = \frac{\partial \mathcal{S}_d (q_0,q_K)}{\partial q_K } = \tilde p^2_K.
\end{equation}
These together prompt the general definition $p_k = \tilde p^1_{k+1} = \tilde p^2_k$ for $k=1,...,K-1$. Taking this definition, and the definitions of $p_0$ and $p_K$ into the equations \eqref{pre_Stoch_Midpoint_VS}, we derive the midpoint method \eqref{Stoch_Midpoint_VS}.
\end{proof}

\begin{remark}\rm 
In the definition of the discrete functional $\mathcal{G}_d(c_d)$ in \eqref{discrete_PS2}, the covectors $p$ are valued only at mid-steps. The generating property of the functional $\mathcal{S}_d(q_0,q_K)$ gives the definition of $p_0$ and $p_N$ at end points. Then, the definition of $p_k$ for $k=1 ,..., K-1$ is motivated by the equality $\tilde p^1_{k+1} = \tilde p^2_k$, a result of the Hamiltonian's principle \eqref{psp2}, and by the definition of $p$ at end points.

Intrinsically, the definition of $p$ at each integer time step can be given by the generating property of the functional $\mathcal{S}_d(q_k,q_{k+1})$ restricted to each time interval in the following manner:
\[
p_k =  - \frac{\partial \mathcal{S}_d (q_k,q_{k+1})}{\partial q_k}, \; \; p_{k+1} = \frac{\partial \mathcal{S}_d (q_k,q_{k+1})}{\partial q_{k+1} }. 
\]
This definition coincides with the one we used in the proof.
\end{remark}

\begin{remark} \rm
The midpoint integrator \eqref{Stoch_Midpoint_VS} of the canonical Hamiltonian system on the vector space $T^*V$, for regular enough $H$ and $H_i$, have been shown in various literature (for example \cite{MiReTr2002}) to have the strong mean square convergence of order 0.5, and in the special case of single or multiple commutative stochatic Hamiltonians, to have the strong mean square convergence of order 1. 

In section \ref{sec_convergence}, we will extend this result to the case of Hamiltonian systems on Lie groups.
\end{remark}

\begin{remark}\label{bis}\rm
The same stochastic midpoint integrator presented above \eqref{Stoch_Midpoint_VS} can also be derived by taking the discrete action functional in the following way:
\begin{equation*}
\begin{aligned} 
\mathcal{G} _d(c_d)&=\sum_{k=0}^{K-1} \Delta t \Big[ \Big\langle p_{k+1}, \frac{ q_{k+1}-\tilde q_{k+1}}{ \Delta t} \Big\rangle  + \Big\langle \tilde p_{k+1}, \frac{\tilde q_{k+1} -q_k}{\Delta t} \Big\rangle \\
&\qquad - H \Big(\frac{q_k + \tilde q_{k+1}}{2},\tilde p_{k+1} \Big)  - \sum^N_{i=1} H_i \Big( \frac{q_k + \tilde q_{k+1}}{2},\tilde p_{k+1} \Big) \frac{ \Delta W^i_{k+1}}{ \Delta t}  \Big].
\end{aligned} 
\end{equation*}
In this definition, the covectors $\tilde p_{k+1} $ are attached to the vectors $(q_k + \tilde q_{k+1})/2$. Note that the difference of the discretizations lies in the evaluation of Hamiltonians at each time step. In the previous definition, they are evaluated at $(\tilde q_{k+1}, (\tilde p^1_{k+1}+ \tilde p^2_{k+1})/2 )$, while in the current definition they are evaluated at  $((q_k + \tilde q_{k+1})/2,\tilde p_{k+1} )$. 
In the next section, the generalization of this alternative definition of the discrete action functional to the Lie group case will turn out to be cumbersome, which is the reason why we prefer the previous definition \eqref{discrete_PS2}. 
\end{remark}\vspace{0.2cm}


The midpoint method is known to be symplectic \cite{MiReTr2002}. The symplecticity of this method can also be shown by exploiting the discrete stochastic phase space principle, following the usual way of proving the symplectic property of variational integrators \cite{MaWe2001} and their stochastic version \cite{BROw2009}. Here we state this result briefly in the next proposition. 

\begin{proposition}
  The midpoint integrator given in the proposition above is symplectic: the canonical symplectic form  $\Omega_{T^*V} (q, p) ={\rm d} q \wedge {\rm d} p$ is preserved by the stochastic flow $F_{0,T}$.
\end{proposition}
\begin{proof}
By the discrete Hamilton's principle \eqref{dS0} and from \eqref{dS_d_proof}, the exterior derivative of $\mathcal{S}_d(q_0, q_K)$ has the form:
 \[
{\rm d} \mathcal{S}_d = -\tilde p^1_1 {\rm d} q_0 + \tilde p^2_K {\rm d} q_K = - p_0 {\rm d} q_0 + p_K {\rm d} q_K,
\]
the second equality due to the generating equations \eqref{GCn2}. Taking the exterior derivative for a second time shows the symplecticity of the flow: 
\begin{align*}
0 &= -{\rm d} {\rm d} \mathcal{S}_d(q_0,q_K) = {\rm d} q_K \wedge {\rm d} p_K - {\rm d} q_0 \wedge {\rm d} p_0 =\Omega_{T^*V} (q_K,p_K)- \Omega_{T^*V}(q_0,p_0)\\
&= \big((F_{0,T} )^* \Omega_{T^*V} - \Omega_{T^*V}\big)(q_0,p_0),
\end{align*} 
where we used $F_{0,T}(q_0,p_0)=(q_K,p_K)$. This proves $(F_{0,T} )^* \Omega_{T^*V} = \Omega_{T^*V}$.
\end{proof}

\section{Stochastic variational integrators on Lie groups}\label{sec_3}

In this section we extend the midpoint variational integrator for stochastic Hamiltonian systems defined on vector spaces to those defined on the cotangent bundle $T^*G$ of some Lie group $G$.
We begin by recalling the Hamilton phase space principle for systems on Lie groups, along with its trivialized version.

\subsection{Stochastic Hamilton phase space principle on Lie groups}

Given smooth Hamiltonian functions $H, H_i:T^*G \rightarrow \mathbb{R} $, $i=1,...,N$, the associated stochastic Hamilton phase space principle reads
\begin{equation}\label{SPS_LG1} 
\delta \int_0^T \left\langle p, \circ {\rm d} g \right\rangle - H(g, p) {\rm d} t - \sum^N_{i=1} H_i(g, p) \circ {\rm d} W_i(t)=0,
\end{equation} 
with $g \in G$ fixed at the endpoints.

To take benefit from the vector space structure of the Lie algebra $ \mathfrak{g}$ of the Lie group $G$, it is convenient to use the trivialization $\lambda: T^*G \rightarrow G \times \mathfrak{g}^*$, $\lambda (g,p)= (g, g^{-1}p)$ of the cotangent bundle, here chosen to be associated to left translation. Correspondingly, we can consider the trivialized Hamiltonians $h, h_i: G \times \mathfrak{g} ^*  \rightarrow \mathbb{R} $ associated to $H, H_i$, defined as 
$h= H\circ\lambda^{-1}$, $h_i= H_i\circ\lambda^{-1}$, giving 
$h(g,\mu) = H(g, g \mu)$, $h_i(g,\mu) = H_i(g, g \mu)$.
With this, the trivialized version of the phase space principle \eqref{SPS_LG1} becomes
\begin{equation}\label{SPS_LG} 
\delta \int_0^T \left\langle \mu , \circ g ^{-1} {\rm d} g \right\rangle - h(g,\mu) {\rm d} t - \sum^N_{i=1} h_i(g,\mu)\circ {\rm d} W_i(t)=0,
\end{equation}
with $g \in G$ fixed at the endpoints. 

It is now standard to check that this principle yields the following stochastic Hamilton equations on $G \times \mathfrak{g} ^* $:
\begin{equation}\label{Phasespace_Lie}
\begin{aligned}
   & g ^{-1} {\rm d} g = \frac{\pa h}{\pa \mu }(g,\mu)  {\rm d} t + \sum^N_{i=1} \frac{\pa h_i}{\pa \mu }(g,\mu) \circ {\rm d} W_i(t), \\ & {\rm d} \mu - \operatorname{ad}^*_{ g ^{-1} {\rm d} g} \mu = - g ^{-1} \frac{\pa h}{\pa g}(g,\mu)  {\rm d} t - \sum^N_{i=1} g ^{-1} \frac{\pa h_i}{\pa g}(g,\mu) \circ {\rm d} W _i (t),
\end{aligned}
\end{equation}
where $ \operatorname{ad}^*_ \xi :\mathfrak{g} ^* \rightarrow \mathfrak{g}^*$  denotes the infinitesimal coadjoint operator defined by $ \left\langle \operatorname{ad}^*_ \xi \mu , \eta \right\rangle  = \left\langle \mu , [ \xi , \eta ] \right\rangle $, for all $ \mu \in \mathfrak{g} ^* $ and all $ \xi , \eta \in \mathfrak{g} $.
 The first, resp., second equations in \eqref{Phasespace_Lie} are found from considering the variations $\delta g$, resp., $\delta \mu$, and using the formula $\delta (g^{-1} {\rm d} g) = {\rm d}(g^{-1}\delta g) +[g^{-1}{\rm d} g, g^{-1} \delta g]$.

Evidently, these equations are Hamiltonian with respect to the symplectic form $\Omega_{G\times\mathfrak{g}^*}$ on $G\times \mathfrak{g}^*$ obtained by pushing forward the canonical symplectic form on $T^*G$ with the diffeomorphism $\lambda$. This gives the symplectic form $\Omega_{G\times\mathfrak{g}^*}=\lambda_*\Omega_{T^*G}$, given by
\begin{equation}\label{trivialized_Omega}
\begin{aligned}
&\Omega_{G\times \mathfrak{g}^*}(g, \mu )\big( ( \delta g_1, \delta \mu_1 ),( \delta g_2, \delta \mu_2 ) \big)\\
&= \left\langle \mu ,[ g ^{-1} \delta g_1, g ^{-1} \delta g_2] \right\rangle + \left\langle g ^{-1} \delta g_1, \delta \mu _2 \right\rangle - \left\langle g ^{-1} \delta g _2, \delta \mu _1 \right\rangle ,
\end{aligned}
\end{equation}
see \cite{AbMa1978}.

\subsection{Discrete stochastic phase space principle on Lie groups}

In order to discretize the stochastic phase space principle on Lie groups, a local diffeomorphism $ \tau : \mathfrak{g} \rightarrow G$ with $ \tau (0)=e$, also called a retraction map, is needed, see \cite{IsMKNoZa2000,BRMa2009}. It is used to express the small discrete changes in the Lie group configuration through the corresponding Lie algebra elements. With such a retraction map at hand, the Lie group version of the discrete stochastic action functional on vector spaces given in \eqref{discrete_PS2} is given as follows:
\begin{equation}\label{discrete_PS_LGa}
\begin{aligned} 
\mathscr{G} _d(\mathfrak{c}_d)&= \sum_{k=0}^{K-1} \Delta t \Big[ \Big\langle \tilde p^2_{k+1}, \tilde g_{k+1} \frac{  \tau ^{-1} (\tilde g_{k+1} ^{-1} g_{k+1})}{\Delta t} \Big\rangle  + \Big\langle \tilde p^1_{k+1}, \tilde g_{k+1} \frac{ \tau ^{-1} ( g_k ^{-1} \tilde g_{k+1})}{\Delta t} \Big\rangle \\
& \quad  - H \Big( \tilde g_{k+1}, \frac{ \tilde p^1_{k+1}+ \tilde p^2_{k+1}}{2} \Big)  - \sum^N_{i=1} H_i \Big( \tilde g_{k+1}, \frac{\tilde p^1_{k+1}+ \tilde p^2_{k+1}}{2} \Big) \frac{ \Delta W^i_{k+1}}{ \Delta t}  \Big] ,
\end{aligned}
\end{equation}
where
\[
\mathfrak{c}_d = \big(g_0, (\tilde g_1, \tilde p^1_1), (\tilde g_1, \tilde p^2_1), g_1, ..., (\tilde g_K, \tilde p^1_K), (\tilde g_K, \tilde p^2_K), g_K\big)
\]
is the discrete curve in the space
\[
\mathscr{C}_d(K) = G \times ((T^*G \oplus T^*G)\times G )^{K}.
\]
Note that, following the approach of Lie group integrators, \cite{IsMKNoZa2000,BRMa2009}, the approximation of the time derivative $\dot{g}$, which in the vector space case is expressed as $(\tilde q_{k+1} -q_k)/\Delta t$, is given by $\tilde g_{k+1} \tau ^{-1} (\tilde g_{k+1} ^{-1} g_{k+1})/\Delta t$ in the Lie group setting \eqref{discrete_PS_LGa}.
The stochastic phase space principle \eqref{SPS_LG1} on Lie groups thus takes the following discretized form:
\[
 {\rm d}   \mathscr{G} _d(\mathfrak{c}_d) \cdot \delta \mathfrak{c}_d = 0.
\]
with $\delta g_0 = \delta g_K = 0$.

In a similar way with the continuous case in \eqref{SPS_LG}, the discrete stochastic action \eqref{discrete_PS_LGa} can be equivalently expressed as a function on the space of trivialized discrete curves as follows:
\begin{equation}\label{discrete_PS_LG}
\begin{aligned} 
\mathcal{G} _d(c_d)&= \sum_{k=0}^{K-1} \Delta t \Big[ \Big\langle \tilde  \mu^2 _{k+1}, \frac{  \tau ^{-1} (\tilde g_{k+1} ^{-1} g_{k+1})}{\Delta t} \Big\rangle  + \Big\langle \tilde \mu^1 _{k+1}, \frac{ \tau ^{-1} ( g_k ^{-1} \tilde g_{k+1})}{\Delta t} \Big\rangle \\
& \quad  - h \Big( \tilde g_{k+1}, \frac{ \tilde\mu^1 _{k+1}+ \tilde \mu^2 _{k+1}}{2} \Big)  - \sum^N_{i=1}h_i \Big( \tilde g_{k+1}, \frac{\tilde\mu^1 _{k+1}+ \tilde \mu^2 _{k+1}}{2} \Big) \frac{ \Delta W^i_{k+1}}{ \Delta t}  \Big] ,
\end{aligned}
\end{equation}
in which 
\begin{equation}\label{discete_curve_LG} 
c_d=\big(g_0, \tilde g_1, \tilde \mu^1_1, \tilde \mu^2_1, g_1,...,g_{N-1},  \tilde g_K, \tilde \mu^1_K, \tilde \mu^2_K, g_K\big) 
\end{equation} 
is a discrete curve in the space 
\begin{equation}\label{discete_curve_space_LG} 
C_d(K) = G \times (G \times ( \mathfrak{g} ^* \oplus  \mathfrak{g} ^* ) \times G)^{K}.
\end{equation}
The left-trivialized momenta $\tilde  \mu^1 _k , \tilde \mu^2 _k \in \mathfrak{g} ^* $ are related to the momenta  $(\tilde g_k, \tilde p^1_k), (\tilde g_k,\tilde p^2_k) \in T^*G$ which appear in the definition \eqref{discrete_PS_LGa}, via the relations $\tilde \mu^1 _k=\tilde g_k ^{-1} \tilde p^1_k$ and $\tilde \mu^2 _k=\tilde g_k ^{-1} \tilde p^2_k$.

The discrete version of the stochastic phase space principle \eqref{SPS_LG} thus becomes:
\begin{equation}\label{discrete_PS_tri} 
{\rm d} \mathcal{G} _d(c_d) \cdot \delta c_d=0, 
\end{equation} 
with $g$ fixed at the endpoints : $\delta g_0 =0$, $\delta g_K =0$.

Below we will use the right trivialized derivative of $  \tau ^{-1}: G \rightarrow \mathfrak{g}$ defined by $ {\rm d} _{\xi }\tau ^{-1} ( \eta  )= {\rm D} \tau ^{-1} ( g) \cdot ( \eta  g) \in \mathfrak{g} $, for $ \xi , \eta \in \mathfrak{g} $, $g =\tau ( \xi )\in G$, and its dual map $[ {\rm d} _ \xi \tau ^{-1}] ^* : \mathfrak{g} ^* \rightarrow \mathfrak{g} ^* $. See \cite{BRMa2009} and \cite{IsMKNoZa2000} for its
definitions and properties with full details.

\begin{proposition}  \label{prop_discrete_PS}
The discrete variational principle \eqref{discrete_PS_tri} yields the following stochastic midpoint Lie group method
\begin{equation}\label{Stoch_Midpoint_LG}
\left\{ 
\begin{array}{l}
\displaystyle\vspace{0.2cm}\frac{1}{\Delta t} \left( \operatorname{Ad}^*_{\tau (\Delta t \tilde \xi _k)} \big[ {\rm d} _{ \Delta t \tilde \xi _k} \tau ^{-1}  \big] ^* \tilde \mu^1 _k  -  \big[ {\rm d} _{ \Delta t\xi _k} \tau ^{-1}  \big] ^* \tilde \mu^2 _k \right) \\
\displaystyle \qquad \qquad = \tilde g _k ^{-1} \frac{\pa h}{\pa g}  \Big( \tilde g_{k}, \frac{ \tilde \mu^1 _{k}+ \tilde \mu^2 _{k}}{2} \Big) + \sum^N_{i=1} \tilde g _k ^{-1} \frac{\pa h_i}{\pa g}\Big( \tilde g_{k}, \frac{ \tilde \mu^1 _{k}+ \tilde \mu^2 _{k}}{2} \Big)\frac{ \Delta W^i_k}{ \Delta t} \\
\displaystyle\vspace{0.4cm}\frac{1}{\Delta t} \left( \operatorname{Ad}^*_{\tau (\Delta t \xi _k)} \big[ {\rm d} _{ \Delta t \xi _k} \tau ^{-1}  \big] ^*  \tilde \mu^2 _k  -  \big[ {\rm d} _{ \Delta t\tilde \xi _{k+1}} \tau ^{-1}  \big] ^*  \tilde\mu^1 _{k+1} \right) = 0\\
\displaystyle\vspace{0.2cm}\tilde \xi _k= \frac{1}{2} \frac{\pa h}{\pa \mu }  \Big( \tilde g_{k}, \frac{  \tilde \mu^1 _{k}+ \tilde \mu^2 _{k}}{2} \Big) + \frac{1}{2} \sum^N_{i=1}   \frac{\pa h_i}{\pa \mu }\Big( \tilde g_{k}, \frac{ \tilde \mu^1 _{k}+ \tilde \mu^2 _{k}}{2} \Big)\frac{ \Delta W^i_k}{ \Delta t} \\
\displaystyle \xi _k =\frac{1}{2} \frac{\pa h}{\pa \mu }  \Big( \tilde g_{k}, \frac{  \tilde \mu^1 _{k}+ \tilde \mu^2 _{k}}{2} \Big) + \frac{1}{2} \sum^N_{i=1}   \frac{\pa h_i}{\pa \mu }\Big( \tilde g_{k}, \frac{ \tilde \mu^1 _{k}+ \tilde \mu^2 _{k}}{2} \Big)\frac{ \Delta W^i_k}{ \Delta t} ,
\end{array}
\right.
\end{equation}
where $ \Delta t \xi _k= \tau ^{-1} ( \tilde g _k ^{-1} g _k )$, $ \Delta t \tilde \xi _k = \tau ^{-1} ( g_{k-1} ^{-1} \tilde g _k )$, and $ \operatorname{Ad}^*_g: \mathfrak{g} ^* \rightarrow \mathfrak{g} ^* $ is the coadjoint action. The indices are $k = 1, .. ,K-1$ for the second equation and $k = 1, .. , K$ for the rest.
\end{proposition}
\begin{proof} 
The key part of the calculation lies in expressing the variation $ \delta \frac{  \tau ^{-1} (\tilde g_{k} ^{-1} g_{k})}{\Delta t}$ in terms of $\delta \tilde g_{k}$ and $\delta g_{k}$. By definition of the right trivilized derivative $ {\rm d} _ {(\cdot)}\tau ^{-1}$, we have that 
\[
\delta \tau ^{-1} (\tilde g_{k} ^{-1} g_{k}) = {\rm D} \tau ^{-1} (\tilde g_{k} ^{-1} g_{k}) \cdot \delta (\tilde g_{k} ^{-1} g_{k}) = {\rm d} _ {\tau ^{-1} (\tilde g_{k} ^{-1} g_{k})}\tau ^{-1} \big(  \delta (\tilde g_{k} ^{-1} g_{k})  \cdot g_k^{-1} \tilde g_{k} \big),
\]
and that
\[
\delta (\tilde g_{k} ^{-1} g_{k})  \cdot g_k^{-1} \tilde g_{k}  = -\tilde g_{k} ^{-1} \delta \tilde g_{k} + \operatorname{Ad}^*_{\tilde g_k ^{-1} g _k } g_k^{-1} \delta g_k.
\]
Because $ {\rm d} _ {(\cdot)}\tau ^{-1} (\cdot)$ is linear in the second argument (see \cite{IsMKNoZa2000}), we have
\[
\delta \tau ^{-1} (\tilde g_{k} ^{-1} g_{k}) = 
{\rm d} _ {\tau ^{-1} (\tilde g_{k} ^{-1} g_{k})}\tau ^{-1} \left( -\tilde g_{k} ^{-1} \delta \tilde g_{k} \right) + {\rm d} _ {\tau ^{-1} (\tilde g_{k} ^{-1} g_{k})}\tau ^{-1} \left( \operatorname{Ad}^*_{\tilde g_k ^{-1} g _k } g_k^{-1} \delta g_k \right).
\]
The calculation of $ \delta \tau ^{-1} ( g_k ^{-1} \tilde g_{k+1}) $ is similar. 

We can then compute the derivative of the entire action as follows:
\begin{equation}\label{dS_d_LG} 
\begin{aligned}
\frac{1}{\Delta t}&{\rm d} \mathcal{G} _d( c_d) \cdot \delta c_d =\\
&\sum_{k=1}^{K}\Bigg[\Big\langle \frac{1}{\Delta t} \operatorname{Ad}^*_{g_{k-1} ^{-1} \tilde g_k} \big[ {\rm d} _{ \tau ^{-1} (g_{k-1} ^{-1} \tilde g_k)} \tau ^{-1} \big] ^* \tilde\mu^1 _k -    \frac{1}{\Delta t}\big[ {\rm d} _{ \tau ^{-1} (\tilde g_k ^{-1} g_k)} \tau ^{-1} \big] ^* \tilde\mu^2 _k - \tilde g_k ^{-1} \frac{\pa \mathsf{h}_k}{\pa g}, \tilde g _k ^{-1} \delta \tilde g _k \Big\rangle \\
& \qquad + \Big\langle \frac{1}{\Delta t} \tau ^{-1} (g_{k-1} ^{-1} \tilde g_k)- \frac{1}{2} \frac{\pa \mathsf{h}_k}{\pa \mu }, \delta \tilde\mu^1 _k  \Big\rangle + \Big\langle \frac{1}{\Delta t} \tau ^{-1} (\tilde g_k ^{-1} g_k)- \frac{1}{2} \frac{\pa \mathsf{h}_k}{\pa \mu }, \delta \tilde\mu^2 _k  \Big\rangle \Bigg]\\
& +\sum_{k=1}^{K-1}\Big\langle \frac{1}{\Delta t} \operatorname{Ad}^*_{\tilde g_k ^{-1} g _k } \big[ {\rm d} _{ \tau ^{-1} ( \tilde g _k ^{-1} g _k )} \tau ^{-1} \big] ^* \tilde\mu^2 _k - \frac{1}{\Delta t}  \big[ {\rm d} _{ \tau ^{-1} (  g _k ^{-1} \tilde g _{k+1} )} \tau ^{-1} \big] ^* \tilde\mu^1 _{k+1}, g _k ^{-1} \delta g _k  \Big\rangle \\
&  +\Big\langle \frac{1}{\Delta t} \operatorname{Ad}^*_{\tilde g_K ^{-1} g _K } \big[ {\rm d} _{ \tau ^{-1} ( \tilde g _K ^{-1} g _K )} \tau ^{-1} \big] ^* \tilde \mu^2 _K , g_K ^{-1} \delta g_K \Big\rangle- \Big\langle\frac{1}{\Delta t}  \big[ {\rm d} _{ \tau ^{-1} (  g _0 ^{-1} \tilde g _1 )} \tau ^{-1} \big] ^* \tilde\mu ^1_1, g _0 ^{-1} \delta g _0  \Big\rangle,
\end{aligned}
\end{equation}  
where we have defined
\[
\mathsf{h}_k:= h \Big( \tilde g_{k}, \frac{ \tilde \mu  _{k}+ \mu  _{k}}{2} \Big) + \sum^N_{i=1} h_i\Big( \tilde g_{k}, \frac{ \tilde \mu  _{k}+ \mu  _{k}}{2} \Big)\frac{ \Delta W^i_k}{ \Delta t}.
\]

By collecting the terms coupled with the independent variations and defining $ \Delta t \xi _k= \tau ^{-1} ( \tilde g _k ^{-1} g _k )$, $ \Delta t \tilde \xi _k = \tau ^{-1} ( g_{k-1} ^{-1} \tilde g _k )$, we get the stochastic midpoint method stated above \eqref{Stoch_Midpoint_LG}. Note that $\delta g_0 = 0$ and $\delta g_K = 0$, an assumption of the variational principle, so the last two terms in the equation \eqref{dS_d_LG} vanish.
\end{proof}


\begin{remark}\rm
It is worthwhile to mention an additional identity of the right trivialized derivative of the retraction map $\tau$ which holds if it satisfies  $\tau(-\xi)=\tau(\xi)^{-1}$:
\begin{equation}\label{taumap_basechange}
\operatorname{Ad}^*_g [{\rm d} _{ \tau ^{-1} (g)} \tau   ^{-1} ] ^* \mu =  [{\rm d} _{- \tau ^{-1} (g)} \tau   ^{-1} ] ^* \mu 
\end{equation}
for all $g \in G$ and $ \mu \in \mathfrak{g} ^* $, see \cite{BRMa2009}.
The first two equations in \eqref{Stoch_Midpoint_LG} can thus be rewritten as
\begin{align*} 
&\frac{1}{\Delta t} \left(  \big[ {\rm d} _{ -\Delta t \tilde \xi _k} \tau ^{-1}  \big] ^* \tilde \mu _k  -  \big[ {\rm d} _{ \Delta t \xi _k} \tau ^{-1}  \big] ^*  \mu _k \right) = \tilde g _k ^{-1} \frac{\pa \mathsf{h}_k}{\pa g} \\
&\frac{1}{\Delta t} \left(  \big[ {\rm d} _{- \Delta t \xi _k} \tau ^{-1}  \big] ^*  \mu _k  -  \big[ {\rm d} _{ \Delta t\tilde \xi _{k+1}} \tau ^{-1}  \big] ^*  \tilde\mu _{k+1} \right) = 0.
\end{align*} 
This formula may become preferable in some applications. Below we keep the original formula with the adjoint operator so as to illustrate some important properties of the midpoint method in the case of symmetry.
\end{remark}

\begin{remark}\rm
We can also generalize the alternative definition of the discrete action functional proposed in Remark \ref{bis} to the Lie group case. However, the formula and the subsequent study of the variational principle are very cumbersome. With the retraction map $\tau$, the alternative discrete action functional has the form
\begin{align*} 
& \sum_{k=0}^{K-1} \Delta t \Big[ \Big\langle \tilde  \mu^2 _{k+1}, \frac{  \tau ^{-1} (\tilde g_{k+1} ^{-1} g_{k+1})}{\Delta t} \Big\rangle  + \Big\langle \tilde \mu^1 _{k+1}, \frac{ \tau ^{-1} ( g_k ^{-1} \tilde g_{k+1})}{\Delta t} \Big\rangle \\
& \quad  - h \Big( g_k \tau(\frac{1}{2}\tau^{-1}(g^{-1}_k \tilde g_{k+1})),\tilde \mu^1_{k+1} \Big)  - \sum^N_{i=1} h_i \Big( g_k \tau(\frac{1}{2}\tau^{-1}(g^{-1}_k \tilde g_{k+1})),\tilde \mu^1_{k+1} \Big) \frac{ \Delta W^i_{k+1}}{ \Delta t}  \Big] .
\end{align*}
The reason why we favour the discretization \eqref{discrete_PS_LG} is that the Lie algebra $\mathfrak{g}$ is a vector space itself. It is thus easy to express the ``average'' of two given Lie algebra elements with the usual arithmetic mean, while we need the retraction map to express the ``average'' of two elements of the Lie group $G$.
\end{remark}

\subsection{Symplecticity}

For a given pair $(g_0,g_K) \in G \times G$, we denote as usual with $\mathcal{S}_d (g_0,g_K)$ the extremum value of the functional $\mathcal{G}_d$ over all the discrete curves with endpoints $(g_0,g_K)$:
\begin{equation}\label{Generating_Function}
     \mathcal{S}_d  (g_0,g_K) =  \underset{g_0, g_K \;\text{given}}{\underset{c_d \in C_d (K)}{\text{ext}}}  \mathcal{G}_d (c_d).
\end{equation}
Equivalently, let $\bar c_d(g_0,g_K)$ be the critical discrete curve that satisfies the discrete phase space principle \eqref{discrete_PS_tri} and that takes the values $g_0$ and $g_K$ at the endpoints. Then 
\[
\mathcal{S}_d  (g_0,g_K) = \mathcal{G}_d (\bar c_d(g_0,g_K)).
\]

In the time interval $[0,T]$, the midpoint method \eqref{Stoch_Midpoint_LG} determines a stochastic flow from $(g_0, \tilde \mu^1_1)$ to $(g_K, \tilde \mu^2_K)$. The momenta $p$ at the endpoints, trivialized or not, do not explicitly appear in the discrete action functional, as in the case of vector space. We can, however, give definition to $p_0 \in T^*_{g_0} G$ and $p_K \in T^*_{g_K} G$ with the generating property of $\mathcal{S}_d$ :

\begin{equation}\label{generating_equations}
   p_0 := - \frac{\partial \mathcal{S}_d (g_0,g_K)}{\partial g_0 } ,  \quad p_K := \frac{\partial \mathcal{S}_d (g_0,g_K)}{\partial g_K }.
\end{equation}
Expressed with the elements of the discrete curve, the definitions are actually: 
\begin{equation}\label{generating_equations2}
   p_0 := g_0 \big[ {\rm d} _{ \tau ^{-1} (  g _0 ^{-1} \tilde g _1 )} \tau ^{-1} \big] ^* \tilde\mu ^1_1  ,  \quad p_K := g_K \operatorname{Ad}^*_{\tilde g_K ^{-1} g _K } \big[ {\rm d} _{ \tau ^{-1} ( \tilde g _K ^{-1} g _K )} \tau ^{-1} \big] ^* \tilde \mu^2 _K,
\end{equation}
as it follows from computing the partial derivatives of $\mathcal{S}_d(g_0,g_K)$ and from using \eqref{dS_d_LG}.
The midpoint method \eqref{Stoch_Midpoint_LG} together with the definitions of $p_0$ and $p_K$ gives rise to the discrete stochastic flow $F_{0,T}: T^*G \rightarrow T^*G $ with 

\begin{equation}\label{discrete_flow}
    F_{0,T} (g_0,p_0) = (g_K,p_K).
\end{equation}
 
A classic result of variational methods, the stochastic flow is symplectic, as the next proposition shows.



\begin{proposition}\label{symplectic} The stochastic flow determined by the equations \eqref{Stoch_Midpoint_LG} and \eqref{generating_equations} is symplectic: it preserves the canonical symplectic form $\Omega_{T^*G} (g,p) = {\rm d} g \wedge {\rm d} p$:
\[
(F_{0,T})^* \Omega_{T^*G} = \Omega_{T^*G}.
\]
\end{proposition}
\begin{proof} 
By the discrete phase space principle \eqref{discrete_PS_tri} as well as  \eqref{dS_d_LG}, \eqref{generating_equations}, and \eqref{generating_equations2}, the exterior derivative of $\mathcal{S}_d(g_0,g_K)$ has the form 
\[
{\rm d} \mathcal{S}_d= - p_0{\rm d} g_0 +g_K{\rm d} g_K.
\]
Taking the exterior derivative for a second time gives 
\begin{align*}
0 &= -{\rm d} {\rm d} \mathcal{S}_d(g_0,g_K) = {\rm d} g_K \wedge {\rm d} p_K - {\rm d} g_0 \wedge {\rm d} g_0 =\Omega_{T^*G} (g_K,p_K)- \Omega_{T^*G}(g_0,p_0)\\
&= \big((F_{0,T} )^* \Omega_{T^*G} - \Omega_{T^*V}\big)(g_0,p_0),
\end{align*} 
where we used $F_{0,T}(g_0,p_0)=(g_K,p_K)$. This proves $(F_{0,T} )^* \Omega_{T^*G} = \Omega_{T^*G}$.
\end{proof}

\medskip

After analysing the terms of the midpoint integrator \eqref{Stoch_Midpoint_LG}, it is convenient to further define the ``shifted" trivialized momenta
\begin{equation}\label{nu_k}
 \tilde \nu^1 _k=[ {\rm d} _{ \tau ^{-1} ( g_{k-1} ^{-1} \tilde g _k )} \tau   ^{-1} ] ^* \tilde\mu^1 _k\quad\text{and}\quad \tilde \nu^2 _k=[ {\rm d} _{ \tau ^{-1} (\tilde g_k ^{-1} g _k )} \tau   ^{-1} ] ^* \tilde\mu^2 _k 
\end{equation}
for $k=1,..,K$. The second equation of \eqref{Stoch_Midpoint_LG} becomes
\[
\tilde \nu^1_{k+1} = \operatorname{Ad}^*_{\tilde g_k ^{-1} g _k } \tilde \nu^2_k,
\]
for $k=1,..,K-1$. This relation, together with the definitions of $p_0$ and $p_K$ as in \eqref{generating_equations2}, motivates us to define the momentum $p_k \in T^*_{g_k}G$ at each time step in the unified manner: 
\begin{equation}\label{def_momentum}
p_k = 
\begin{cases}
g_k  \tilde \nu^1_{k+1} & \text{for all $k$ except } k = K, \\
g_k \operatorname{Ad}^*_{\tilde g_k ^{-1} g _k } \tilde \nu^2_k &\text{for all $k$ except } k = 0,
\end{cases} 
\end{equation}
and the trivialized momentum in the Lie algebra $\mathfrak{g}^*$:
\begin{equation}\label{def_triv_momentum}
   \mu_k = g_k^{-1} p_k.
\end{equation}

Intrinsically, the definition of $p$ at each integer time step can be given by the generating property of the functional $\mathcal{S}_d(g_k,g_{k+1})$ restricted to each time interval in the following manner:
\[
p_k =  - \frac{\partial \mathcal{S}_d (g_k,g_{k+1})}{\partial g_k}, \; \; p_{k+1} = \frac{\partial \mathcal{S}_d (g_k,g_{k+1})}{\partial g_{k+1} }. 
\]
This definition coincides with the definition \eqref{def_momentum}.

The following table summarizes the variables and the relations among them: 
\begin{center}
\begin{tabular}{|c|c|c|}
\hline
 & $T^*_g G$ & $\mathfrak{g}^*$ \\
\hline
 & ... & ...\\
\hline
\multirow{2}{*}{$t_{k-\frac{1}{2}} $} 
&$(\textcolor{red}{\tilde g_k }, \tilde p^1_k)$ &
   $\tilde g_k^{-1} \tilde p^1_k = \textcolor{red}{\tilde \mu^1_k}$\\ 
&$(\textcolor{red}{\tilde g_k } , \tilde p^2_k)$
 & $\tilde g_k^{-1} \tilde p^2_k = \textcolor{red}{\tilde \mu^2_k}$ \\
\hline
$t_k$ & $(\textcolor{red}{g_k}, p_k)$ & $ g_k^{-1} p_k = \mu_k = \begin{cases}
\operatorname{Ad}^*_{\tilde g_k ^{-1} g _k } \tilde \nu^2_k = \operatorname{Ad}^*_{\tilde g_k ^{-1} g _k } [ {\rm d} _{ \tau ^{-1} (\tilde g_k ^{-1} g _k )} \tau   ^{-1} ] ^* \tilde\mu^2 _k \\
\tilde \nu^1_{k+1} = [ {\rm d} _{ \tau ^{-1} ( g_{k} ^{-1} \tilde g _{k+1} )} \tau   ^{-1} ] ^* \tilde\mu^1 _{k+1}
\end{cases} $\\
\hline
\multirow{2}{*}{$t_{k+\frac{1}{2}} $} 
&$(\textcolor{red}{\tilde g_{k+1} }, \tilde p^1_{k+1})$ &
   $\tilde g_{k+1}^{-1} \tilde p^1_{k+1} = \textcolor{red}{\tilde \mu^1_{k+1}}$\\ 
&$(\textcolor{red}{\tilde g_{k+1} } , \tilde p^2_{k+1})$
 & $\tilde g_{k+1}^{-1} \tilde p^2_{k+1} = \textcolor{red}{\tilde \mu^2_{k+1}}$ \\
\hline
 & ... & ... \\
\hline
\end{tabular}
\end{center}
The variables appearing explicitly in the left-trivialized version of the midpoint integrator are shown in red. 

With such definitions, we can also derive the trivialized flow of the midpoint integrator defined on $G \times \mathfrak{g}^*$, 
\begin{equation}\label{flow_triv}
    \mathcal{F}_{0,T} (g_0,\mu_0) = (g_K,\mu_K),
\end{equation}
naturally defined as 
\[
\mathcal{F} _{0,T}= \lambda \circ F_{0,T} \circ \lambda ^{-1},
\]
where $F_{0,T}$ is the stochastic flow of the midpoint integrator defined on $T^*G$ and $\lambda: T^*G \rightarrow G \times \mathfrak{g}^*$ the left trivialization diffeomorphism defined earlier. 

It is clear that, since the flow $F_{0,T}: T^*G \rightarrow T^*G$ is symplectic with respect to the canonical symplectic form $ \Omega _{T^*G}$, the flow $ \mathcal{F} _{0,T}: G \times \mathfrak{g} ^* \rightarrow G \times \mathfrak{g} ^* $ is also symplectic with respect to the trivialized canonical symplectic form $ \Omega _{G \times \mathfrak{g} ^* }$ on $G \times \mathfrak{g} ^*  $ given in \eqref{trivialized_Omega}.
The diagram below provides an illustration of the relationship between $F_{0,T}$ and $\mathcal{F}_{0,T}$:
\[
\begin{xy}
\xymatrix{
(g_0,p_0)\in T^*G\ar[dd]_{\lambda}\ar[rr]^{F_{0,T}}& & (g_K,p_K)\in T^*G \ar[dd]_{\lambda}& \Omega_{T^*G} \\
& & & \\
(g_0,\mu_0)\in G \times\mathfrak{g}^*\ar[rr]^{\mathcal{F}_{0,T}}& & (g_K,\mu_K) \in G \times\mathfrak{g}^*& \Omega_{G \times\mathfrak{g}^*}
}
\end{xy}
\]
The trivialized formulation is very useful when the Hamiltonian system displays symmetry and the trivialized Hamiltonians become reduced Hamiltonians. 

\section{Symmetries, discrete momentum maps, coadjoint orbits, and Casimirs}\label{sec_4}


In this section, we assume that the stochastic Hamiltonian system on $T^*G$ is invariant under some subgroup symmetries and show that our scheme satisfies a discrete version of Noether's theorem. When the symmetry is given by the Lie group $G$ itself, we prove that the coadjoint orbits are preserved, that the discrete flow preserves the Lie-Poisson structure on $\mathfrak{g}^*$, and that the flow is symplectic on each coadjoint orbit with respect to the Kirillov-Kostant-Souriau symplectic form, just as the continuous flow.

\subsection{Momentum maps and discrete Noether's theorem}

Consider a subgroup $K \subset G$ acting on $G$ by multiplication on the left. Assume that the Hamiltonians $H, H_i: T^*G \rightarrow \mathbb{R} $ are invariant under the cotangent lift of the left translation: $H(kg,kp) = H(g,p)$ for any $k \in K$. As a consequence the trivialized Hamiltonians $h, h_i: G \times \mathfrak{g} ^* \rightarrow \mathbb{R}$ satisfy $h(kg, \mu )= h(g, \mu )$ and $h_i( kg, \mu )= h_i(g, \mu )$ for all $k \in K$. In this case, we also say that the trivialized Hamiltonians are left $K$-invariant. In particular, our scheme preserves the Casimir functions of the Lie-Poisson bracket.

\begin{proposition}[Momentum map and discrete Noether's theorem]\label{Noether1}
Let $K \subset G$ be a subgroup, and $ \mathfrak{k}$ its Lie algebra. Assume that the Hamiltonians $h$ and $h_i$ are left $K$-invariant, i.e. $h(kg, \mu )= h(g, \mu )$ and $h_i( kg, \mu )= h_i(g, \mu )$ for all $k \in K$, $g \in G$, and $ \mu \in \mathfrak{g} ^* $. Then the momentum map $J : T^*G \rightarrow \mathfrak{k} ^*$ in regard to the left action of $K$ on $G$, is 
\begin{equation}\label{discrete_momap} 
J(g, p) =  i_{ \mathfrak{k}} ^*\big( p g^{-1} \big),
\end{equation} 
where $i_{ \mathfrak{k}} ^* : \mathfrak{g} ^* \hookrightarrow \mathfrak{k} ^* $ is the dual of the Lie algebra inclusion $ i_{ \mathfrak{k}} : \mathfrak{k} \hookrightarrow \mathfrak{g}$. The momentum is preserved by the discrete stochastic flow \eqref{discrete_flow} generated with the midpoint method. 
\end{proposition}
\begin{proof} Recall that if a Lie group $K$ acts on a manifold $Q$, then there is a naturally associated action of $K$ on the cotangent bundle $T^*Q$, called the cotangent lifted action, which is symplectic relative to $\Omega_{T^*Q}$. The momentum map associated to this action is $J:T^*Q\rightarrow\mathfrak{k}^*$, given by $\langle J(q,p), \zeta\rangle = \langle p, \zeta_Q(q)\rangle$, for all $\zeta\in\mathfrak{k}$, see \cite{MaRa1999}. In this formula the vector field $\zeta_Q$ is the infinitesimal generator of the Lie group action of $K$ on $Q$. The duality pairings are, respectively, between the Lie algebra and its dual, and between the tangent space and the cotangent space at $q$.
In our case, with the action of $K$ given by left translation on $G$, this formula gives $J:T^*G \rightarrow\mathfrak{k}^*$ given by: 
\[
\langle J(g,p) , \zeta \rangle = \langle p , \zeta g \rangle = \langle p g^{-1}, \zeta \rangle,
\]
for any $\zeta \in \mathfrak{k}$. Thus $J(p,g) = i_{ \mathfrak{k}} ^*\big( p g^{-1} \big)\in \mathfrak{k}^*$.

Since $h$ and $h_i$ are left $K$-invariant, the discrete stochastic action functional $ \mathcal{G} _d$ in \eqref{discrete_PS_LG} is also left $K$-invariant: $ \mathcal{G} _d( k c_d)= \mathcal{G}_d(c_d)$, where the left action of $K$ on the discrete curve of random variables $c_d=(g_0 , \tilde g_1, \tilde \mu^1 _1, \tilde \mu^2 _1, g_1, ... g_K) $ is $k c_d=(k g_0 , k \tilde g_1, \tilde \mu^1 _1, \tilde \mu^2 _1, k g_1, ... k g_K)$. Let $\bar c_d$ be a critical curve of the phase space principle \eqref{discrete_PS_tri}, corresponding to the discrete stochastic flow. Computing the derivative of the map $\varepsilon \rightarrow \mathcal{S} _d( k_{\varepsilon} \bar c_d)$ at the identity in some direction $ \zeta \in \mathfrak{k} $ , we get, from \eqref{dS_d_LG} :
\begin{equation}\label{dS_d_LG_b}
\begin{aligned}
& \left. \frac{{\rm d} \mathcal{S} _d( k_{\varepsilon} \bar c_d)}{{\rm d}  \varepsilon} \right|_{0}= {\rm d} \mathcal{S} _d(\bar c_d) \cdot \zeta  \bar c_d  \\
&= \Big\langle  \operatorname{Ad}^*_{\tilde g_K ^{-1} g _K } \big[ {\rm d} _{ \tau ^{-1} ( \tilde g _K ^{-1} g _K )} \tau ^{-1} \big] ^* \tilde \mu^2 _K, g_K ^{-1} \zeta g_K \Big\rangle- \Big\langle \big[ {\rm d} _{ \tau ^{-1} (  g _0 ^{-1} \tilde g _1 )} \tau ^{-1} \big] ^* \tilde\mu^1 _1, g _0 ^{-1} \zeta  g _0  \Big\rangle\\
&= \Big\langle  p_K g_K^{-1} - p_0 g_0^{-1} , \zeta  \Big\rangle.
\end{aligned}
\end{equation}
From the $K$-invariance of the action we have $ \left. \frac{\pa \mathcal{S} _d( k_{\varepsilon} \bar c_d)}{\pa \varepsilon} \right|_{0}=0$, for all $ \zeta \in \mathfrak{k} $, hence \eqref{dS_d_LG_b} shows that $i_{ \mathfrak{k}} ^*\big( p_0 g_0^{-1} \big) = i_{ \mathfrak{k}} ^*\big( p_K g_K^{-1} \big)$. As a consequence, the momentum map is preserved by the stochastic flow.
\end{proof}

\begin{remark}[Direct proof of Noether's Theorem]\label{remark1}\rm 
The discrete Noether theorem can also be directly proved from the midpoint integrator formula \eqref{Stoch_Midpoint_LG}. 
In fact, the first equation of \eqref{Stoch_Midpoint_LG}, expressed with the momenta $p_{k-1}$ and $p_k$, is
\[
p_{k-1}g_{k-1}^{-1} - p_k g_k^{-1} = \Delta t \frac{\pa \mathsf{h}}{\pa g}  \Big( \tilde g_k, \frac{ \tilde \mu^1 _k+ \tilde \mu^2 _k}{2} \Big) \tilde g _k ^{-1}.
\]
Since the Hamiltonians $h$, $h_i$ are $K$-invariant, this is in the kernel $\text{ker} \, i_{ \mathfrak{k}} ^* $:
\[
0 = \left. \frac{{\rm d} h\Big( k
_{\varepsilon} \tilde g_k, \frac{ \tilde \mu^1 _k+ \tilde \mu^2 _k}{2} \Big)}{{\rm d} \varepsilon} \right |_0  = \Big\langle  \frac{\pa h}{\pa g}  \Big( \tilde g_k, \frac{ \tilde \mu^1 _k+ \tilde \mu^2 _k}{2} \Big), \zeta \tilde g_k \Big\rangle = \Big\langle  \frac{\pa h}{\pa g}  \Big( \tilde g_k, \frac{ \tilde \mu^1 _k+ \tilde \mu^2 _k}{2} \Big) \tilde g_k^{-1}, \zeta  \Big\rangle,
\]
for any $\zeta \in \mathfrak{k}$. Thus $J (g_{k-1},p_{k-1}) =  i_{ \mathfrak{k}} ^*\big( p_{k-1} g_{k-1}^{-1} \big) =  i_{ \mathfrak{k}} ^*\big( p_{k} g_{k}^{-1} \big) = J (g_{k},p_{k})$. By induction, it is shown that the momentum map is preserved by the stochastic flow.
\end{remark}

\subsection{Preservation of coadjoint orbits, Casimirs, and Lie-Poisson brackets}

In the case where $K=G$, that is, the Hamiltonians are left $G$-invariant, the partial derivatives of the Hamiltonians with respect to the variable $g$ are zero. Furthermore the trivialized Hamiltonians can now be presented as reduced Hamiltonians: $h^{\rm red} (\mu) := h(g, \mu)$, since $h(g, \mu) = H(g , g\mu) = H(e, \mu) = h(e, \mu)$. With an abuse of notation, we will also denote the reduced Hamiltonians as $h(\mu)$, $h_i (\mu)$. The mid-point integrator now has the simplified (reduced) form: 
\begin{equation}\label{midpoint_reduced}
\left\{ 
\begin{array}{l}
\displaystyle\vspace{0.2cm}\operatorname{Ad}^*_{\tau (\Delta t \tilde \xi _k)} \big[ {\rm d} _{ \Delta t \tilde \xi _k} \tau ^{-1}  \big] ^* \tilde \mu^1 _k  -  \big[ {\rm d} _{ \Delta t\xi _k} \tau ^{-1}  \big] ^* \tilde \mu^2 _k =0 \\
\displaystyle\vspace{0.2cm} \operatorname{Ad}^*_{\tau (\Delta t \xi _k)} \big[ {\rm d} _{ \Delta t \xi _k} \tau ^{-1}  \big] ^*  \tilde \mu^2 _k  -  \big[ {\rm d} _{ \Delta t\tilde \xi _{k+1}} \tau ^{-1}  \big] ^*  \tilde\mu^1 _{k+1}  = 0\\
\displaystyle\vspace{0.2cm}\tilde \xi _k= \frac{1}{2} \frac{\pa h}{\pa \mu }  \Big(  \frac{  \tilde \mu^1 _{k}+ \tilde \mu^2 _{k}}{2} \Big) + \frac{1}{2} \sum^N_{i=1}   \frac{\pa h_i}{\pa \mu }\Big( \frac{ \tilde \mu^1 _{k}+ \tilde \mu^2 _{k}}{2} \Big)\frac{ \Delta W^i_k}{ \Delta t} \\
\displaystyle \xi _k =\frac{1}{2} \frac{\pa h}{\pa \mu }  \Big(  \frac{  \tilde \mu^1 _{k}+ \tilde \mu^2 _{k}}{2} \Big) + \frac{1}{2} \sum^N_{i=1}   \frac{\pa h_i}{\pa \mu }\Big(\frac{ \tilde \mu^1 _{k}+ \tilde \mu^2 _{k}}{2} \Big)\frac{ \Delta W^i_k}{ \Delta t}.
\end{array}
\right.
\end{equation}

In absence of $g$, the integrator now determines a reduced stochastic flow
\[
\F^{\rm red}_{0,T} :\mathfrak{g}^*\rightarrow\mathfrak{g}^*, \quad \F^{\rm red}_{0,T} (\mu_0) = \mu_K
\]
with $\mu_0$ and $\mu_K$ determined by the formulas \eqref{def_triv_momentum}. 

\medskip

We recall that in the $G$-invariant case, a deterministic Hamiltonian system with respect to the canonical symplectic form on $T^*G$ can be reduced to give a Lie-Poisson systems on the dual Lie algebra $\mathfrak{g}^*$, the Lie-Poisson bracket on $\mathfrak{g}^*$ being given by
\begin{equation}\label{LP}
\{f,g\}= -\left\langle \mu, \left[\frac{\delta f}{\delta \mu}, \frac{\delta g}{\delta \mu}\right]\right\rangle 
\end{equation}
for smooth functions $f,g:\mathfrak{g}^*\rightarrow\mathbb{R}$. A particularly important property of Lie-Poisson systems, is that they preserve the coadjoint orbits, given as
\[
\mathcal{O}(\mu) = \{ \operatorname{Ad}^*_g \mu \mid g \in G \} \subset \mathfrak{g} ^*.
\]
Moreover, when restricted to coadjoint orbits, the flow of Lie-Poisson system is symplectic with respect to the Kirillov-Kostant-Souriau symplectic form on $\mathcal{O}(\mu)$ given by
\begin{equation}\label{KKS}
\omega_{\mathcal{O}(\mu)}(\nu)(\operatorname{ad}^*_\xi\nu,\operatorname{ad}^*_\eta\nu ) =-\langle \nu,[\xi,\eta]\rangle
\end{equation}
for all $\nu\in \mathcal{O}(\mu)$ and all $\xi,\eta\in\mathfrak{g}$. We shall show that all these properties are still satisfied in the stochastic and discrete case.

We start with the next proposition which shows that the discrete stochastic flow preserves the coadjoint orbits if the Hamiltonians are left $G$-invariant.

\begin{proposition}[Preservation of coadjoint orbits and Casimirs]\label{coadj_orb}
When the Hamiltonians are left $G$-invariant, the stochastic discrete flow $\F^{\rm red}_{0,T}$ preserves the coadjoint orbits in $\mathfrak{g}^*$. That is to say, $\mathcal{O}(\mu_k) = \mathcal{O}(\mu_0)$, with $\mu_k$ is the trivialized momentum defined in \eqref{def_triv_momentum}. In particular, Casimirs are constant along the discrete stochastic flow.
\end{proposition}
\begin{proof} In the $G$-invariant case, the first equation in \eqref{midpoint_reduced} is
\[
\tilde \nu^2_{k}= \operatorname{Ad}^*_{g_{k-1} ^{-1} \tilde g_k } \tilde \nu^1_k ,
\] 
or equivalently,
\[
\operatorname{Ad}^*_{g_{k} ^{-1} \tilde g_k } \mu_k = \operatorname{Ad}^*_{g_{k-1} ^{-1} \tilde g_k } \mu_{k-1}.
\]
Hence, $\mu_k = \operatorname{Ad}^*_{g_{k-1} ^{-1} g_k } \mu_{k-1} \in \mathcal{O}(\mu_{k-1})$. By induction, this shows that the stochastic flow preserves the coadjoint orbits. Since the Casimir functions are constant on coadjoint orbits, their conservation trivially holds.
\end{proof}\vspace{0.5cm}

The reduced flow $\mathcal{F}^{\rm red}_{0,T}$ has richer geometry-preserving properties, as the next proposition shows.

\begin{proposition}[Lie-Poisson integrator]\label{Lie-Poisson}
When the Hamiltonians are $G$-invariant, the stochastic discrete flow $ \mathcal{F}^{\rm red}_{0,T}: \mathfrak{g} ^* \rightarrow \mathfrak{g} ^* $ preserves the Lie-Poisson structure, i.e.
\[
\{f \circ \mathcal{F}^{\rm red}_{0,T}, g \circ \mathcal{F}^{\rm red}_{0,T}\}=\{f,g\} \circ \mathcal{F}^{\rm red}_{0,T},
\]
for all $f,g: \mathfrak{g} ^* \rightarrow \mathbb{R} $. Moreover, the flow is symplectic on each coadjoint orbit, with respect to the Kirillov-Kostant-Souriau symplectic form.
\end{proposition}
\begin{proof} 
Note that in equations \eqref{midpoint_reduced} we always have the relations: $ \Delta t \xi _k= \tau ^{-1} ( \tilde g _k ^{-1} g _k )$, $ \Delta t \tilde \xi _k = \tau ^{-1} ( g_{k-1} ^{-1} \tilde g _k )$. Thus the flow of $g$ can be determined by:
\[
g_K = g_0 \prod^{K}_{k=1} \tau(\Delta t \tilde \xi _k)  \tau(\Delta t \xi _k) .
\]
Since $\xi_k$ and $\tilde \xi_k$ are independent of $g_0$, this formula actually shows that the discrete flow in $T^*G$ \eqref{SPS_LG} determined by the midpoint integrator is also left $G$-invariant: 
\[
F_{0,T} (hg_0,hp_0) = h F_{0,T} (g_0,p_0).
\]
Therefore, the reduced flow $ \mathcal{F} ^{\rm red}_{0,T}: \mathfrak{g} ^* \rightarrow \mathfrak{g} ^*$ can be induced from the left $G$-invariant flow $ F _{0,T}$ via the formula $\pi\circ F_{0,T} = \mathcal{F}^{\rm red} _{0,T}\circ \pi$ with $\pi:T^*G \rightarrow \mathfrak{g}^*$, $\pi(g,p)= g^{-1}p$.

\rem{
\textcolor{blue}{Cyclic definition!? Surely the two ways of introducing $ \mathcal{F} ^{\rm red}_{0,T}$ are equivalent.}
\todo{FGB: What are you doubts here?\\
- First you defined $\mathcal{F}^{\rm red}_{0,T}$ as the flow of \eqref{midpoint_reduced}.\\
- Then, we remember that $F_{0,T}$ is defined as the flow of \eqref{Stoch_Midpoint_LG} (or we can also consider $\mathcal{F}_{0,T}$ if we prefer).\\
- Finally, by comparing these two systems in the $G$-invariant case and using the formula $g_K=g_0\prod^{K}_{k=1} \tau(\Delta t \tilde \xi _k)  \tau(\Delta t \xi _k)$ above we have $\pi_L\circ F_{0,T}= \mathcal{F}_{0,T}\circ \pi_L$, with $\pi_L(g,p)= g^{-1}p$. Is this working? Nothing cyclic?}
}

Then it is a general result in Poisson reduction that a $G$-invariant Poisson map induces a Poisson map on the quotient. Indeed, let $(P, \{\cdot, \cdot\})$ be a Poisson manifold and let $\Phi:G \times P \rightarrow P$ be an action of a Lie group $G$ which is Poisson, i.e. $\{F\circ\Phi_g, H\circ \Phi_g\}=\{F,H\}\circ \Phi_g$, for all $g\in G$ and for all functions $F,H$. Assume also that the action is such that $\pi:P \rightarrow P/G$ is a principal bundle. Then by Poisson reduction,  \cite{MaRa1999}, there is a unique Poisson structure $\{\cdot , \cdot\}_{\rm red}$ on $P/G$ such that $\{f\circ \pi, h\circ \pi\}= \{f,h\}_{\rm red}\circ \pi$. Consider a $G$-invariant map $\Psi: P \rightarrow P$ which is Poisson: $\{F\circ \Psi, H\circ \Psi\}=\{F,H\}\circ \Psi$, for each function $F,H$. Then the induced map $\psi:P/G\rightarrow P/G$ on the quotient, i.e., $\psi\circ \pi= \pi\circ \Psi$ is Poisson with respect to $\{\cdot, \cdot\}_{\rm red}$. Indeed, we have
\begin{align*}
\{f\circ \psi, h \circ\psi\}_{\rm red}\circ \pi&=\{f\circ \psi \circ \pi, h \circ\psi \circ \pi\}\\
&=\{f\circ \pi\circ \Psi, h\circ \pi\circ \Psi\}\\
&=\{f\circ \pi, h\circ \pi\}\circ \Psi\\
&=\{f, h\}_{\rm red} \circ \pi \circ \Psi\\
&=\{f, h\}_{\rm red} \circ \psi \circ \pi ,
\end{align*}
proving that $\psi$ is Poisson.

In our case $P=T^*G$ with the canonical Poisson bracket, while $P/G=\mathfrak{g}^*$ with the Lie-Poisson bracket, with $\Psi= F_{0,T}$ and $\psi= \mathcal{F}^{\rm red}_{0,T}$. The flow $F_{0,T}$ is symplectic with respect to the canonical symplectic form, hence it is Poisson with respect to the canonical Poisson bracket, so that the above applies.

\color{black}

The fact that $ \mathcal{F} _{0,T}^{\rm red}$ preserves the coadjoint orbits is already proven above. Then, we can use another general fact in Poisson geometry that a Poisson diffeomorphism that preserves the symplectic leaves is automatically symplectic on these leaves. This follows easily by using the following two facts. First, given a Poisson manifold $(P, \{\cdot , \cdot \})$ and a symplectic leaf $\mathcal{L}\subset P$, the symplectic structure on $\mathcal{L}$ at $p\in \mathcal{L}$ is given by $\omega(p)(u_p, v_p)=\{f,g\}(p)$, where $f,g:P\rightarrow \mathbb{R}$ are such that $X_f(p)= u_p$ and $X_g(p)= v_p$. Second, if $\Psi:P\rightarrow P$ is a Poisson map, then $T\Psi \circ X_{F\circ \Psi}= X_F\circ \Psi$, for each function $F: P \rightarrow\mathbb{R}$.
We apply this result to our case with $P=\mathfrak{g}^*$ endowed with the Lie-Poisson structure and $\mathcal{L}= \mathcal{O}$ a coadjoint orbit, with its Kirillov-Kostant-Souriau symplectic form.
\end{proof}

\color{black}

\section{Hamiltonians with an advected parameter and semidirect products}\label{Section_4}

\rem{    
\color{blue} 
\begin{framed} \color{blue} FGB: I checked that all your developments are correct. I now think we should present this a bit differently. Not as an extension of \S3, but as a special case of \S3 (the extension will be the case of $T^*(G \times W)$ that I recall later). In fact we can present the case of an advected quantity as a special case of the above in which $K=G_{ \alpha _0}$. So I moved it in this chapter. In this case, the abstract quotient $(T^*G)/K$ can be identified as $(T^*G)/G_{ \alpha _0} \simeq \mathfrak{g} ^* \times \operatorname{Orb}( \alpha _0)$ via the identification $[(g, p)] \simeq ( \mu = g ^{-1} p,  \alpha =g ^{-1} \cdot \alpha _0)$, with $ \operatorname{Orb}( \alpha _0)$ the $G$-orbit of $ \alpha _0$. This means that the reduced Hamiltonian is $h( \mu , \alpha ): \mathfrak{g} ^* \times  \operatorname{Orb}( \alpha _0) \rightarrow \mathbb{R}$ (instead of $h(g, \mu , \alpha )$), so it simplifies a bit below since $ \frac{\pa h}{\pa g}=0$ so it is easier to understand and this is what happens in applications.

I added some details below to justify the form of the variational principle.

Thanks to this approach, we can find a nice generalization of the coadjoint orbit result when $K=G_{ \alpha _0}$, see later for some starting point.
\end{framed} 

\color{black}
}

In various applications, such as heavy top dynamics and compressible fluids, the Hamiltonian of the system depends on some parameter $ \alpha_0$ in the dual $V^*$ of a vector space $V$ on which the Lie group acts linearly from the left. In addition, the Hamiltonian is invariant under the isotropy subgroup of $ \alpha _0$, \cite{MaRaWe1984}.
We shall make the same assumptions on the stochastic Hamiltonians $H_i$.

\subsection{Stochastic Lie group variational integrators with advected parameter}

Let $V$ be a vector space on which $G$ acts by a left representation, denoted $v \in V\mapsto g \cdot v \in V$, for all $g \in G$. We consider the left representation $ \alpha  \in V^* \mapsto g \cdot \alpha \in V^*$ induced on the dual, defined as $ \left\langle g \cdot \alpha, v \right\rangle = \left\langle \alpha, g ^{-1} \cdot v \right\rangle $.
We then consider the Hamiltonians $H_ { \alpha _0}, H_{ \alpha _0,i}:T^*G \rightarrow \mathbb{R} $ for $ \alpha _0$ fixed, and assume $H_{\alpha _0}(hg, hp)=H_{\alpha _0}(g,p)$ for all $h \in G_{ \alpha _0}$. $G_{ \alpha _0}$ is the isotropy subgroup of $\alpha _0$.

The quotient space $(T^*G)/G_{ \alpha _0}$ can be identified as $ \mathfrak{g} ^* \times \operatorname{Orb}( \alpha _0)$ via the diffeomorphism $[(g, p)] \simeq ( \mu = g ^{-1} p,  \alpha =g ^{-1} \cdot \alpha _0)$. Due to the symmetry, we can define the reduced Hamlitonians $h, h_i: \mathfrak{g} ^* \times \operatorname{Orb}( \alpha _0) \rightarrow \mathbb{R} $ as $H_{ \alpha _0}(g, p)= h( g ^{-1} p, g ^{-1} \cdot \alpha _0)$ and $H_{i, \alpha _0}(g, p)= h_i( g ^{-1} p, g ^{-1} \cdot \alpha _0)$.

The stochastic phase space principle \eqref{SPS_LG1} for $H_{ \alpha _0}$ and $H_{i,\alpha _0}$ induces the following principle in the reduced form: 
\begin{equation}\label{SPS_LGex_1} 
\delta \int_0^T \left\langle \mu , \circ g ^{-1} {\rm d} g \right\rangle - h(\mu,\alpha) {\rm d} t - \sum^N_{i=1} h_i(\mu,\alpha)\circ {\rm d} W_i(t)=0,
\end{equation} 
with $g$ fixed at endpoints and $\alpha_t = g_t^{-1} \cdot \alpha_0$. As a result, the evolution equation of $\alpha_t$ can be expressed as ${\rm d} \alpha_t = -  g_t^{-1} {\rm d} g_t \cdot \alpha_t $.

The reduced variational principle \eqref{SPS_LGex_1}
yields the stochastic Hamiltonian system
\begin{equation}\label{LP_alpha} \begin{aligned}
    g ^{-1} {\rm d} g &= \frac{\pa h}{\pa \mu }  {\rm d} t + \sum^N_{i=1} \frac{\pa h_i}{\pa \mu } \circ {\rm d} W_i(t), \\
   {\rm d} \mu -\operatorname{ad}^*_{ g ^{-1} {\rm d} g} \mu &= -  \frac{\pa h}{\pa \alpha} \diamond \alpha  {\rm d} t - \sum^N_{i=1}  \frac{\pa h_i}{\pa \alpha} \diamond \alpha  \circ {\rm d} W _i (t), 
\end{aligned}
\end{equation} 
where the diamond operator is defined as 
\[
\left\langle v \diamond \alpha , \xi  \right\rangle_{\mathfrak{g}} =\left\langle - v, \xi  \cdot \alpha \right\rangle= \left\langle \xi \cdot v,  \alpha \right\rangle 
\]
for $ \xi  \in \mathfrak{g}$, $\alpha \in V^*$ and $v  \in V$. The derivation of \eqref{LP_alpha} from \eqref{SPS_LGex_1} is analogous to that of \eqref{Phasespace_Lie}. One uses the formula $\delta \alpha = - ( g^{-1}\delta g) \cdot \alpha$ for $\alpha= g^{-1}\cdot \alpha_0$, in addition of $\delta (g^{-1} {\rm d} g) = {\rm d}(g^{-1}\delta g) +[g^{-1}{\rm d} g, g^{-1} \delta g]$.

In the discrete sense, as an analogue to \eqref{discrete_PS_LGa}, the discrete functional can be defined as
\begin{equation}\label{discrete_PS_LGa_advected}
\begin{aligned} 
\mathscr{G} _d(\mathfrak{c}_d)&= \sum_{k=0}^{K-1} \Delta t \Big[ \Big\langle \tilde p^2_{k+1}, \tilde g_{k+1} \frac{  \tau ^{-1} (\tilde g_{k+1} ^{-1} g_{k+1})}{\Delta t} \Big\rangle  + \Big\langle \tilde p^1_{k+1}, \tilde g_{k+1} \frac{ \tau ^{-1} ( g_k ^{-1} \tilde g_{k+1})}{\Delta t} \Big\rangle \\
& \quad  - H_{\alpha_0} \Big( \tilde g_{k+1}, \frac{ \tilde p^1_{k+1}+ \tilde p^2_{k+1}}{2} \Big)  - \sum^N_{i=1} H_{i,\alpha_0} \Big( \tilde g_{k+1}, \frac{\tilde p^1_{k+1}+ \tilde p^2_{k+1}}{2} \Big) \frac{ \Delta W^i_{k+1}}{ \Delta t}  \Big] ,
\end{aligned}
\end{equation}
where
\[
\mathfrak{c}_d = \big(g_0, (\tilde g_1, \tilde p^1_1), (\tilde g_1, \tilde p^2_1), g_1, ..., (\tilde g_K, \tilde p^1_K), (\tilde g_K, \tilde p^2_K), g_K\big)
\]
is the discrete curve. 

Due to symmetry, it can be equivalently defined with the reduced Hamiltonians:
\begin{equation}\label{discrete_PS_LGa_advected}
\begin{aligned} 
\mathcal{G} _d(c_d)&= \sum_{k=0}^{K-1} \Delta t \Big[ \Big\langle \tilde  \mu^2 _{k+1}, \frac{  \tau ^{-1} (\tilde g_{k+1} ^{-1} g_{k+1})}{\Delta t} \Big\rangle  + \Big\langle \tilde \mu^1 _{k+1}, \frac{ \tau ^{-1} ( g_k ^{-1} \tilde g_{k+1})}{\Delta t} \Big\rangle \\
& \quad  - h \Big( \frac{ \tilde\mu^1 _{k+1}+ \tilde \mu^2 _{k+1}}{2}, \tilde \alpha _{k+1} \Big)  - \sum^N_{i=1}h_i \Big(  \frac{\tilde\mu^1 _{k+1}+ \tilde \mu^2 _{k+1}}{2}, \tilde \alpha _{k+1} \Big) \frac{ \Delta W^i_{k+1}}{ \Delta t}  \Big] ,
\end{aligned}
\end{equation}
where
\[
c_d = \big(g_0, (\tilde g_1, \tilde \mu^1_1), (\tilde g_1, \tilde \mu^2_1), g_1, ..., (\tilde g_K, \tilde \mu^1_K), (\tilde g_K, \tilde \mu^2_K), g_K\big)
\]
is the discrete curve, and 
\[
\tilde \alpha_k = \tilde g_k^{-1} \cdot \alpha_0,
\]
comes from the definition of the reduced Hamiltonians. We can also additionally define 
\[
\alpha_k = g_k^{-1} \cdot \alpha_0.
\]

The discrete version of the phase space principle of a system with an advected parameter now has the form: 
\begin{equation}\label{phase_space_advected}
 \delta  \mathcal{G} _d(c_d) = 0
\end{equation}
with $\delta g_0 = \delta g_K = 0$.
Note that due to the definition $\tilde \alpha_k = \tilde g_k^{-1} \cdot \alpha_0$, the variation of  $\tilde \alpha_k$ can be calculated as $\delta \tilde \alpha_k = -\tilde g_k^{-1} \delta \tilde g_k \cdot \tilde \alpha_k$. We regard the discrete action functional \eqref{discrete_PS_LGa_advected} as being defined on the same discrete curves as earlier in \eqref{discete_curve_LG}.

The principle gives the following stochastic midpoint Lie group method: 
 
\begin{equation}\label{Stoch_Midpoint_LGex}
\left\{ 
\begin{array}{l}
\displaystyle\vspace{0.2cm}\frac{1}{\Delta t} \left( \operatorname{Ad}^*_{\tau (\Delta t \tilde \xi _k)} \big[ {\rm d} _{ \Delta t \tilde \xi _k} \tau ^{-1}  \big] ^* \tilde \mu^1 _k  -  \big[ {\rm d} _{ \Delta t\xi _k} \tau ^{-1}  \big] ^* \tilde \mu^2 _k \right) \\
\displaystyle \qquad\quad =  \frac{\pa h}{\pa \alpha}  \Big(\frac{ \tilde \mu^1 _{k}+ \tilde \mu^2 _{k}}{2} , \tilde \alpha_k \Big) \diamond \tilde \alpha_k + \sum^N_{i=1}  \frac{\pa h_i}{\pa \alpha}  \Big(  \frac{ \tilde \mu^1 _{k}+ \tilde \mu^2 _{k}}{2} , \tilde \alpha_k \Big) \diamond \tilde \alpha_k   \frac{ \Delta W^i_k}{ \Delta t} \\
\displaystyle\vspace{0.4cm}\frac{1}{\Delta t} \left( \operatorname{Ad}^*_{\tau (\Delta t \xi _k)} \big[ {\rm d} _{ \Delta t \xi _k} \tau ^{-1}  \big] ^*  \tilde \mu^2 _k  -  \big[ {\rm d} _{ \Delta t\tilde \xi _{k+1}} \tau ^{-1}  \big] ^*  \tilde\mu^1 _{k+1} \right) = 0\\
\displaystyle\vspace{0.2cm}\tilde \xi _k= \frac{1}{2} \frac{\pa h}{\pa \mu }  \Big( \frac{  \tilde \mu^1 _{k}+ \tilde \mu^2 _{k}}{2}, \tilde \alpha_k \Big) + \frac{1}{2} \sum^N_{i=1}   \frac{\pa h_i}{\pa \mu }\Big( \frac{ \tilde \mu^1 _{k}+ \tilde \mu^2 _{k}}{2}, \tilde \alpha_k \Big)\frac{ \Delta W^i_k}{ \Delta t} \\
\displaystyle \xi _k =\frac{1}{2} \frac{\pa h}{\pa \mu }  \Big(\frac{  \tilde \mu^1 _{k}+ \tilde \mu^2 _{k}}{2}, \tilde \alpha_k \Big) + \frac{1}{2} \sum^N_{i=1}   \frac{\pa h_i}{\pa \mu }\Big( \frac{ \tilde \mu^1 _{k}+ \tilde \mu^2 _{k}}{2} , \tilde \alpha_k\Big)\frac{ \Delta W^i_k}{ \Delta t} ,
\end{array}
\right.
\end{equation}
where $ \Delta t \xi _k= \tau ^{-1} ( \tilde g _k ^{-1} g _k )$, $ \Delta t \tilde \xi _k = \tau ^{-1} ( g_{k-1} ^{-1} \tilde g _k )$. $k = 1, .. ,K-1$ for the second equation and $k = 1, .. , K$ for the rest. The evolution equation for $\alpha$ has the form,
\begin{equation}\label{evo_alpha}
 \tilde \alpha_{k+1} = \tau (\Delta t \tilde \xi_{k+1})^{-1} \cdot \alpha_k, \qquad \alpha_{k+1} =  \tau (\Delta t \xi_{k+1})^{-1} \cdot \tilde \alpha_{k+1},
\end{equation}
a direct result from the definition $\tilde \alpha_k = \tilde g_k^{-1} \cdot \alpha_0$ and $\alpha_k = g_k^{-1} \cdot \alpha_0$.

We have thus derived a stochastic integrator for this system, and we can define the values $\nu^1_k$, $\nu^2_k$, $p_k$ and $\mu_k$ the same way as in the previous sections, with the definitions \eqref{nu_k}, \eqref{def_momentum} and \eqref{def_triv_momentum}.

\subsection{Momentum map and discrete Noether's theorem}

This integrator is also symplectic, following the same proof for the Proposition \ref{symplectic}, and observing that the advected parameter $\alpha$ does not appear at the endpoints in the discrete phase space principle. Furthermore, we have the following Noether's theorem. 

\begin{proposition}[Momentum map and discrete Noether's theorem]\label{NoetherTheorem} Let $G_{\alpha_o}$ be the isotropy subgroup of $\alpha_0$, and $ \mathfrak{g}_{\alpha_0}$ its Lie algebra. Assume that the Hamiltonians $H_{\alpha_0}$ and $H_{i,\alpha_0}$ are left $G_{\alpha_0}$-invariant, so that we can define the reduced Hamiltonians $h$ and $h_i$ as mentioned above. The discrete momentum map,  is then
\begin{equation}\label{discrete_momap2} 
J(g_k, p_k) =  i_{ \alpha_0} ^*\big( p_k g_k^{-1} \big) \in \mathfrak{g}_{\alpha_0} ^* ,
\end{equation} 
where $i_{ \alpha_0} ^* : \mathfrak{g} ^* \rightarrow \mathfrak{g}_{\alpha_0} ^* $ is the dual map of the Lie algebra inclusion. The momentum map is preserved by the discrete stochastic flow determined by the midpoint integrator \eqref{Stoch_Midpoint_LGex}. 
\end{proposition}

\rem{
\todo{\color{blue} FGB: I think we can write something like\[
i_{ \alpha _0} ^*\big( p_k g_k^{-1} \big)= 
 i_{ \alpha _0} ^* \operatorname{Ad}^*_{ g_k ^{-1} } \mu _k = \operatorname{Ad}^*_{ g_k ^{-1} } i^*_{ g _k ^{-1} \cdot \alpha _0}\mu _k =  \operatorname{Ad}^*_{ g_k ^{-1} }i_{ \alpha _k}^* \mu _k 
\]
which allows to use $i_{ \alpha _k }: \mathfrak{g} _{ \alpha _k} \rightarrow \mathfrak{g} $ (the current value $ \alpha _k $) instead of $i_ { \alpha _0}:\mathfrak{g} _{ \alpha _k} \rightarrow \mathfrak{g}$ (the initial value $ \alpha _0$) to express the momentum map. This can be quite useful later for the Kelvin Theorem.
I let you check all this.
}}

\begin{proof}
The left action of $G_{\alpha_0}$ on the discrete curve of random variables
\[
c_d=(g_0 , \tilde g_1, \tilde \mu^1 _1, \mu^2 _1, g_1, ... g_K)\quad\text{is}\quad k c_d=(k g_0 , k \tilde g_1, \tilde \mu^1 _1, \mu^2 _1, k g_1, ... k g_K),
\]
for any $k \in G_{\alpha_0}$. Since $G_{\alpha_0}$ is the isotropy group of $\alpha_0$,  the evaluation of the reduced Hamiltonians $h \Big( \frac{ \tilde\mu^1 _{k}+ \tilde \mu^2 _{k}}{2}, \tilde \alpha _{k} (c_d) \Big)$, $h_i \Big( \frac{ \tilde\mu^1 _{k}+ \tilde \mu^2 _{k}}{2}, \tilde \alpha _{k} (c_d) \Big)$ is also invariant under this action, as a result of :
\[
\tilde \alpha _{k} (k c_d) = (k \tilde g_k)^{-1} \cdot \alpha_0 =  \tilde g_k^{-1} k^{-1} \cdot \alpha_0  =  \tilde g_k^{-1} \cdot \alpha_0 = \tilde \alpha _{k} ( c_d).
\]
Thus the discrete stochastic action functional $ \mathcal{G} _d$ is also left $G_{\alpha_0}$-invariant: $ \mathcal{G} _d( k c_d)= \mathcal{G}_d(c_d)$. We have thus ${\rm d} \mathcal{G} _d( c_d) \cdot \zeta  c_d = 0$, for all $\zeta \in \mathfrak{g}_{\alpha_0}$. The rest of the proof is identical as in Proposition \ref{Noether1}.
\end{proof}

\begin{remark}[Direct proof of Noether's Theorem]\rm Similar to Theorem \ref{Noether1}, the preservation of the momentum map can also be shown directly using the first equation of \eqref{Stoch_Midpoint_LGex}. In fact, expressing with $p_k$ and $p_{k-1}$, the equation can be rewritten as:
\[
\frac{1}{\Delta t}\left(p_{k-1}g_{k-1}^{-1} - p_k g_k^{-1}\right) =  \frac{\pa \mathsf{h}}{\pa g}  \Big( \tilde g_k, \frac{ \tilde \mu^1 _k+ \tilde \mu^2 _k}{2}, \tilde  \alpha_k \Big) \tilde g _k ^{-1} +\operatorname{Ad}^*_{\tilde g_k^{-1}} \left[ \frac{\pa \mathsf{h}}{\pa \alpha}  \Big( \tilde g_{k}, \frac{ \tilde \mu^1 _{k}+ \tilde \mu^2 _{k}}{2} , \tilde \alpha_k \Big) \diamond \tilde \alpha_k \right].
\]
The first term on the right side is in the kernel $\text{ker} \, i_{ \alpha_0} ^* $, as shown before. For the second term, since $G_{\alpha_0}$ is the isotropy group of $\alpha_0$, we have that
\[
\Big\langle \operatorname{Ad}^*_{\tilde g_k^{-1}} \left[ \frac{\pa h}{\pa \alpha}  \diamond \tilde \alpha_k \right] , \zeta \Big\rangle = -\Big\langle  \tilde g_k \cdot \frac{\pa h}{\pa \alpha} , \zeta \tilde g_k \cdot \tilde \alpha_k \Big\rangle = -\Big\langle  \tilde g_k \cdot \frac{\pa h}{\pa \alpha} , \zeta \cdot \alpha_0 \Big\rangle = 0
\]
for any $\zeta \in \mathfrak{g}_{\alpha_0}$. This shows that  $J (g_{k-1},p_{k-1}) =  i_{ \alpha_0} ^*\big( p_{k-1} g_{k-1}^{-1} \big) =  i_{ \alpha_0} ^*\big( p_{k} g_{k}^{-1} \big) = J (g_{k},p_{k})$. By induction, the momentum map is preserved by the flow.
\end{remark}

\begin{remark}[Alternative expression of the momentum map]\rm
Using the equality $p_k g_k^{-1} = \operatorname{Ad}^*_{ g_k ^{-1} } \mu _k$, we can also write the momentum map as
\begin{equation}\label{shifting}
 i_{ \alpha _0} ^*\big( p_k g_k^{-1} \big)= 
 i_{ \alpha _0} ^* \operatorname{Ad}^*_{ g_k ^{-1} } \mu _k =\operatorname{Ad}^*_{ g_k ^{-1} } i^*_{ g _k ^{-1} \cdot \alpha _0}\mu _k =  \operatorname{Ad}^*_{ g_k ^{-1} }i_{ \alpha _k}^* \mu _k,   
\end{equation}
where the second and the third coadjoint operator $\operatorname{Ad}^*_{ g_k ^{-1} }$, with an abuse of notation, is in the restricted sense: $\operatorname{Ad}^*_{ g_k ^{-1} } : \mathfrak{g}^*_{g^{-1}_k \cdot \alpha_0} \rightarrow \mathfrak{g}^*_{\alpha_0}$, which is justified by the fact that for any $\zeta \in \mathfrak{g}_{\alpha_0}$, $\operatorname{Ad}_{ g_k ^{-1} } \zeta \cdot (g^{-1}_k \cdot \alpha_0) = g_k ^{-1} \cdot (\zeta \cdot \alpha_0) = 0$, thus $\operatorname{Ad}_{ g_k ^{-1} } \zeta  \subset \mathfrak{g}_{g^{-1}_k \cdot \alpha_0} $. The equality \eqref{shifting} can be shown following the calculation for any $\zeta \in \mathfrak{g}_{g^{-1}_k}$: 
\begin{align*}
&\left \langle i_{ \alpha_0} ^* \operatorname{Ad}^*_{ g_k ^{-1} } \mu_k , \zeta \right \rangle_{\mathfrak{g}_{\alpha_0}} = \left \langle  \operatorname{Ad}^*_{ g_k ^{-1} } \mu_k , i_{ \alpha_0} \zeta \right \rangle_{\mathfrak{g}} = \left \langle  \operatorname{Ad}^*_{ g_k ^{-1} } \mu_k , \zeta \right \rangle_{\mathfrak{g}} = \left \langle   \mu_k , \operatorname{Ad}_{ g_k ^{-1} } \zeta \right \rangle_{\mathfrak{g}} \\
&= \left \langle   \mu_k ,  i_{ g_k ^{-1}\cdot \alpha_0} \operatorname{Ad}_{ g_k ^{-1} } \zeta \right \rangle_{\mathfrak{g}} = \left \langle   i^*_{ g_k ^{-1}\cdot \alpha_0} \mu_k , \operatorname{Ad}_{ g_k ^{-1} } \zeta \right \rangle_{\mathfrak{g}_{g_k ^{-1}\cdot \alpha_0}} = \left \langle  \operatorname{Ad}^*_{ g_k ^{-1} } i^*_{ g_k ^{-1}\cdot \alpha_0} \mu_k , \zeta \right \rangle_{\mathfrak{g}_{\alpha_0}} .
\end{align*}

This formula allows to use $i_{ \alpha _k }: \mathfrak{g} _{ \alpha _k} \rightarrow \mathfrak{g} $ (the current value $ \alpha _k $) instead of $i_ { \alpha _0}:\mathfrak{g} _{ \alpha _k} \rightarrow \mathfrak{g}$ (the initial value $ \alpha _0$) for the expression of the momentum map, which can be useful for the presentation of Kelvin's Theorem in the fluid case, which will be addressed in a future work.

\end{remark}

\subsection{Preservation of coadjoint orbits, Casimirs, and Lie-Poisson brackets}

In the particular case $K= G_{ \alpha _0}$ there is a concrete description of the symplectic leaves in $(T^*G)/K= (T^*G)/G_{ \alpha _0}\simeq \mathfrak{g} ^* \times \operatorname{Orb}( \alpha _0) \subset \mathfrak{g} ^* \times V^*$. They are given as the coadjoint orbits for the semidirect product Lie group $G \,\circledS\, V$, \cite{MaRaWe1984}. This is concretely checked as follows. First recall that the semidirect Lie group multiplication is given as $(g,u)(h,v)=(gh,g \cdot v+u)$ and that we get the coadjoint actions $ \operatorname{Ad}^*_{(g,u)}( \mu  , \alpha )=\left(\operatorname{Ad}^*_g ( \mu - u \diamond \alpha ), g ^{-1} \cdot \alpha  \right) $ and $\operatorname{ad}^*_{( \xi , u)}( \mu , \alpha ) =  \left( \operatorname{ad}^*_ \xi \mu - u \diamond \alpha , - \xi \cdot \alpha \right) $. Therefore the stochastic Hamiltonian system \eqref{LP_alpha} can be written as a stochastic Lie-Poisson system on the dual of the semidirect product Lie algebra $ \mathfrak{g} \,\circledS\, V$ as
\[
( {\rm d} \mu , {\rm d} \alpha )= \operatorname{ad}^*_{ \left( \frac{\pa h}{\pa \mu }, \frac{\pa h}{\pa \alpha }  \right)  } ( \mu , \alpha ) {\rm d} t +\sum^N_{i=1} \operatorname{ad}^*_{ \left( \frac{\pa h_i}{\pa \mu }, \frac{\pa h_i}{\pa \alpha }  \right)  } ( \mu , \alpha ) \circ {\rm d} W_i(t).
\]
As a consequence, the coadjoint orbits \begin{align*} 
\mathcal{O} ( \mu _0, \alpha _0)&= \{ \operatorname{Ad}^*_{(g,u)}( \mu _0, \alpha _0)\mid (g, u) \in G \,\circledS\, V\} \\
&=\{( \operatorname{Ad}^*_g( \mu _0 - u \diamond \alpha _0), g ^{-1} \cdot \alpha _0)\mid  (g, u) \in G \,\circledS\, V\} \subset ( \mathfrak{g} \,\circledS\, V) ^* 
\end{align*} 
of the semidirect product group are preserved by the continuous flow.

This is also true in the discrete case with the discrete flow determined with \eqref{Stoch_Midpoint_LGex}, as shown in the following proposition.

\begin{proposition}
The discrete flow given by \eqref{Stoch_Midpoint_LGex} preserves the coadjoint orbits $\mathcal{O} ( \mu _0, \alpha _0)$ in $( \mathfrak{g} \,\circledS\, V) ^*$, as well as the Casimirs of the Lie-Poisson bracket.
\end{proposition}

\begin{proof}
First of all we have that 
\[
\alpha_k = (g^{-1}_{k-1}g_k)^{-1}\cdot \alpha_{k-1}.
\]
The first equation of \eqref{Stoch_Midpoint_LGex}, expressed with the trivialized momentum defined in \eqref{def_triv_momentum} is 
\[
\operatorname{Ad}^*_{g^{-1}_{k}\tilde g_k}  \mu_k = \operatorname{Ad}^*_{g^{-1}_{k-1}\tilde g_k}  \mu _{k-1}  - v \diamond \tilde \alpha_k,
\]
with $v = \frac{\pa h}{\pa \alpha}  \Big(\frac{ \tilde \mu^1 _{k}+ \tilde \mu^2 _{k}}{2} , \tilde \alpha_k \Big) \Delta t + \sum^N_{i=1}  \frac{\pa h_i}{\pa \alpha}  \Big(  \frac{ \tilde \mu^1 _{k}+ \tilde \mu^2 _{k}}{2} , \tilde \alpha_k \Big) \Delta W^i_k$.
This equation is equivalent to
\[
\mu_k = \operatorname{Ad}^*_{g^{-1}_{k-1} g_k} \left( \mu_k - (g^{-1}_{k-1}\tilde g_k \cdot v ) \diamond \alpha_{k-1} \right).
\]
Expressed as an adjoint action on the semidirect product group, we get that
\[
(\mu_k,\alpha_k) = \operatorname{Ad}^*_{(g^{-1}_{k-1} g_k \; , \; g^{-1}_{k-1}\tilde g_k \cdot v)} (\mu_{k-1},\alpha_{k-1}).
\]
This shows that the coadjoint orbit is preserved by the scheme \eqref{Stoch_Midpoint_LGex}.
\end{proof}

\color{black}

\section{Convergence: rigid body}\label{sec_convergence}

The strong convergence of several groups of symplectic stochastic integrators of Hamiltonian systems on Eucliean spaces is well known (see for example \cite{MiReTr2002} and \cite{WaHoXu2017}). However, the study of strong convergence of Hamiltonian systems on Lie groups poses difficulty, with the involvement of the retraction map $\tau: \mathfrak{g} \rightarrow G$ whose reverse is used to ``flattens'' the Lie group locally. $\tau$ approximates the exact exponential map from the Lie algebra to the Lie group. The study of the order of approximation of $\tau$ a priori is required for the study of the convergence of the numerical scheme of the Hamiltonian system.

In this section, we study and prove the strong convergence in the mean square sense of the midpoint stochastic integrator applied to the rigid body, with $\tau$ chosen to be the Cayley transform.

\rem{
\todo{At some point, we have to say that we cannot exploit the results in \cite{HoTy2018} because they are not complete.}
\todo{Somewhere we can say how the results of convergence in \cite{MiReTr2002} compare to our case, I=in terms of the order of convergence.}

\color{purple}
\todo{
WM: Compare to the work of \cite{MiReTr2002}. The same degrees of convergence is achieved, strong order 0.5 in the general case, and 1 if commutative noise or only one stochastic part (results in the multiplicative noise section). 

But, it’s hard to compare the results because the nature of the integrators are different. Their scheme is Euler-type scheme, but ours is highly implicit variational scheme. 

From this paper, the strong convergence for the Hamiltonian system on a vector space is well proven. This is indeed a well developed theory. The lack of rigour in the statement of Tyranowski's paper could be due to this, thus causing no harm. }}
\color{black}

\subsection{Rigid body dynamics}\label{5_1}

As a classical example, the motion of a rigid body can be characterized by the rotation matrix $R$ in $SO(3)$, and the Lagrangian $L:TSO(3)\rightarrow\mathbb{R}$ is 
\[
L(R,\dot R) = \frac{1}{2} {\rm Tr}\big(\dot R\mathbb{J} \dot R^\mathsf{T}\big).
\]
It is clearly left $SO(3)$-invariant. It is well known that the Lie algebra $\mathfrak{so}(3)$ is isomorphic to $\mathbb R^3$ via the hat map   $\;\widehat{} \;: \mathbb R^3 \rightarrow \mathfrak{so}(3)$:
\[
\omega = (x,y,z) \rightarrow \widehat{\omega} = \begin{bmatrix}
0 & -z & y \\
z & 0 & -x \\
-y & x & 0
\end{bmatrix}.
\]
The reduced Lagrangian can be expressed as 
\[
l(\Omega) = \frac{1}{2} {\rm Tr }\big(\widehat{\Omega} \mathbb{J} \widehat{\Omega}^\mathsf{T}\big) = \frac{1}{2} \Omega^\mathsf{T} \mathbb{I} \Omega,
\]
where $\Omega$ is the body angular velocity and $\mathbb{I}$ is the body moment of inertia. The resulting Hamiltonian, is thus also left $SO(3)$-invariant, and the reduced Hamiltonian is
\[
h( \Pi) = \frac{1}{2} \Pi \cdot (\mathbb{I}^{-1} \Pi ),
\]
with the body angular momentum $\Pi \in \mathbb{R}^3 \simeq \mathfrak{so}^*(3)$.  Readers can refer to for example \cite{MaRa1999} for a more detailed development.

Given stochastic Hamiltonians $h_i: \mathbb{R}^3\rightarrow\mathbb{R}$, $i=1,...,N$, the continuous stochastic equations \eqref{Phasespace_Lie} for the rigid body is:
\begin{equation}\label{rigidbodydynamics}
{\rm d} \Pi = -\mathbb{I}^{-1}\Pi \times \Pi {\rm d} t - \sum^N_{i=1} \frac{\pa h_i (\Pi) }{\pa \Pi} \times \Pi \circ {\rm d} W_i.
\end{equation}

In the following we take $\tau : \mathbb{R}^3 \simeq \mathfrak{so}(3) \rightarrow SO(3) $ to be the Cayley transform $\tau(A) = \big( I - \frac{\widehat{A}}{2} \big)^{-1} \big( I + \frac{\widehat{A}}{2} \big)$. Its trivialized derivative reads $\left[ {\rm d} _{\Delta t\xi}\tau^{-1}\right]^*\Pi = \Pi + \frac{\Delta t}{2} \xi \times \Pi - \frac{(\Delta t)^2}{4}(\xi \cdot\Pi) \xi$. The stochastic variational integrator \eqref{midpoint_reduced} now reads:
\begin{equation}\label{integrator_rigid_body}
\left\{ 
\begin{array}{l}
\displaystyle\vspace{0.2cm}\frac{\tilde{\Pi}^1_{k}}{\Delta t} - \frac{\tilde{\xi}_{k} \times \tilde{\Pi}^1_{k}}{2}  -  \frac{ \Delta t}{4}(\tilde{\xi}_{k} \!\cdot\!\tilde{\Pi}^1_{k}) \tilde{\xi}_{k}\! =\! \frac{\tilde \Pi^2_{k}}{\Delta t} + \frac{\xi_{k} \times \tilde \Pi^2_{k}}{2}  - \frac{ \Delta t}{4}(\xi_{k}\! \cdot\!\tilde \Pi^2_{k}) \xi_{k}, \\ 
\displaystyle\vspace{0.2cm}\frac{ \tilde \Pi^2_{k}}{\Delta t} - \frac{\xi_{k} \times \tilde \Pi^2_{k}}{2}  - \frac{ \Delta t}{4}(\xi_{k}\! \cdot\!\tilde \Pi^2_{k}) \xi_{k}\!=\! \frac{\tilde{\Pi}^1_{k+1}}{\Delta t} + \frac{\tilde{\xi}_{k+1} \times \tilde{\Pi}^1_{k+1}}{2}   - \frac{ \Delta t}{4}(\tilde{\xi}_{k+1} \!\cdot\!\tilde{\Pi}^1_{k+1}) \tilde{\xi}_{k+1} , \\   
\displaystyle\vspace{0.0cm}\tilde{\xi}_{k}= \xi_{k} = \frac{1 }{2} \mathbb{I}^{-1} \frac{\tilde{\Pi}^1_{k} + \tilde \Pi^2_{k}}{2} + \sum^N_{i=1} \frac{1}{2} \frac{\pa h_i}{\pa \Pi}\left( \frac{\tilde{\Pi}^1_{k} + \tilde \Pi^2_{k}}{2}\right)\frac{\textcolor{purple}{\overline{\Delta W^i_k}}}{\Delta t}.
\end{array}
\right. 
\end{equation}

Recall that by definition,  $ \Delta t \xi _k= \tau ^{-1} ( \tilde R_k ^{-1} R_k )$ and $ \Delta t \tilde \xi _k = \tau ^{-1} ( R_{k-1} ^{-1} \tilde R_k )$. Thus the reconstruction equations for the rotation matrix $R$ are: $R_k =  \tilde R_k \tau(\Delta t \xi _k)$ and $\tilde R_{k} = R_{k-1} \tau(\Delta t \tilde \xi _{k})$.

One small modification is made in the last equation of the integrator for the rigid body \eqref{integrator_rigid_body}. Specifically, we take the truncated increment of the Wiener process, $\overline{\Delta W^i_k}$, at each time step, which is defined as follows:
\begin{equation}\label{truncation}
\overline{\Delta W^i_k} =
\begin{cases}
    & \Delta W^i_k \; \; \mathrm{when} \, |\Delta W^i_k | \leqslant D_{\Delta t}, \\ 
    & D_{\Delta t} \; \; \mathrm{when} \, \Delta W^i_k > D_{\Delta t}, \\
    & -D_{\Delta t} \; \; \mathrm{when} \, \Delta W^i_k < -D_{\Delta t},
\end{cases}  
\end{equation}
where $D_{\Delta t}$ is the truncation value, a function of the time step size $\Delta t$. The exact choice of its expression depends on the targeted convergence order of the scheme. For a strong convergence rate of 1, we set $D_{\Delta t} = \sqrt{4|{\rm ln} \Delta t | \Delta t}$. See \cite{MiReTr2002} for a detailed study on truncation. 

Since $\Delta W^i_k$ follows a normal distribution $\mathcal{N}(0, \Delta t)$, we can readily estimate the moments of the truncated $\overline{\Delta W^i_k}$ from its definition as follows:
\begin{equation}\label{truncation_moments}
E \left[ \left( \overline{\Delta W^i_k} \right)^p \right]
\begin{cases}
    & =0 \; \; \text{when} \; p \; \text{is odd,}\\ 
    & \leqslant E \left[ \left( \Delta W^i_k \right)^p \right] = (\Delta t)^{\frac{p}{2}} (p-1)!! \; \; \mathrm{when} \; p \; \text{is even.}
\end{cases}  
\end{equation}
Truncating $\Delta W^i_k$ imposes a bound on the increment of the Wiener processes, which is necessary for the numerical solvability of implicit schemes. However, truncation does not invalidate the conservation properties proven in the previous sections, as the normal distribution property of $\Delta W^i_k$ was not utilized in the proofs.

\rem{
\todo{In reality, all the results apart from convergence we have are valid for any stochastic processes as the stochastic integrator $\circ {\rm\;  d} Y$. This will make it possible to simulate a dissipative stochastic process, by finding a suitable integrator function?}}

\color{black}
In accordance with the definition \eqref{def_momentum}, the body momentum (trivialized momentum) at the time step $k = 1,..,K-1$ has two equivalent expressions
\begin{equation}\label{def_Pi}
\Pi_k = \tilde \Pi^2_{k} - \frac{\Delta t}{2} \xi_{k} \times \tilde \Pi^2_{k} - \frac{ \Delta t^2}{4}(\xi_{k}\! \cdot\!\tilde \Pi^2_{k}) \xi_{k} = \tilde{\Pi}^1_{k+1} + \frac{\Delta t}{2}\tilde{\xi}_{k+1} \times \tilde{\Pi}^1_{k+1}   - \frac{ \Delta t^2}{4}(\tilde{\xi}_{k+1} \!\cdot\!\tilde{\Pi}^1_{k+1}) \tilde{\xi}_{k+1},
\end{equation}
and at $k=0$ and $k=K$:
\begin{equation}\label{def_PiEndpoint}
\begin{aligned}
  &\Pi_0 = \tilde{\Pi}^1_{1} + \frac{\Delta t}{2}\tilde{\xi}_{1} \times \tilde{\Pi}^1_{1}  -  \frac{ \Delta t^2}{4}(\tilde{\xi}_{1} \!\cdot\!\tilde{\Pi}^1_{1}) \tilde{\xi}_{1}, \\
  &\Pi_K= \tilde \Pi^2_{K} - \frac{\Delta t}{2}\xi_{K} \times \tilde \Pi^2_{K}  - \frac{ \Delta t^2}{4}(\xi_{K}\! \cdot\!\tilde \Pi^2_{K}) \xi_{K}.    
\end{aligned} 
\end{equation}

\subsection{Conservation properties}

The coadjoint orbit $\mathcal{O}_{0} = \{ \operatorname{Ad}^*_R  \Pi_0 = R^{-1} \Pi_0 \mid R \in SO(3) \}$ is the sphere in $\mathbb{R}^3$ of radius $\| \Pi_0 \|$, which is preserved by the scheme by Proposition \ref{coadj_orb}.
This means that the Euclidean norm $\| \Pi_k \|$ is constant, i.e., the path of $\Pi_k$ adheres to the ball of radius $\| \Pi_0 \|$, for any choice of stochastic Hamiltonian.

The kinetic energy $E(\Pi_k) = h(\Pi_k) = \frac{1}{2}\Pi_k \cdot (\mathbb{I}^{-1}\Pi_k)$ is no longer conserved in general, due to the introduction of stochastic Hamiltonians, already at the continuous case. Its time rate of change is  ${\rm d} h= \sum_{i=1}^N\{h,h_i\}\circ {\rm d}W_i$. It is, nevertheless, bounded,  $\frac{1}{2}\Pi_k \cdot (\mathbb{I}^{-1}\Pi_k ) \leqslant \frac{1}{2}\|\mathbb{I}^{-1} \| \| \Pi_k \| = \frac{1}{2}\|\mathbb{I}^{-1} \| \| \Pi_0 \|$.
In absence of stochastic Hamiltonians, the deterministic midpoint variational method preserves the energy exactly. The proof for energy conservation is in the Appendix \ref{proof_energy}. In all numerical simulations to be mentioned below with the deterministic integrator, the energy is conserved up to order $E - 10$. 

The momentum map \eqref{discrete_momap} is found as $J_k= \operatorname{Ad}^*_{ R_k^{-1}} \Pi_k = R_k \Pi_k$, and has the following form:
\begin{equation}\label{def_J_k}
J_k = \begin{cases}
R_k \left(\tilde \Pi^2_{k} - \frac{\Delta t}{2} \xi_{k} \times \tilde \Pi^2_{k} - \frac{ \Delta t^2}{4}(\xi_{k}\! \cdot\!\tilde \Pi^2_{k}) \xi_{k} \right) & \text{for all $k$ except } k = 0, \\
R_k \left( \tilde{\Pi}^1_{k+1} + \frac{\Delta t}{2}\tilde{\xi}_{k+1} \times \tilde{\Pi}^1_{k+1}   - \frac{ \Delta t^2}{4}(\tilde{\xi}_{k+1} \!\cdot\!\tilde{\Pi}^1_{k+1}) \tilde{\xi}_{k+1} \right) & \text{for all $k$ except } k = K.
\end{cases} 
\end{equation}

This is in fact the spatial angular momentum $\pi_k = R_k \Pi_k$. It is preserved by the flow following Noether's Theorem. A fact which can be directly shown by applying the operator $\operatorname{Ad}^*_{\Delta t \xi_k}$ to the first equation of \eqref{integrator_rigid_body} and by comparing it with the second equation. Note that $\tau(-\Delta t \xi_k)\tau(-\Delta t \tilde \xi_{k}) = R^{-1}_k R_{k-1}$. In all the numerical experiments of unit order ($\|\Pi\| \sim 1 $), the momentum map $J_k$ is conserved up to order $E-10$.

\subsection{Convergence}\label{convergence}
In most literature (for example \cite{Mi1995, KoPl1992, MiReTr2002}), strong convergence has been proven for various numerical methods of stochastic differential equations, with the assumption that the drift and the diffusion coefficients are Lipschitz and have linear growth bound. These conditions, however, are not fulfilled by the drift coefficient of the Euler equation of the rigid body: $-\mathbb{I}^{-1} \Pi \times \Pi$.

In \cite{MaSz2013}, Mao and Szpruch have shown that in the lack of Lipschitz continuity and linear growth bound of the coefficients, a Euler–Maruyama type scheme that has both the continuous and the discrete solutions bounded can still be proven to be strongly convergent with the order 0.5.

This inspires our study of convergence, since the midpoint variational stochastic scheme that we have derived, inheriting the variational nature of Lagrangian mechanics, perfectly preserves the coadjoint orbits, which will facilitate the study of convergence. In this section, we will prove the convergence of the midpoint variational scheme in the particular case of rigid body and with the particular choice of stochastic Hamiltonians: $h_i(\Pi)=\chi_i\cdot \Pi$, $i=1,...,N$.

\begin{theorem}\label{convergence_theorem}
With stochastic Hamiltonians in the form $h_i(\Pi)=\chi_i\cdot \Pi$, $i=1,...,N$, where $\chi_i$ are constant vectors in $\mathbb{R}^3$ and the choice of trunctation value $D_{\Delta t} = \sqrt{4|{\rm ln} \Delta t | \Delta t}$, the midpoint stochastic integrator \eqref{integrator_rigid_body} has the strong mean square convergence order 0.5. Precisely speaking, at horizon $T=t_K$ ($K$ is the number of the time steps such that $K \Delta t = T$), 
\[
\sqrt{E\left[ \sup_{0\leqslant k \leqslant K} \left(\Pi_k - \pi(t_k)\right)^2\right]} =  O(\Delta t ^ {1/2}),
\]
where $\pi(t)$ is the solution of SPDE ${\rm d} \pi = -\mathbb{I}^{-1}\pi \times \pi {\rm d} t - \sum^N_{i=1} \chi_i \times 
\pi \circ {\rm d} W_i $, and $\Pi_k$ is the result given by the midpoint stochastic integrator.

The strong mean square convergence order 1 can be obtained,
\[
\sqrt{E\left[ \sup_{0\leqslant k \leqslant K} \left(\Pi_k - \pi(t_k)\right)^2\right]} =  O(\Delta t )
\]
when there is only one stochastic component $N=1$.
\end{theorem}

The proof of this important result, is presented in the appendix. 

\section{Numerical experiment: rigid body}\label{rigid_body}
\subsection{Numerical simulations}

In the following numerical simulations, we take $\mathbb{I} = {\rm diag}(1,2,3)$,  and $\Delta t = 0.01$. We will simulate the rigid body motion with initial conditions that satisfy $\| \Pi_0 \| = 1$. Due to the conservation of the body angular momentum $\Pi$ and of the kinetic energy, $E(\Pi) = h(\Pi) = \frac{1}{2}\Pi \cdot (\mathbb{I}^{-1}\Pi)$, the deterministic paths of $\Pi(t)$ are on the intersection of the angular momentum sphere and the energy ellipsoid, shown as the black contours in the figures below (Figure \ref{paths}).

As the patterns of deterministic paths clearly show, when rotating about the principal axis with the largest moment ($\Pi = (0,0,\pm 1)$) and about the principal axis with the smallest moment ($\Pi = (\pm 1,0,0)$) the equilibrium is stable, while rotating about the principal axis with the intermediate moment ($\Pi = (0, \pm 1,0)$), the equilibrium is unstable. Apart from the paths passing through the unstable equilibrium states, all other deterministic paths encircle around one of the four stable equilibrium states on the angular momentum sphere. 

For the stochastic experiments, we take the stochastic Hamiltonians to be $h_i (\Pi) =  \Pi \cdot \chi_i $, $N=3$ and $\chi_1 = ( 0.02,0,0)$, $\chi_2 = ( 0,0.02,0 )$ and $\chi_3 =(0,0,0.02)$. We have thus $\tilde{\xi}_{k}= \xi_{k} = \frac{1 }{2} \mathbb{I}^{-1} \frac{\tilde{\Pi}^1_{k} + \tilde \Pi^2_{k}}{2} + \sum^3_{i=1} \frac{1}{2} \chi_i \frac{\Delta W^i_k}{\Delta t}$. With this stochastic forcing, the evolution of $\Pi$ may transition between different energy levels while still adhering to the angular momentum sphere. This feature makes it possible for the stochastic path to significantly move away from the original deterministic path after passing into a different quarter of the sphere, changing the principle axis of rotation. The figure below on the left showcases a stochastic path that keeps close to the deterministic path between $0 \leqslant t \leqslant 50$, while the figure on the right shows a stochastic path that moves away significantly from its deterministic path. 

\begin{figure}[h!]
\begin{subfigure}[c]{0.5\textwidth}
\includegraphics[width=\linewidth]{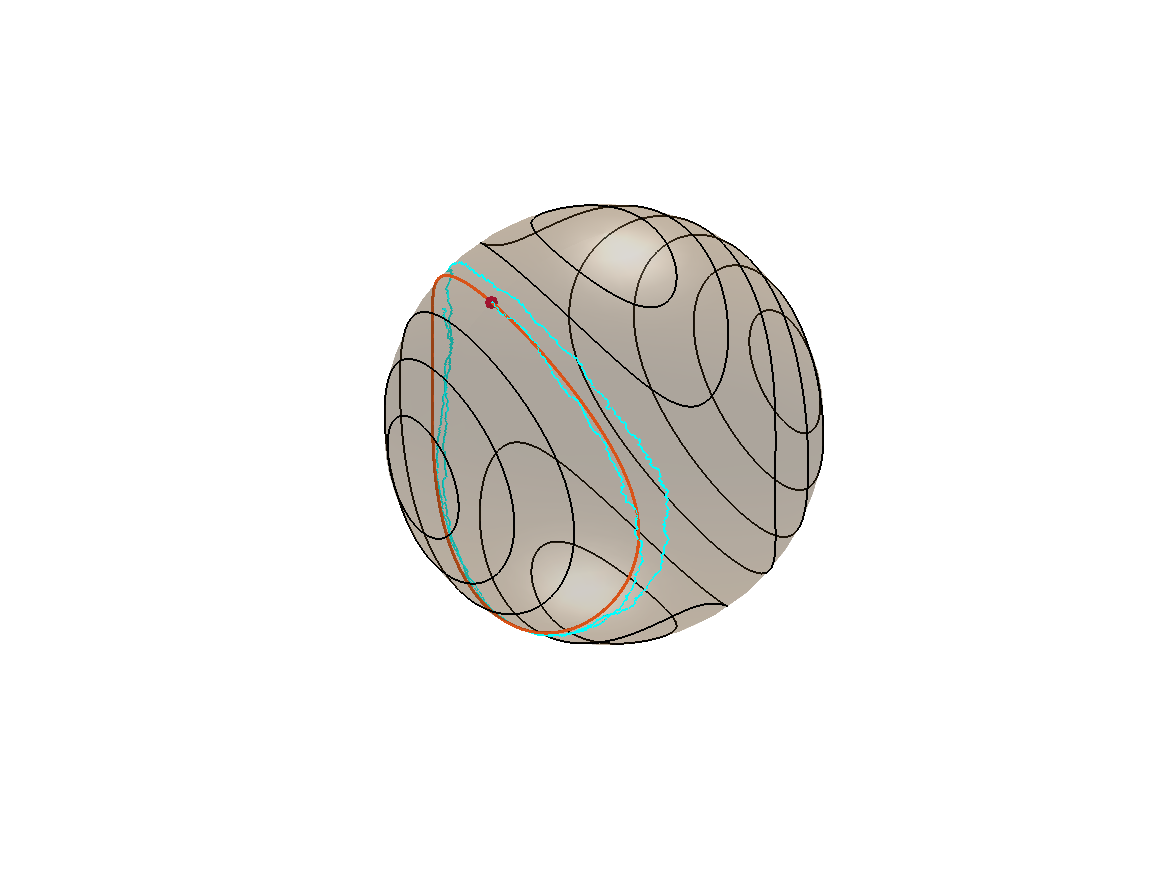} 
\label{fig:subim1}
\end{subfigure}
\begin{subfigure}[c]{0.5\textwidth}
\includegraphics[width=\linewidth]{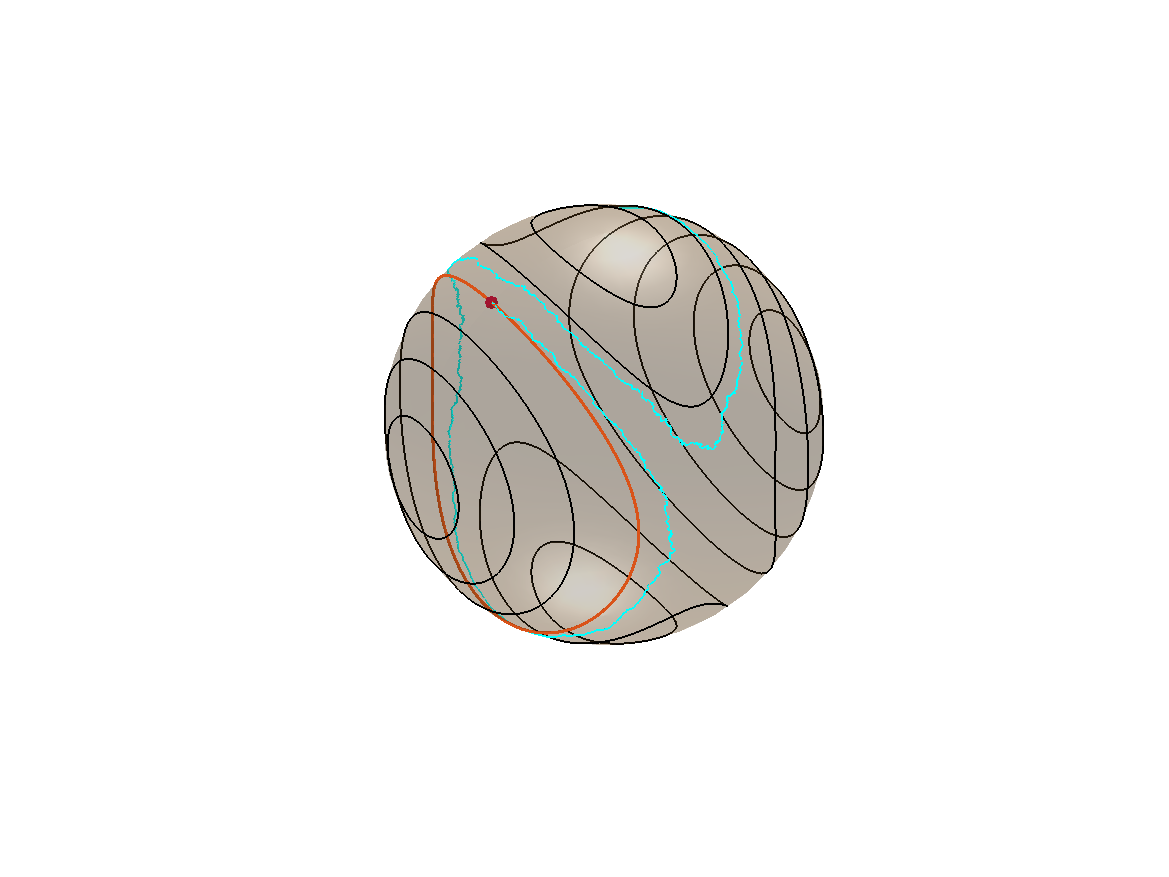}
\label{fig:subim2}
\end{subfigure}
\caption{Two stochastic paths on the angular momentum sphere, with the initial condition $\Pi_0 = (-0.5878, 0, 0.8090)$, marked by the red dot. The deterministic path of the same initial condition is highlighted in red.}
\label{paths}
\end{figure}

The stochastic integrator provides a consistent technique for ensemble forecasting. Starting from an initial condition, an ensemble of stochastic paths can be generated with the midpoint method. The result can be used for qualitative studies of the uncertainties generated with the stochastic Hamiltonians, with the benefit of the above mentioned conservation properties always observed. 

\begin{figure}
    \centering
    \begin{subfigure}[b]{0.49\textwidth}
        \centering
        \includegraphics[width=\textwidth]{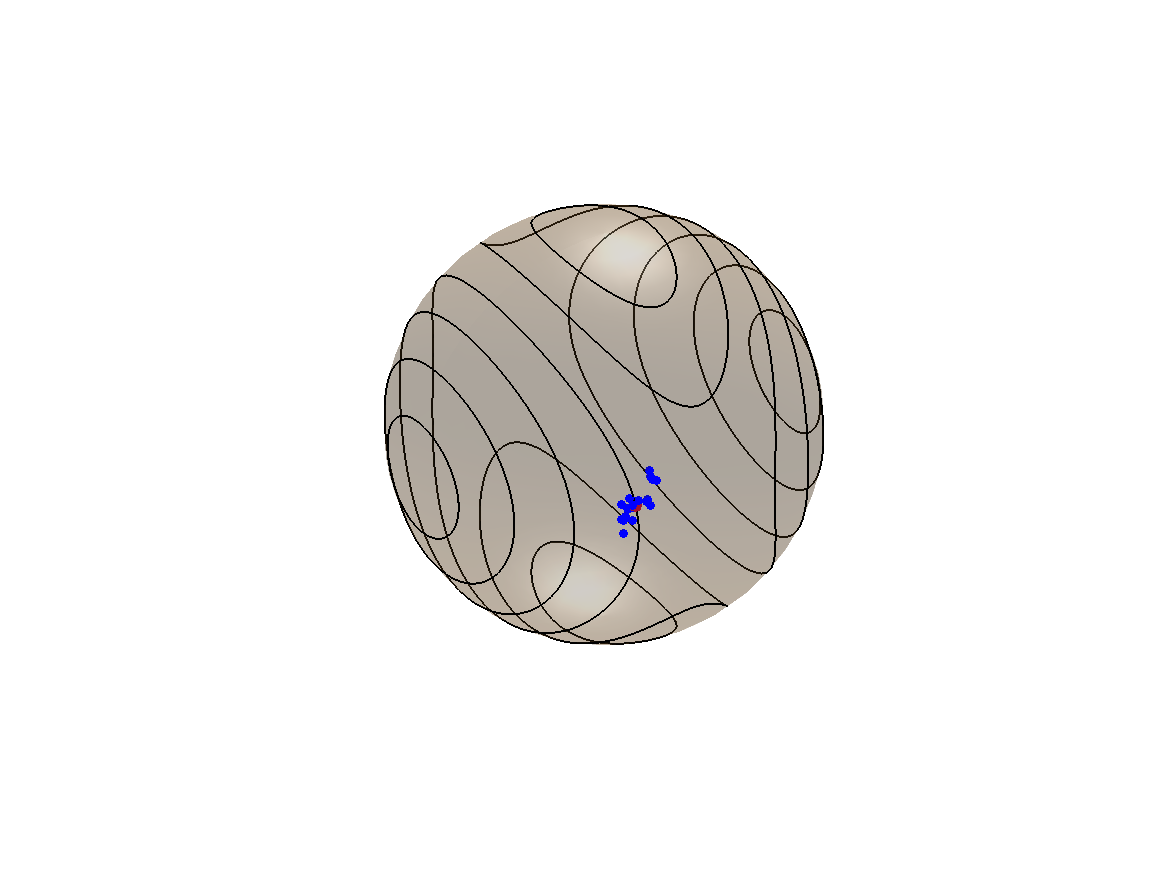}
        \caption{$t = 5$}
        \label{fig:image1}
    \end{subfigure}
    \hfill
    \begin{subfigure}[b]{0.49\textwidth}
        \centering
        \includegraphics[width=\textwidth]{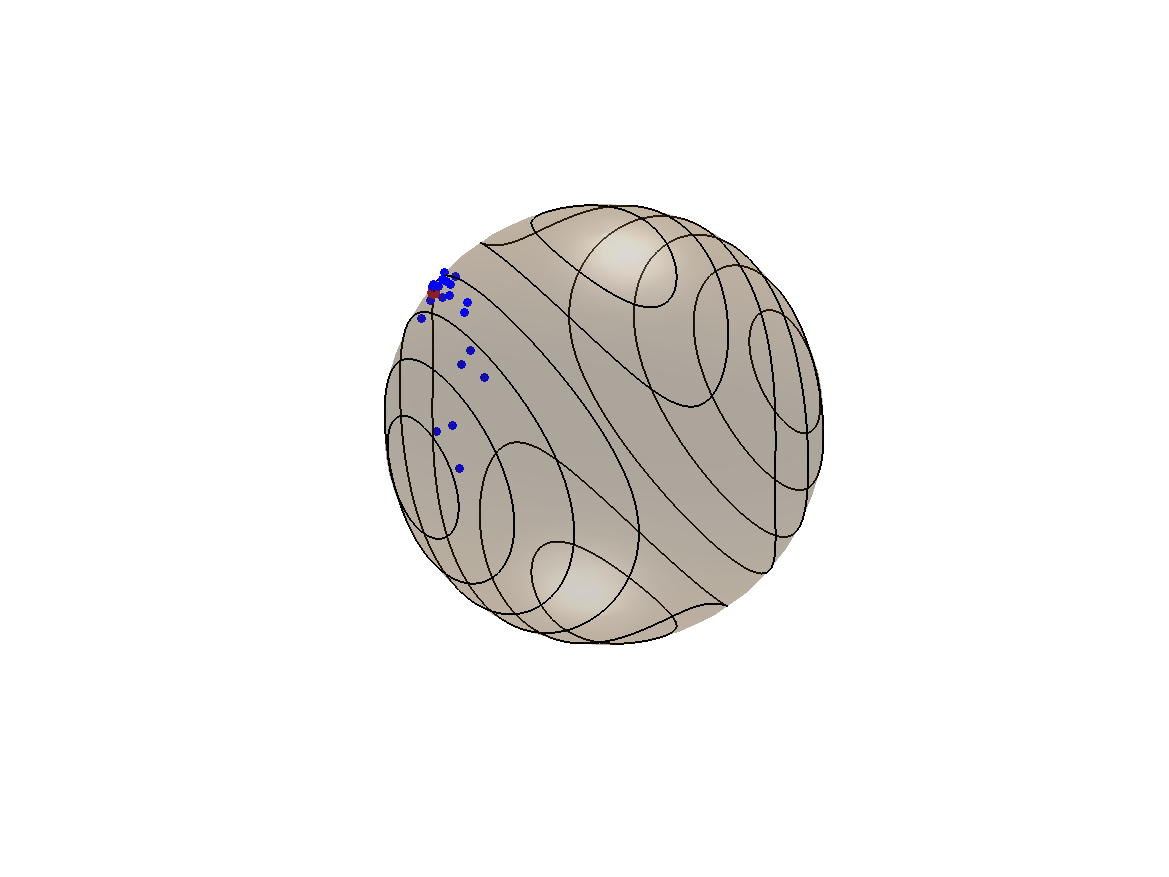}
        \caption{$t = 20$}
        \label{fig:image2}
    \end{subfigure}
    \vskip\baselineskip
    \begin{subfigure}[b]{0.49\textwidth}
        \centering
        \includegraphics[width=\textwidth]{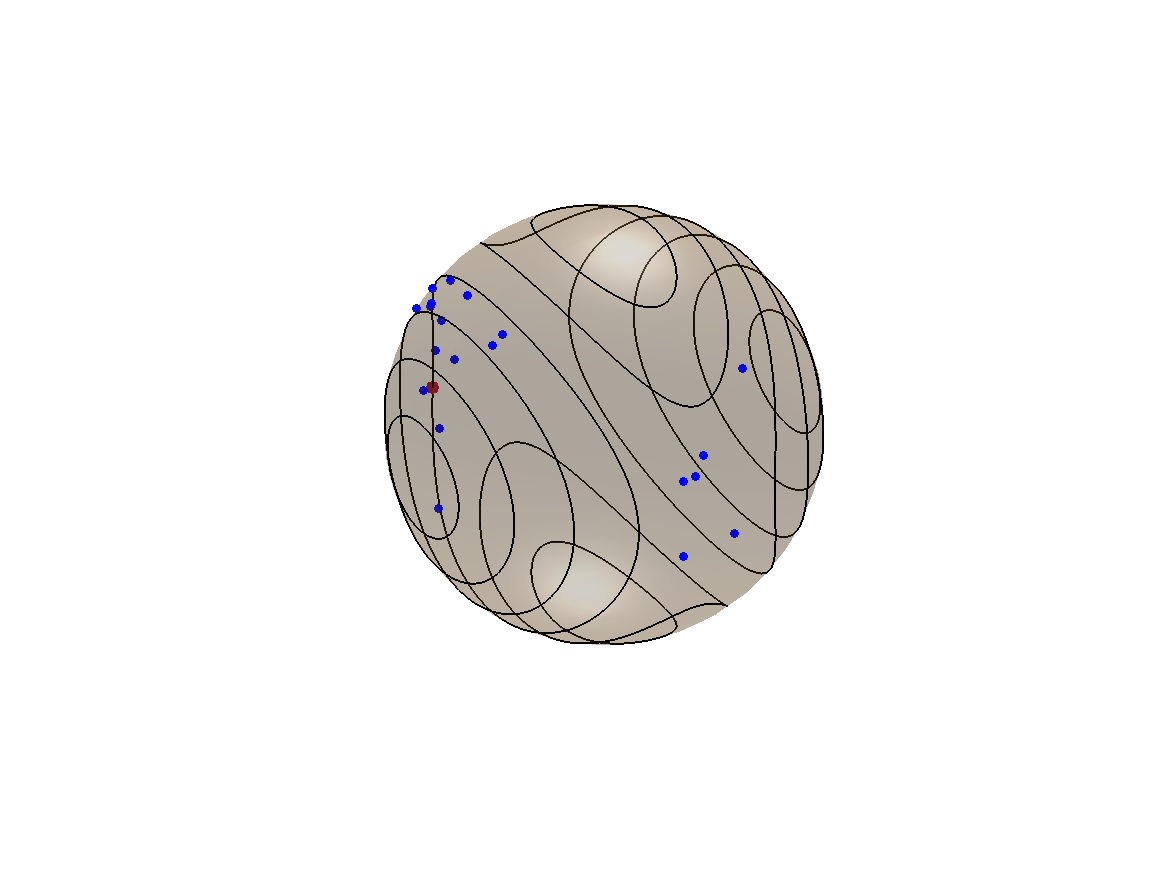}
        \caption{$t = 40$}
        \label{fig:image3}
    \end{subfigure}
    \hfill
    \begin{subfigure}[b]{0.49\textwidth}
        \centering
        \includegraphics[width=\textwidth]{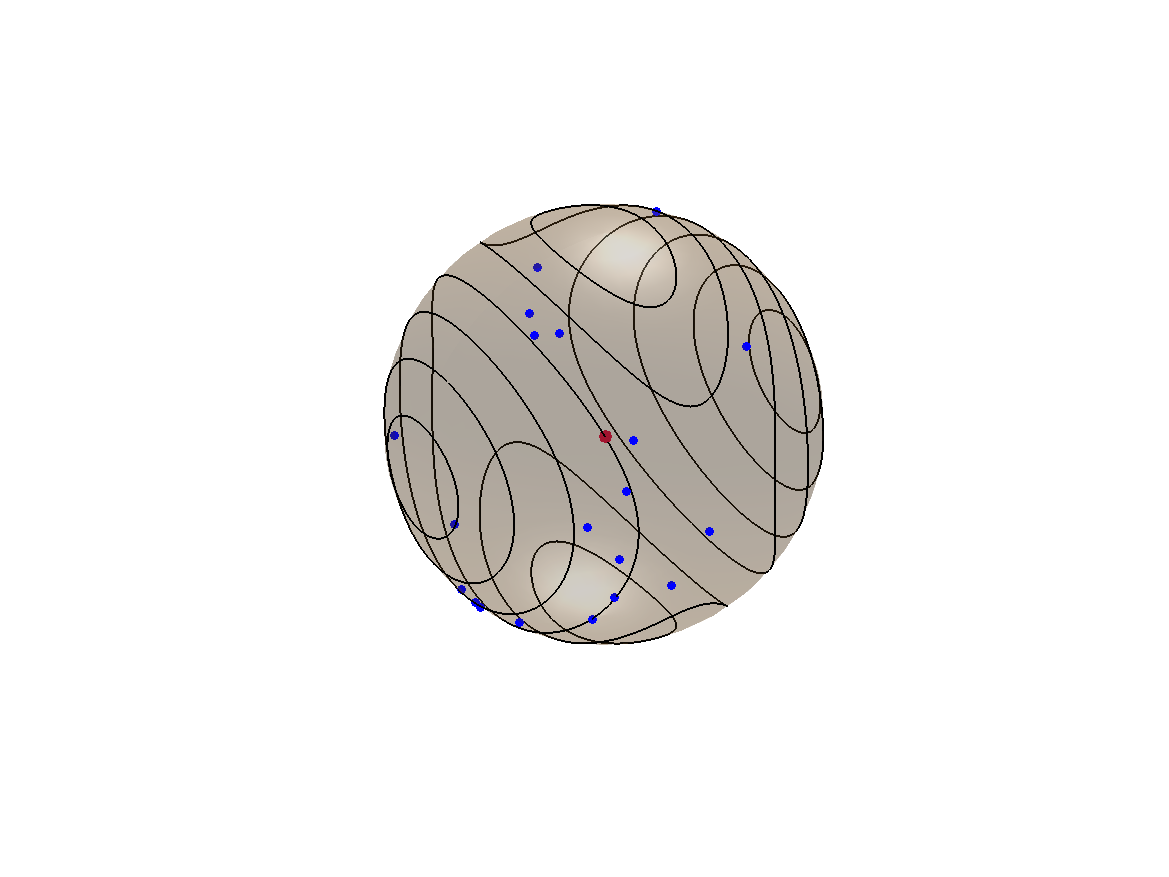}
        \caption{$t = 50$}
        \label{fig:image4}
    \end{subfigure}
    \caption{With the initial condition $\Pi_0 = (-0.5878, 0, 0.8090)$, an ensemble of 20 stochastic paths are generated. Each image shows the positions of $\Pi_k$ of the ensemble at the given time horizon. The deterministic $\Pi_k$ is marked in red as reference. The distances between the forecasts of the stochastic paths tend to increase as time increases. The possibility of bifurcation significantly changes the scattering pattern.}
    \label{fig:four_images}
\end{figure}

\section{Numerical experiments: heavy top}\label{heavy_top}

\subsection{Heavy top dynamics}
The heavy top is a classical example of systems with advected parameters. Its Lagrangian with parameter $v$, in non-reduced form is:
\[
L_v (R,\dot R) = \frac{1}{2} {\rm Tr}\big(\dot R\mathbb{J} \dot R^\mathsf{T}\big) - mgv \cdot (Ra),
\]
with $R \in SO(3)$ the rotation matrix, $m$ the mass of the body, $g$ the gravitational acceleration, $a \in \mathbb{R}^3$ the constant vector in the body coordinate pointing from the attachment point to the mass center of the heavy top and $v =(0,0,1) \in \mathbb{R}^3 \simeq (\mathbb{R}^3)^*$ the $z$-directional unit vector. The additional term, $mgv \cdot (Ra)$ is the potential energy induced by the force of gravity. 

Its reduced Lagrangian, is thus
\[
\ell(\Omega, \Gamma) = \frac{1}{2} \Omega^\mathsf{T} \mathbb{I} \Omega - mg \Gamma \cdot a,
\]
with $\Omega \in \mathbb{R}^3 \simeq \mathfrak{so}(3)$, $\mathbb{I}$ the spatial moment of inertia tensor and $\Gamma := R^{-1}v$ the advected parameter. 

The reduced Hamiltonian is 
\[
h(\Pi, \Gamma) = \frac{1}{2} \Pi \cdot (\mathbb{I}^{-1} \Pi ) + mga \cdot \Gamma,
\]
where $\Pi \in \mathbb{R}^3 \simeq \mathfrak{so}^*(3)$ is the body angular momentum.

Using the midpoint method for systems with an advected parameter, as obtained from \eqref{Stoch_Midpoint_LGex} with $\tau$ as the Cayley transform (see \S\ref{5_1}), and with the reduced stochastic Hamiltonians  $h_i (\Pi,\Gamma)$, we obtain
\begin{equation}\label{integrator_heavytop}
\left\{ 
\begin{array}{l}
\displaystyle\vspace{0.0cm}\Big(\frac{\tilde{\Pi}^1_{k}}{\Delta t} - \frac{\tilde{\xi}_{k} \times \tilde{\Pi}^1_{k}}{2}  -  \frac{ \Delta t}{4}(\tilde{\xi}_{k} \!\cdot\!\tilde{\Pi}^1_{k}) \tilde{\xi}_{k}\Big)\! - 
\Big( mga +  \sum^N_{i=1} \frac{ \pa h_i }{\pa \Gamma} \Big( \frac{ \tilde{\Pi}^1_{k} + \tilde{\Pi}^2_{k}}{2},\tilde \Gamma_k \Big)\frac{\Delta W^i_k}{\Delta t} \Big) \times \tilde \Gamma_k  \\
\displaystyle\vspace{0.4cm}\quad =\! \Big(\frac{\tilde{\Pi}^2_{k}}{\Delta t} + \frac{\xi_{k} \times \tilde{\Pi}^2_{k}}{2}  - \frac{ \Delta t}{4}(\xi_{k}\! \cdot\!\tilde{\Pi}^2_{k}) \xi_{k}\Big) ,\\ 
\displaystyle\vspace{0.4cm}\Big(\frac{\tilde{\Pi}^2_{k}}{\Delta t} - \frac{\xi_{k} \times \tilde{\Pi}^2_{k}}{2}  - \frac{ \Delta t}{4}(\xi_{k}\! \cdot\!\tilde{\Pi}^2_{k}) \xi_{k}\Big)\!=\!  \Big(\frac{\tilde{\Pi}^1_{k+1}}{\Delta t} + \frac{\tilde{\xi}_{k+1} \times \tilde{\Pi}^1_{k+1}}{2}   - \frac{ \Delta t}{4}(\tilde{\xi}_{k+1} \!\cdot\!\tilde{\Pi}^1_{k+1}) \tilde{\xi}_{k+1} \Big), \\   \displaystyle\vspace{0.4cm}\tilde{\xi}_{k}= \xi_{k} = \frac{1}{2} \mathbb{I}^{-1} \frac{\tilde{\Pi}^1_{k} + \tilde{\Pi}^2_{k}}{2} + \sum^N_{i=1} \frac{1}{2} \frac{ \pa h_i}{\pa \Pi}  \Big( \frac{\tilde{\Pi}^1_{k} + \tilde{\Pi}^2_{k}}{2},\tilde \Gamma_k \Big) \frac{\Delta W^i_k}{\Delta t}, \\
\displaystyle\vspace{0.2cm}\Gamma_{k} = \tau(-\Delta t \xi_{k}) \tilde \Gamma_{k} , \quad \tilde \Gamma_{k} = \tau(-\Delta t \tilde \xi_{k})  \Gamma_{k-1}.
\end{array}
\right.
\end{equation} 
Note that as always $ \Delta t \xi _k= \tau ^{-1} ( \tilde R_k ^{-1} R_k )$ and $ \Delta t \tilde \xi _k = \tau ^{-1} ( R_{k-1} ^{-1} \tilde R_k )$. The last two equations are the evolution equations for the advected parameter $\Gamma$.

\subsection{Conservation properties}

We can similarly define the body angular momentum  $\Pi_k$	as in \eqref{def_Pi} and \eqref{def_PiEndpoint}. With the introduction of the advected parameter, the symmetry is broken, meaning the spatial angular momentum 
$\pi_k = R_k \Pi_k$ is no longer conserved. However, the Hamiltonians remain $G_{v}$-invariant, where 
$G_{v}$ is the isotropy group of 
$v=(0,0,1)$, i.e., the group of rotation matrices around the $z$-axis. By Noether's theorem (Theorem \ref{NoetherTheorem}), the momentum map $J_k = i^*_{v} (R_k \Pi_k)$ is preserved along the flow, where  $i_{ v} ^* : \mathfrak{g} ^* \rightarrow \mathfrak{g}_{v} ^* $ is the dual map of the Lie algebra inclusion $ \mathfrak{g}_{v} \hookrightarrow \mathfrak{g}$. When identifying $\mathfrak{so}(3)$ with $\mathbb{R}^3$, $\mathfrak{g}_{v}$ corresponds to the $z$-axis, making it clear that $i^*_{v}$ is the projection of $\mathbb{R}^3 \simeq \mathfrak{g}$ onto the $z$-axis. Therefore, the $z$-coordinate of $R_k \Pi_k$ is conserved. Numerical experiments with $\Pi$ and $mga$ of unit order, confirm that the $z-$component of the momentum map is conserved up to $E - 13$.

The Lie-Poisson bracket 
for the heavy top is associated to the semidirect product $\mathfrak{se}(3)=\mathfrak{so}(3)\,\circledS\,\mathbb{R}^3$ and has the form: 
\[
\{f,h\} (\Pi,\Gamma) = - \Pi \cdot \left( \frac{\partial f}{\partial \Pi}\times\frac{\partial h}{\partial \Pi} \right) - \Gamma \cdot \left( \frac{\partial f}{\partial \Pi}\times\frac{\partial h}{\partial \Gamma}  - \frac{\partial f}{\partial \Gamma}\times\frac{\partial h}{\partial \Pi}  \right).
\]
The Casimirs are $|\Gamma|^2$ and $\Pi \cdot \Gamma$ which are preserved by our scheme. It turns out that $\Pi \cdot \Gamma= \Pi\cdot R^{-1}\hat z= R\Pi\cdot \hat z= J(R,\Pi)$ hence the conservation of this Casimir coincides with the Noether theorem.

Unlike the case of the rigid body, the total energy of the heavy top is not exactly conserved by the deterministic midpoint method, induced by the approximation function $\tau$ in the update equation of $\Gamma$. The energy fluctuation depends on the size of time step. With time step $h=0.01$, $\Pi$ and $mga$ of unit order, the energy fluctuation is of degree $E-5$, thereby following the typical oscillating behavior of symplectic integrators.

In the case of a symmetric top, whose principal moments of inertia about the first two axes are equal, and the displacement $a$ is lined up with the third axis, the body angular momentum about the third axis is an another quantity conserved by the deterministic motion. 
This is a case that we will cover in the next section. Take $\mathbb{I} = {\rm diag}(I,I,I_3)$, and $a = (0,0,a_z)$, then the $z$-component of $\Pi_k$ is conserved by the deterministic midpoint method. This can be checked as follows: from the first equation of \eqref{integrator_heavytop} : 
\[
\Pi_k - \Pi_{k-1} = - \frac{1}{4} \mathbb{I}^{-1} (\tilde \Pi^1_k + \tilde \Pi^2_k ) \times (\tilde \Pi^1_k + \tilde \Pi^2_k ) - mga \times \tilde \Gamma_k.
\]
Since $\mathbb{I}^{-1} = {\rm diag}(I^{-1},I^{-1},I_3^{-1})$, and $a$ is a vector parallel to the $z$-axis, we get that $(\Pi_k - \Pi_{k-1})_{z} = 0$. The $z$-component of $\Pi_k$ is conserved.

\subsection{Modelling gyroscopic precession} 

We consider the motion of a flat symmetric top with $\mathbb{I} = {\rm diag}(I_1 ,I_1 ,I_3 )$, $a = (0,0,1)$ aligned with the $z$-axis of the body. With its $z$-axis of the body inclined at some angle $\theta_0 \in (0, \pi)$ and with the body angular momentum principally in the $z$ direction, the motion of the heavy top is generally described as gyroscopic precession. The three major components of the motion: precession, nutation and spin, can be better described with the Euler angles: $(\phi, \theta, \psi)$, and their time derivatives.

For a given triple of Euler angles $(\phi, \theta, \psi)$, the corresponding rotation matrix is given by
\[
\resizebox{\hsize}{!}{$
R(\phi, \theta, \psi) = \begin{bmatrix}
\cos\phi \cos\psi - \cos\theta \sin\phi \sin\psi & -\cos\phi \sin\psi - \cos\theta \sin\phi \cos\psi & \sin\theta\sin\phi  \\
\sin\phi \cos\psi - \cos\theta \cos\phi \sin\psi & -\sin\phi \sin\psi + \cos\theta \cos\phi \cos\psi & -\sin\theta\cos\phi \\
\sin\theta \sin\psi & \sin\theta \cos\psi & \cos\theta
\end{bmatrix}.$}\]
For the triples in the range $(\phi, \theta, \psi) \in (0,\pi) \times \left[ 0, 2\pi \right)\times\left[ 0, 2\pi \right) $, the mapping is injective, thus it is possible to determine uniquely the Euler angles in this range for any rotation matrix in the subgroup $\left\{ -1 < R_{33} < 1 \left| \right. R \in SO(3) \right\}$, which represents all the rotations that do not fix the $z$-axis. 

The angular velocities in the body coordinates can be related to the angular velocities in respect to the Euler angles with the formulas
\begin{align*}
\omega_x &= \omega_{\phi} \sin\theta \sin\psi + \omega_{\theta} \cos\psi, \\
\omega_y &= \omega_{\phi} \sin\theta \cos\psi - \omega_{\theta} \sin\psi, \\ 
\omega_z &= \omega_{\phi} \cos\theta + \omega_{\psi}.
\end{align*}
With the triple Euler angles $(\phi, \theta, \psi)$ obtained from the rotation matrix $R$, it is thus also possible to retrieve the the angular velocities in respect to the Euler angles from the body angular momentum $\Pi$, knowing the relation $\Pi = \mathbb{I} \omega$. Note that in the stochastic case, we do not necessarily have $\omega_{\alpha} = \dot{\alpha}$ for the Euler angles. A detailed introduction of Euler angles and their applications to describe gyroscopic precessions can be found in Chapter 11 of \cite{ThMa2004}. In the following we will give a quick review of the fundamental physical results that are important to both the deterministic and the stochastic simulations. 

The two preserved quantities $(R\Pi)_z$ and $\Pi_z$ are actually the conjugate momenta to $\phi$ and $\psi$ respectively, which are both cyclic in the heavy top Lagrangian due to symmetry. They have the expression with Euler angles: 
\[
p_{\phi} = (R\Pi)_z = (I_1 \sin^2\theta + I_3 \cos^2\theta)\omega_{\phi}+I_3 \omega_{\psi} \cos \theta, \; \; p_{\psi}=\Pi_z=I_3 (\omega_{\psi}+\omega_{\phi}\cos \theta).
\]
Thus the angular velocity of precession and spin can be expressed as functions of $p_{\phi}$, $p_{\psi}$ and the inclination angle $\theta$:
\begin{equation}\label{omega_phi-omega_psi}
\omega_{\phi} = \frac{p_{\phi}-p_{\psi}\cos \theta}{I_1 \sin^2\theta}, \quad \omega_{\psi} = \frac{p_{\psi}}{I_3} - \frac{(p_{\phi}-p_{\psi}\cos \theta)\cos \theta}{I_1 \sin^2 \theta}.
\end{equation}

To study the nutation, note that the energy can be expressed with Euler angles and $h:= \| a\|$ as:
\[
E = \frac{1}{2}I_1(\omega_{\phi}^2 \sin^2 \theta + \omega_{\theta}^2) +\frac{1}{2}\frac{p_{\psi}^2}{I_3} + Mgh \cos \theta.
\]
Thus in the deterministic case, the quantity
\[
E' = E - \frac{1}{2}\frac{p_{\psi}^2}{I_3} =  \frac{1}{2}I_1(\omega_{\phi}^2 \sin^2 \theta + \omega_{\theta}^2) + Mgh \cos \theta = \frac{1}{2} I_1 \omega_{\theta}^2 + \frac{(p_{\phi} - p_{\psi}\cos \theta)^2}{2 I_1 \sin^2 \theta} + Mgh \cos \theta
\]
is also constant. Define the ``effective potential" $V$ to be 
\[
V(\theta, p_{\psi}, p_{\phi} ) := \frac{(p_{\phi} - p_{\psi}\cos \theta)^2}{2 I_1 \sin^2 \theta} + Mgh \cos \theta
\]
a function of $\theta$, $p_{\psi}$ and $p_{\phi} $. The evolution of $\theta$ can be qualitatively studied with the energy method, knowing that 
\[
E' = \frac{1}{2} I_1 \omega_{\theta}^2  + V(\theta, p_{\psi}, p_{\phi})
\]
is constant. 

The potential $V$ is a convex function of $\theta$ in $(0,\Pi)$. Thus the nutation angle $\theta$ oscillates between the two roots $\theta_1, \theta_2$ of $V(\theta) = E'$. $V(\theta)$ takes its minima at $\theta_0$. The graph of $V(\theta)$ is illustrated by the figure \ref{VthetaD}.

\begin{figure}[H]
    \centering
    \includegraphics[width=0.6\textwidth]{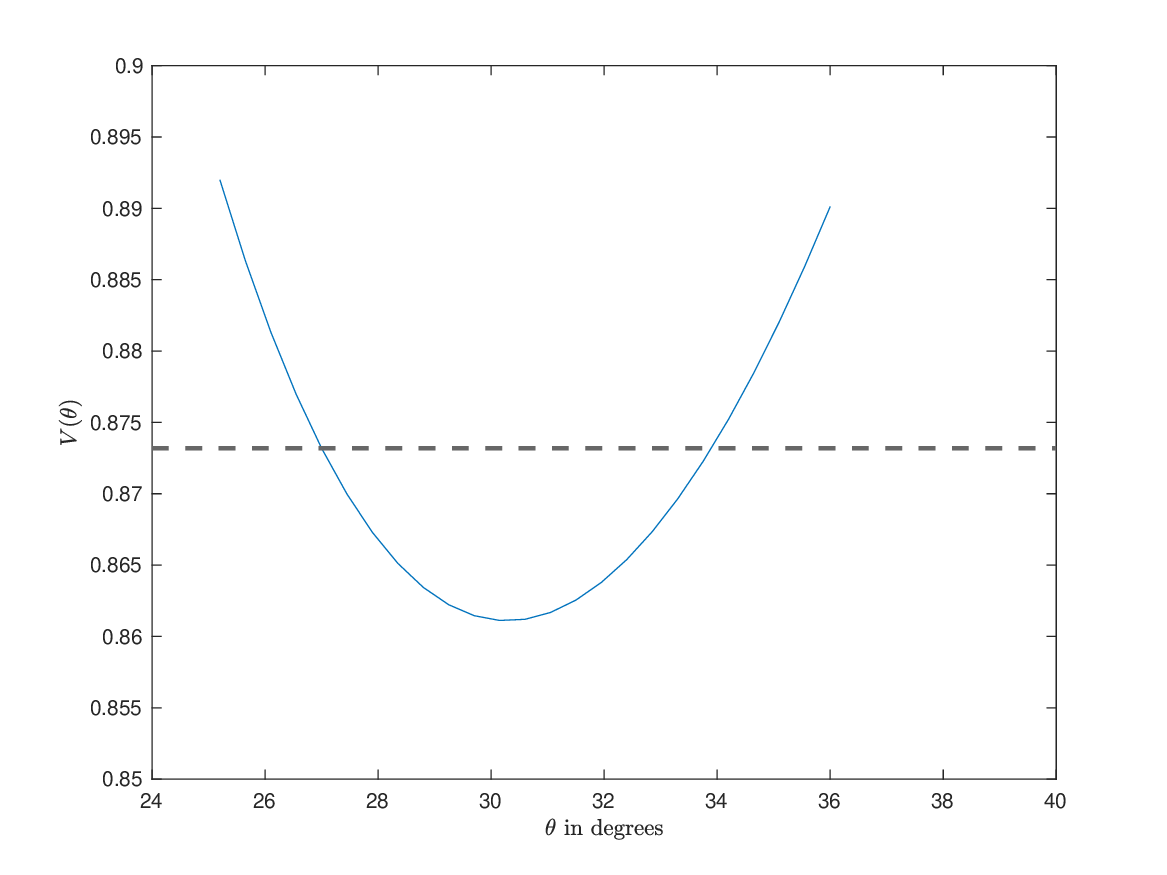}
    \caption{$V(\theta)$ as function of $\theta$, the dashed line marks the level of $E'$, with parameters of the deterministic case discussed below.}
    \label{VthetaD}
\end{figure}

Following the discussions in the previous sections, we have witnessed that as a result of Noether's theorem, the quantity $p_{\phi}$ is preserved by the stochastic mid-point integrators with all kinds of stochatic Hamiltonians, while $p_{\psi}$ and energy $E$ are not. Each of the three quantities plays an important role in the dynamics of gyroscopic precession from the analysis above. In the following parts, we will see how the integrators with different stochastic Hamiltonians that preserve (or not) $p_{\psi}$ and $E$ will behave in the simulation.

\paragraph{Deterministic case.} We consider the case $\mathbb{I} = {\rm diag}(I_1 = 0.1,I_1 = 0.1,I_3 = 1)$, $a = (0,0,1)$, $m = 0,1$ and $g = 9.8$. Starting from an inclined position, with the body $z$-axis inclined towards the spatial $y$-axis at an angle $\theta_0 = 0.15 \pi$, thus $R_0 = \begin{bmatrix}
1 & 0 & 0 \\
0 & \cos\theta_0 & \sin\theta_0 \\
0 &-\sin\theta_0 & \cos\theta_0
\end{bmatrix}$, and with the body angular momentum $\Pi_0 = (0,0,1)$. The figures below show the evolution of $\Pi$ and the spatial position of the point $(0,0,1)$ in the body coordinates, showing the pattern of precession and nutation.

\begin{figure}[H]
    \centering
    
    \begin{subfigure}{0.48\textwidth}
        \centering
        \includegraphics[width=\linewidth]{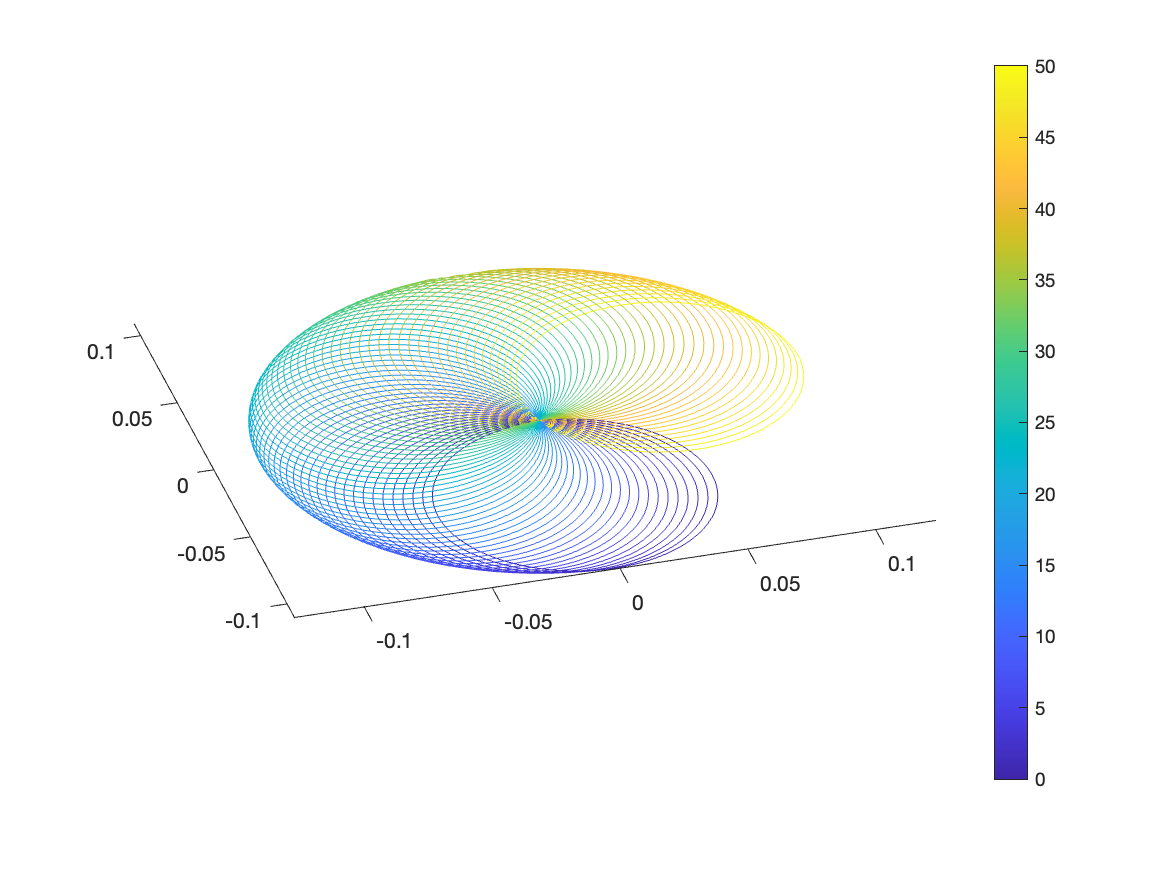}
        \caption{$\Pi$}
    \end{subfigure}
    \hfill
    \begin{subfigure}{0.48\textwidth}
        \centering
        \includegraphics[width=\linewidth]{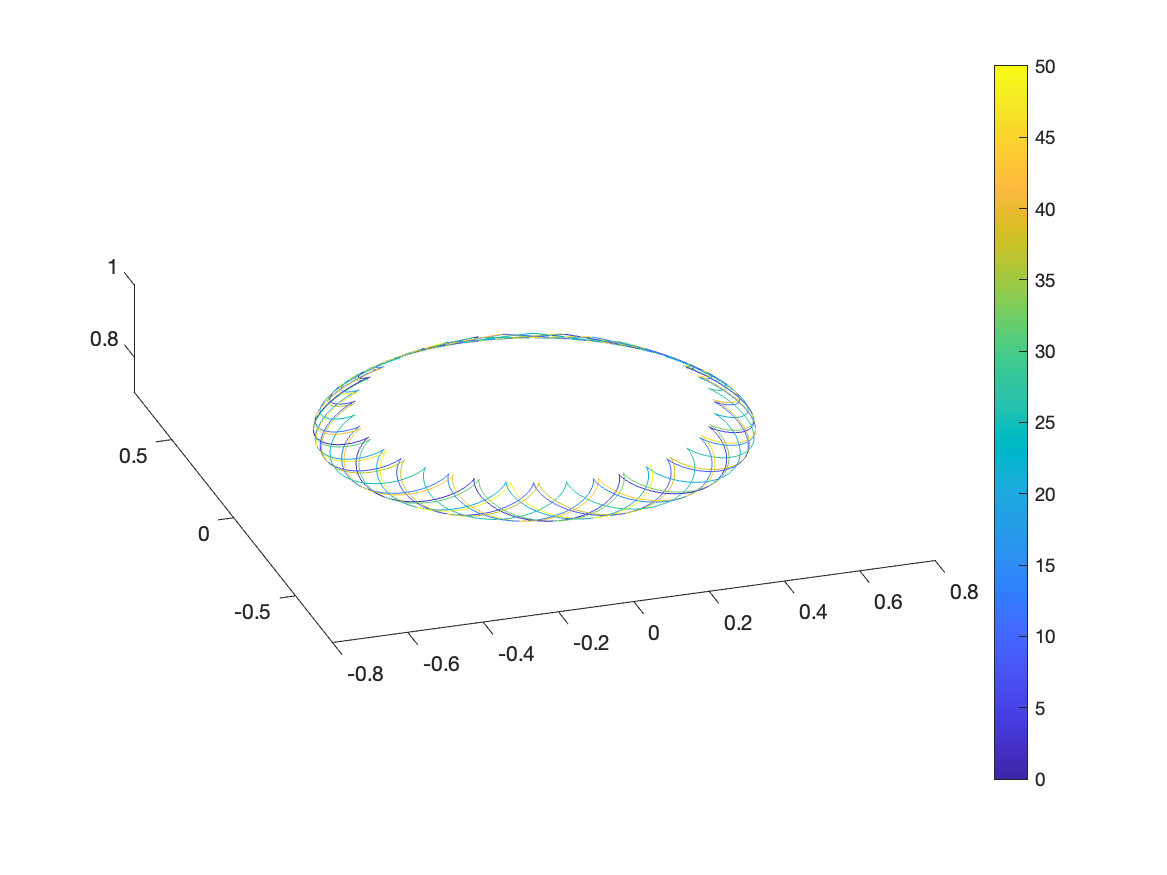}
        \caption{Position of $(0,0,1)$ in the body coordinate}
    \end{subfigure}
    
    \caption{Deterministic}
\end{figure}

In the deterministic case, $p_\psi$ is conserved up to $E-13$, as a result, the value $\theta_0$, at which $V(\theta)$ takes its minima, is also preserved .The energy $E$ is conserved up to $E-6$.

\begin{figure}[H]
    \centering
    
    \begin{subfigure}{0.48\textwidth}
        \centering
        \includegraphics[width=\linewidth]{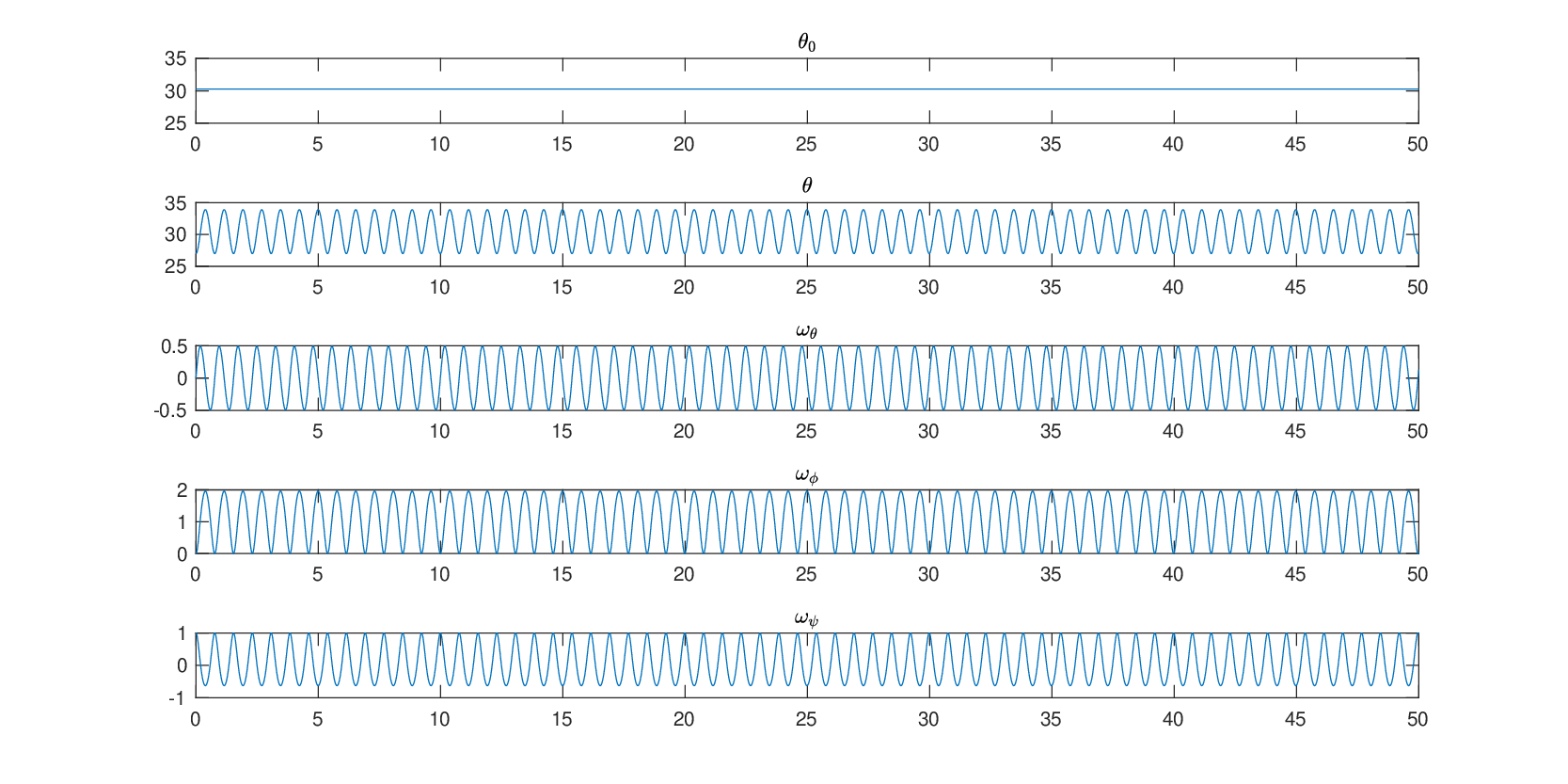}
    \end{subfigure}
    \hfill
    \begin{subfigure}{0.48\textwidth}
        \centering
        \includegraphics[width=\linewidth]{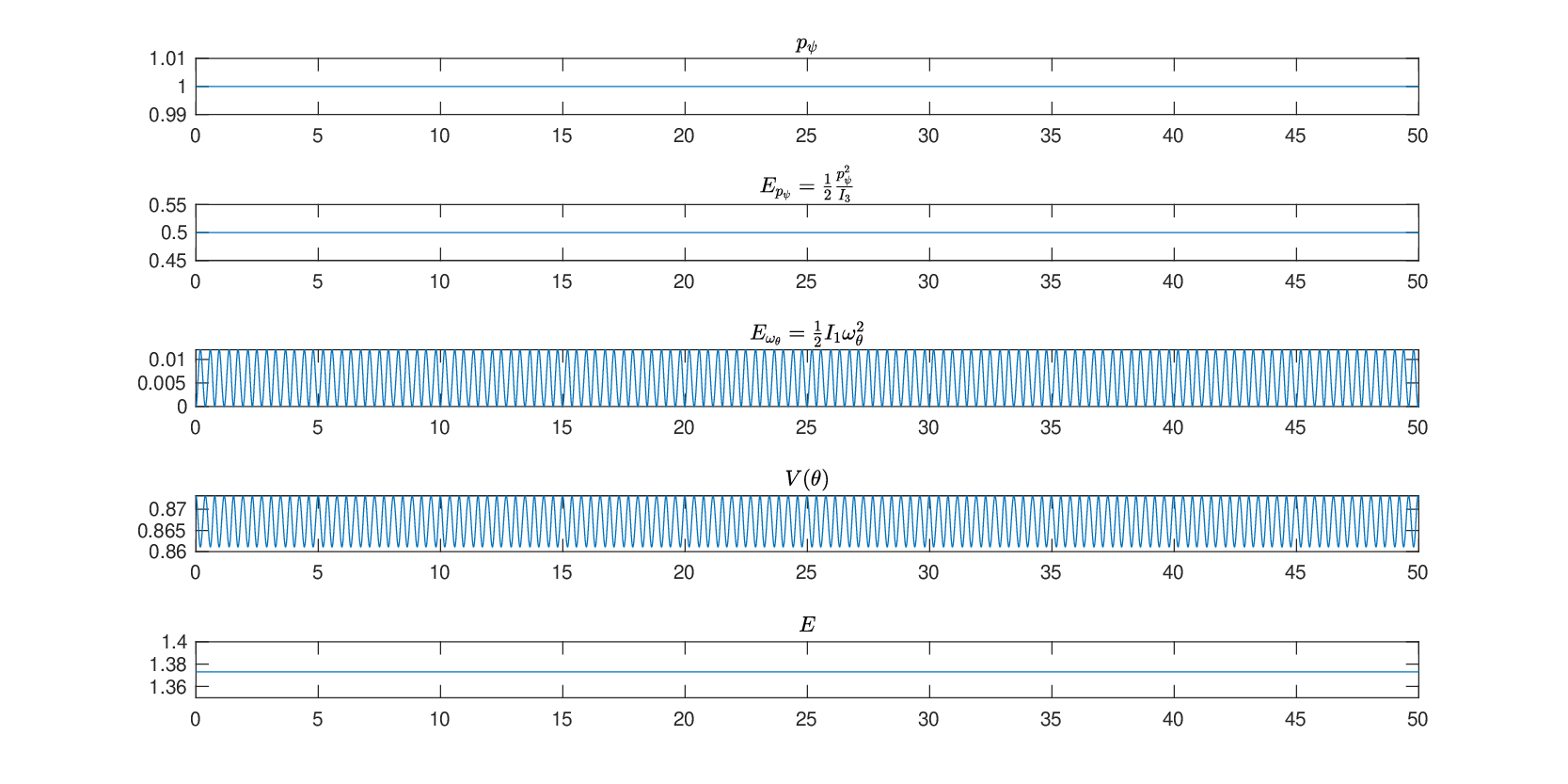}
    \end{subfigure}    
    \caption{Deterministic}
\end{figure}

\paragraph{Stochastic case 1: $h_1(\Pi,\Gamma) = k \Pi_z $.} We run the simulation with the same setting except that there is the stochastic Hamiltonian $h_1(\Pi,\Gamma) = k \Pi_z $ with $k = 0.1$, a relatively large level of stochasticity, which renders the path of $\Pi$ a distinguishable stochastic pattern. 

Nevertheless, the specific choice of the stochastic Hamiltonian also induces the preservation of $p_\psi = \Pi_z$ and of energy $E$ in the continuous case, as can be shown with the Poisson bracket calculation: $\frac{\dd p_\psi}{\dd t} = \{p_\psi ,h_1\} = 0$, $\frac{\dd E}{\dd t} = \{E ,h_1\} = 0$ .  With the mid-point integrator, $p_\psi$ is conserved up to $E-10$ and the energy $E$ is conserved up to $E-5$. This results in the very little disturbed evolution of all the Euler angle parameters compared to the deterministic case. The stochastic feature of the precession-nutation pattern, is hardly distinguishable.

\begin{figure}[H]
    \centering
    
    \begin{subfigure}{0.48\textwidth}
        \centering
        \includegraphics[width=\linewidth]{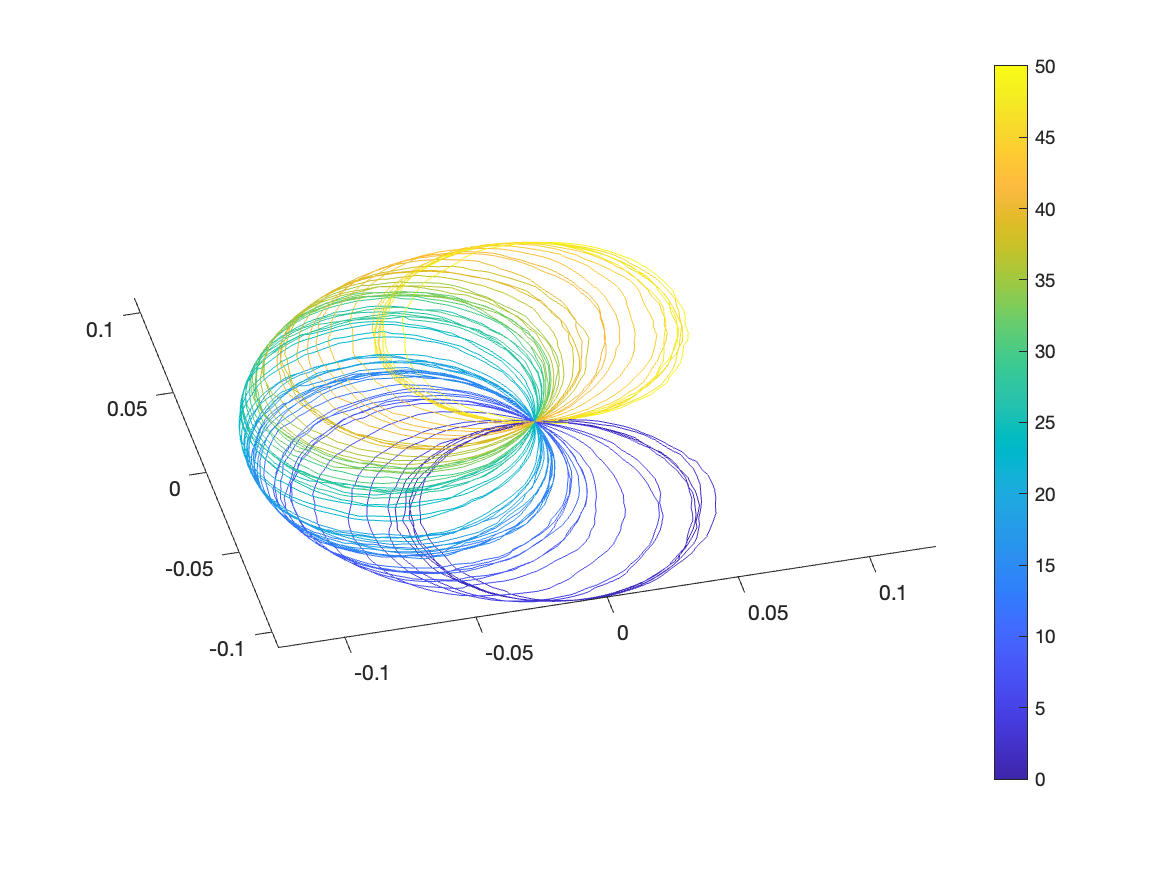}
        \caption{$\Pi$}
    \end{subfigure}
    \hfill
    \begin{subfigure}{0.48\textwidth}
        \centering
        \includegraphics[width=\linewidth]{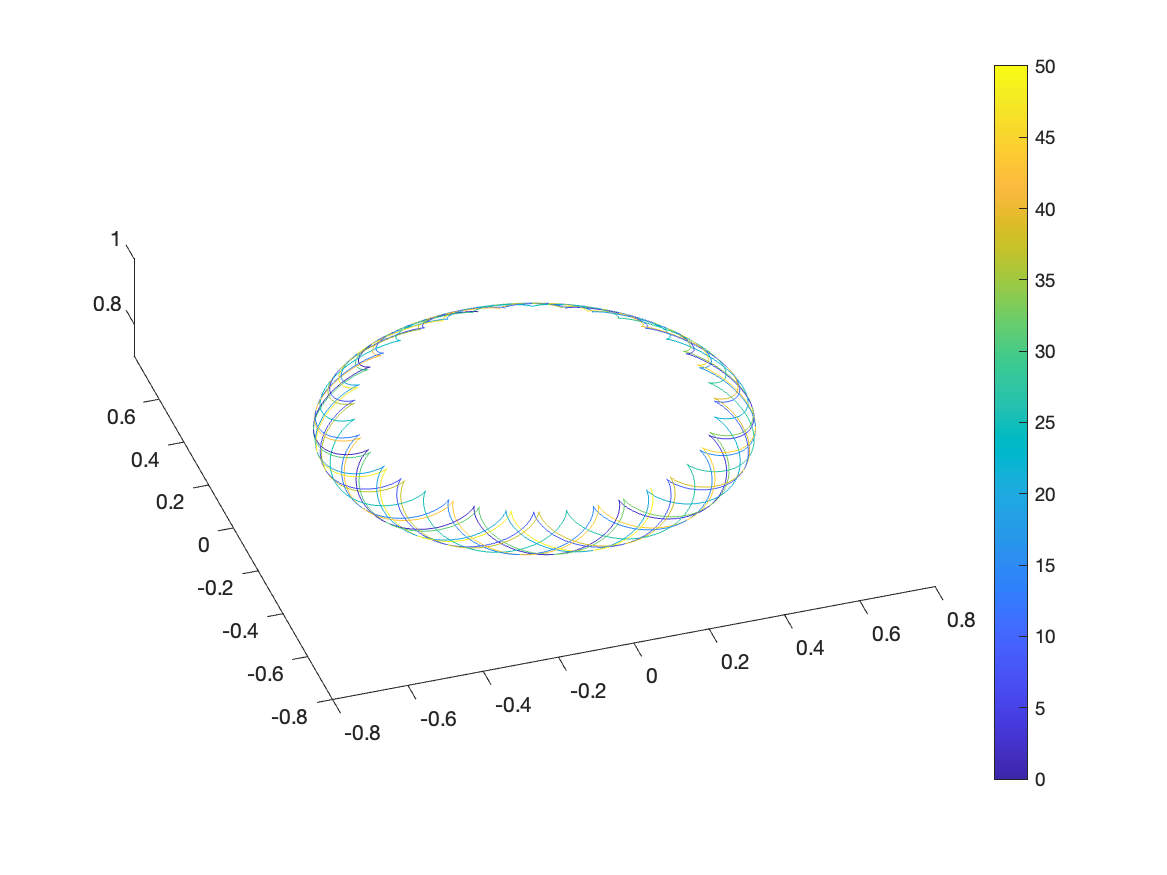}
        \caption{Position of $(0,0,1)$ in the body coordinate}
    \end{subfigure}
    
    \caption{$h(\Pi,\Gamma) = k \Pi_z $}
\end{figure}

\begin{figure}[H]
    \centering
    
    \begin{subfigure}{0.48\textwidth}
        \centering
        \includegraphics[width=\linewidth]{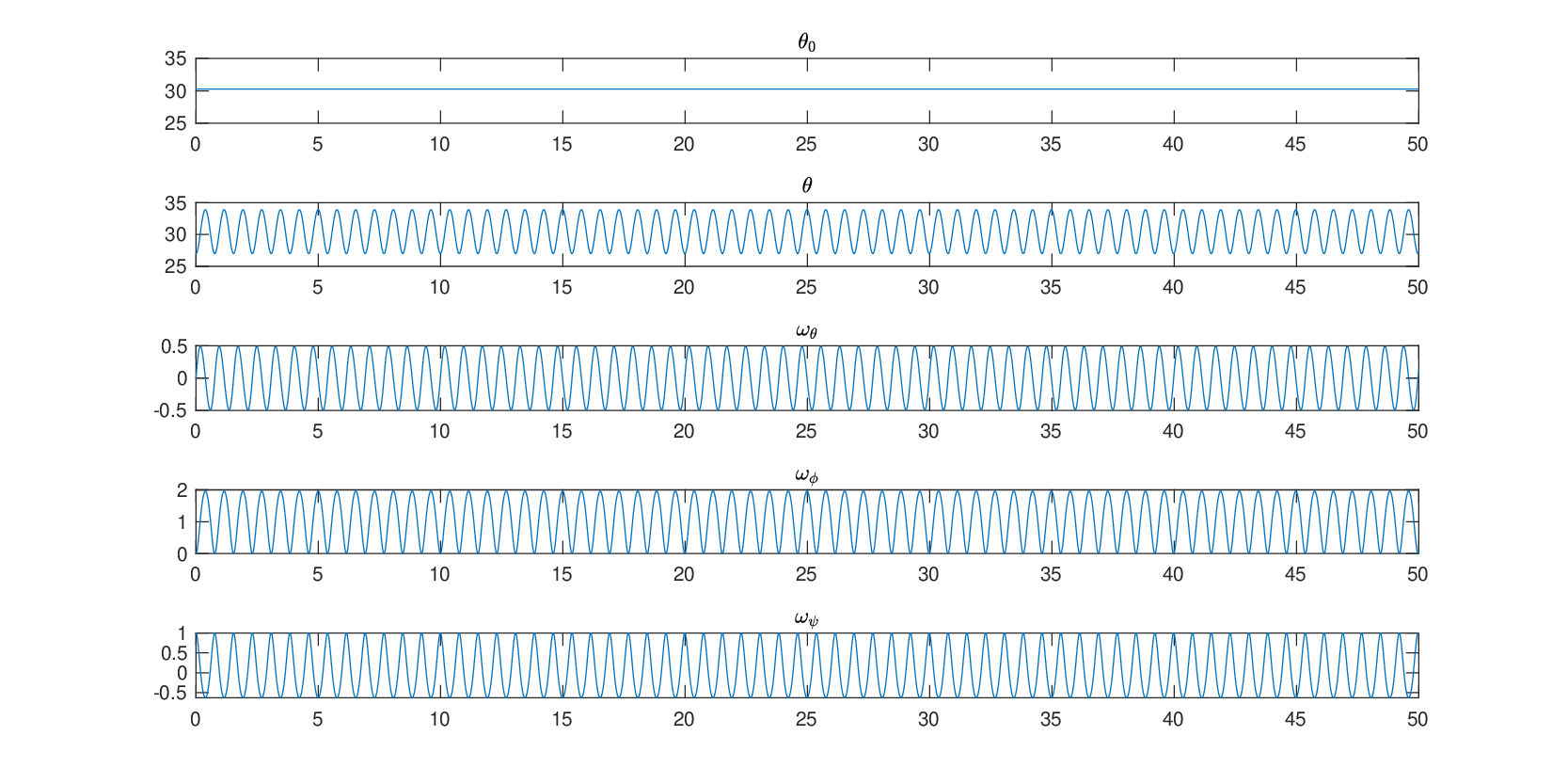}
    \end{subfigure}
    \hfill
    \begin{subfigure}{0.48\textwidth}
        \centering
        \includegraphics[width=\linewidth]{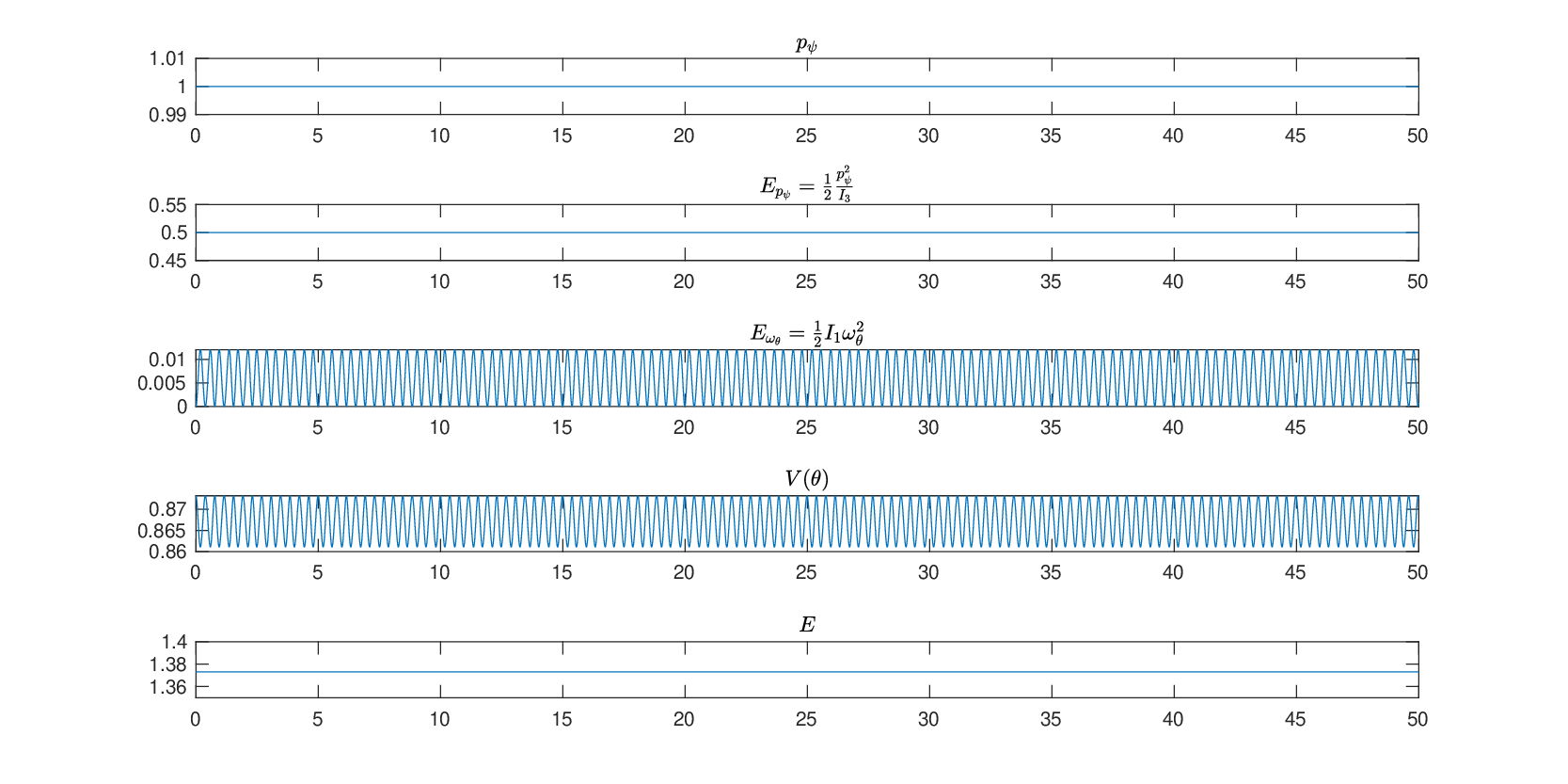}
    \end{subfigure}    
    \caption{$h(\Pi,\Gamma) = k \Pi_z $}
\end{figure}

\paragraph{Stochastic case 2: $h_1(\Pi,\Gamma) = mgk \Gamma_z $ or $h_1(\Pi,\Gamma) = k \| \Pi \|^2 $ .} We run the simulation with the same setting as in the deterministic case except that there is the stochastic Hamiltonian $h_1(\Pi,\Gamma) = mg \Gamma_z $. 

With these choices of stochastic Hamiltonian, $p_\psi$ is still conserved in the continuous case, as is shown by $\frac{\dd p_\psi}{\dd t} = \{p_\psi ,h_1\} = 0$, but energy $E$ is not. With the mid-point integrator, $p_\psi$ is conserved up to $E-13$.. As a result, $\theta_0$ is constant but the magnitude of nutation is variable, and so is the velocity of precession and of spin. This can be illustrated again with the figure of $V(\theta,p_{\psi},p_{\phi})$ as function of $\theta$, showing that with the fluctuation of $E'$, the upper and lower bound of the nutation angle $\theta$, also fluctuates. 

The stochastic Hamiltonian $h_1(\Pi,\Gamma) = \| \Pi \|^2 $, reproduces a similar pattern. 

\begin{figure}[H]
    \centering
    
    \begin{subfigure}{0.48\textwidth}
        \centering
        \includegraphics[width=\linewidth]{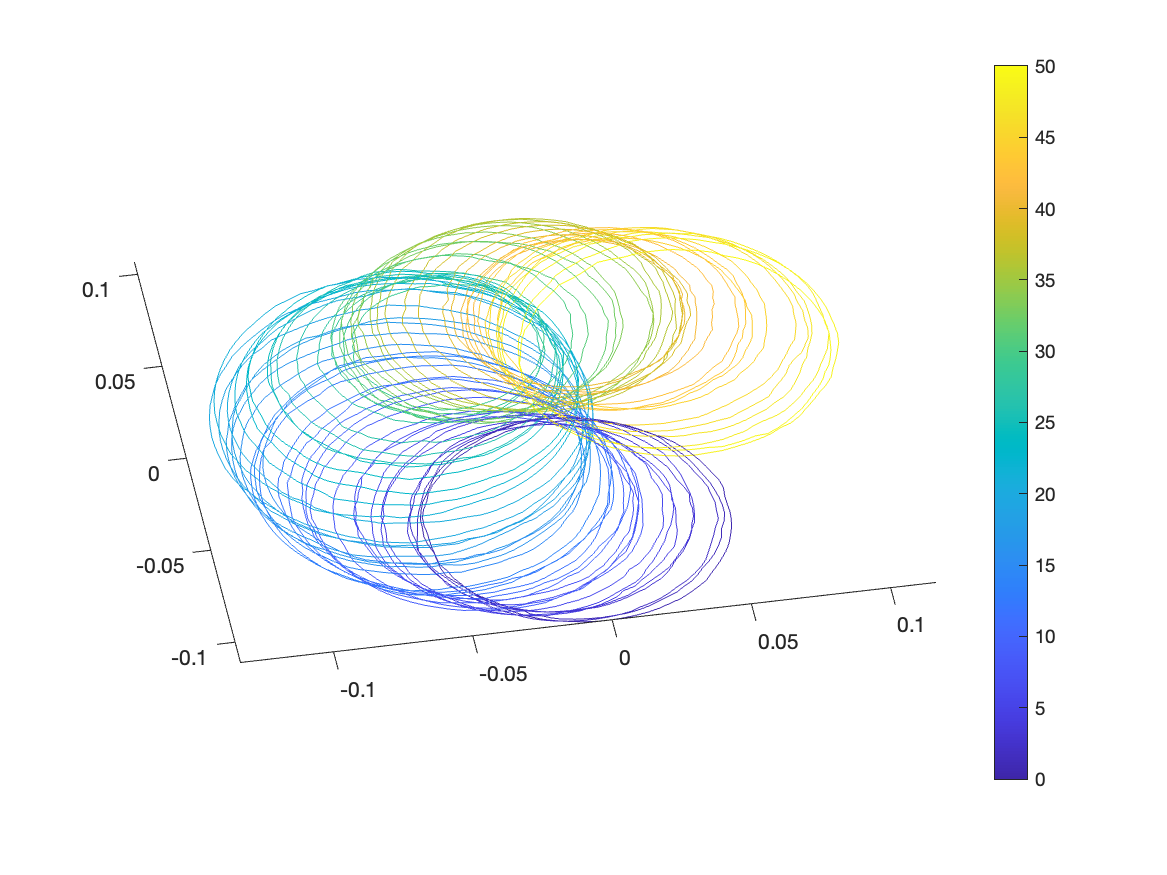}
        \caption{$\Pi$}
    \end{subfigure}
    \hfill
    \begin{subfigure}{0.48\textwidth}
        \centering
        \includegraphics[width=\linewidth]{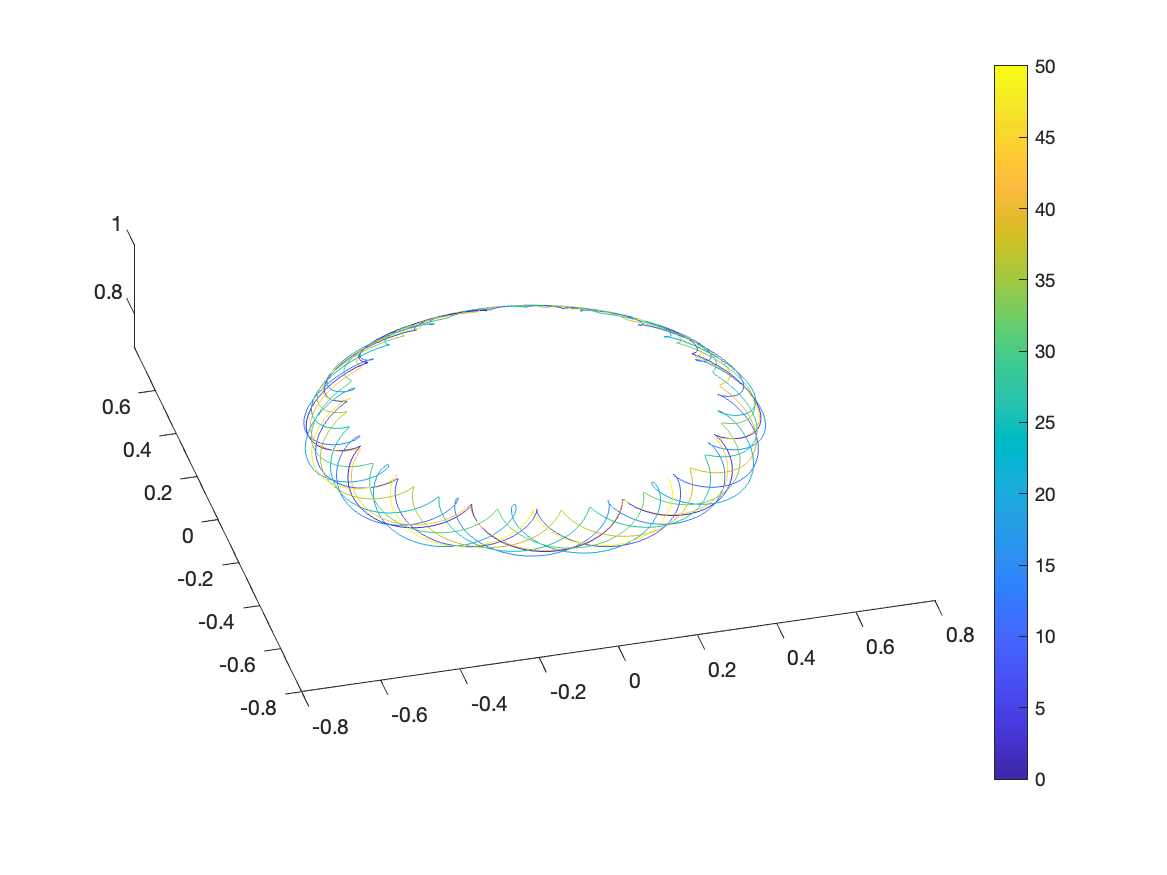}
        \caption{Position of $(0,0,1)$ in the body coordinate}
    \end{subfigure}
    
    \caption{$h_1(\Pi,\Gamma) = mg \Gamma_z $}
\end{figure}

\begin{figure}[H]
    \centering
    
    \begin{subfigure}{0.48\textwidth}
        \centering
        \includegraphics[width=\linewidth]{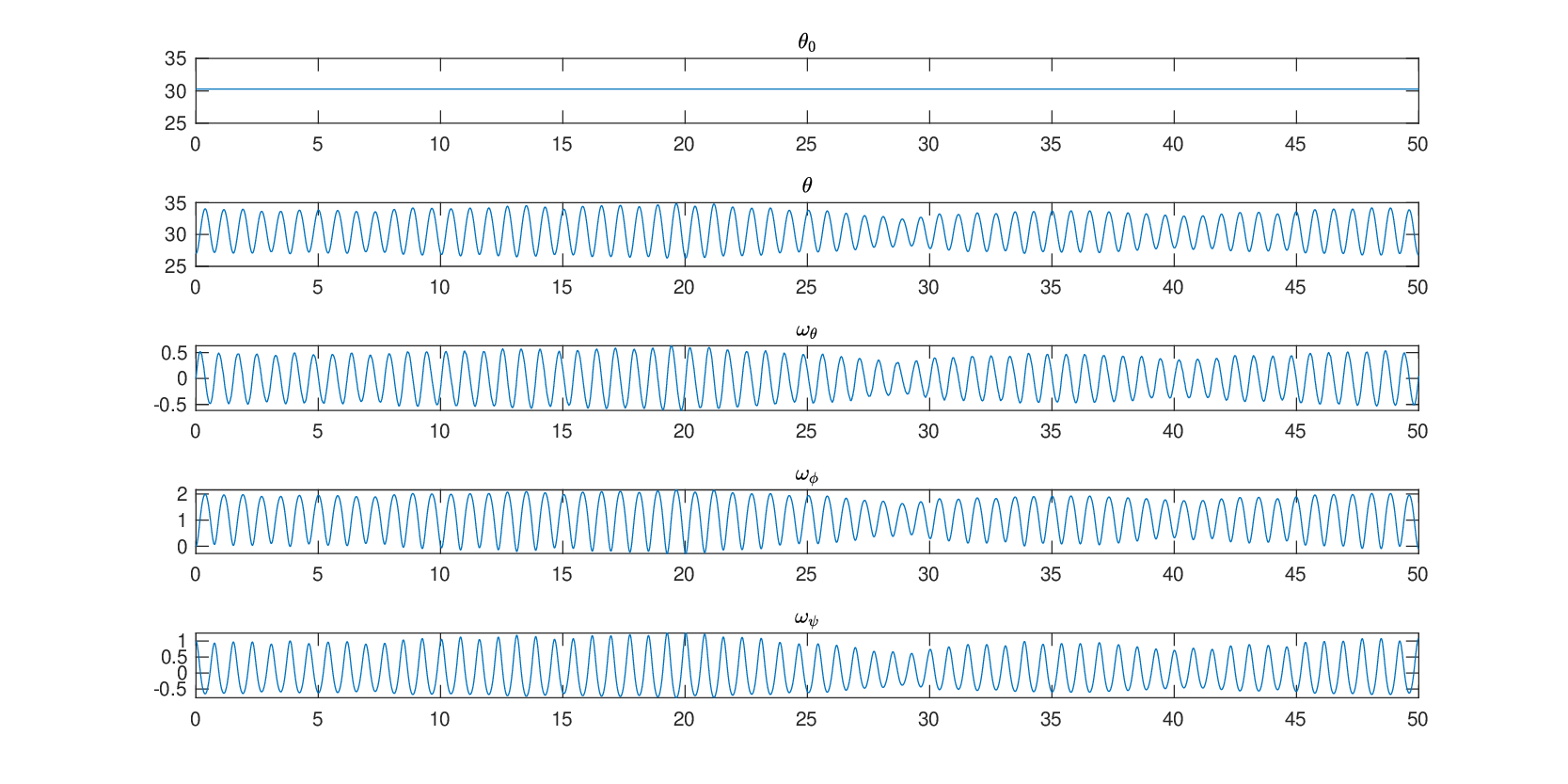}
    \end{subfigure}
    \hfill
    \begin{subfigure}{0.48\textwidth}
        \centering
        \includegraphics[width=\linewidth]{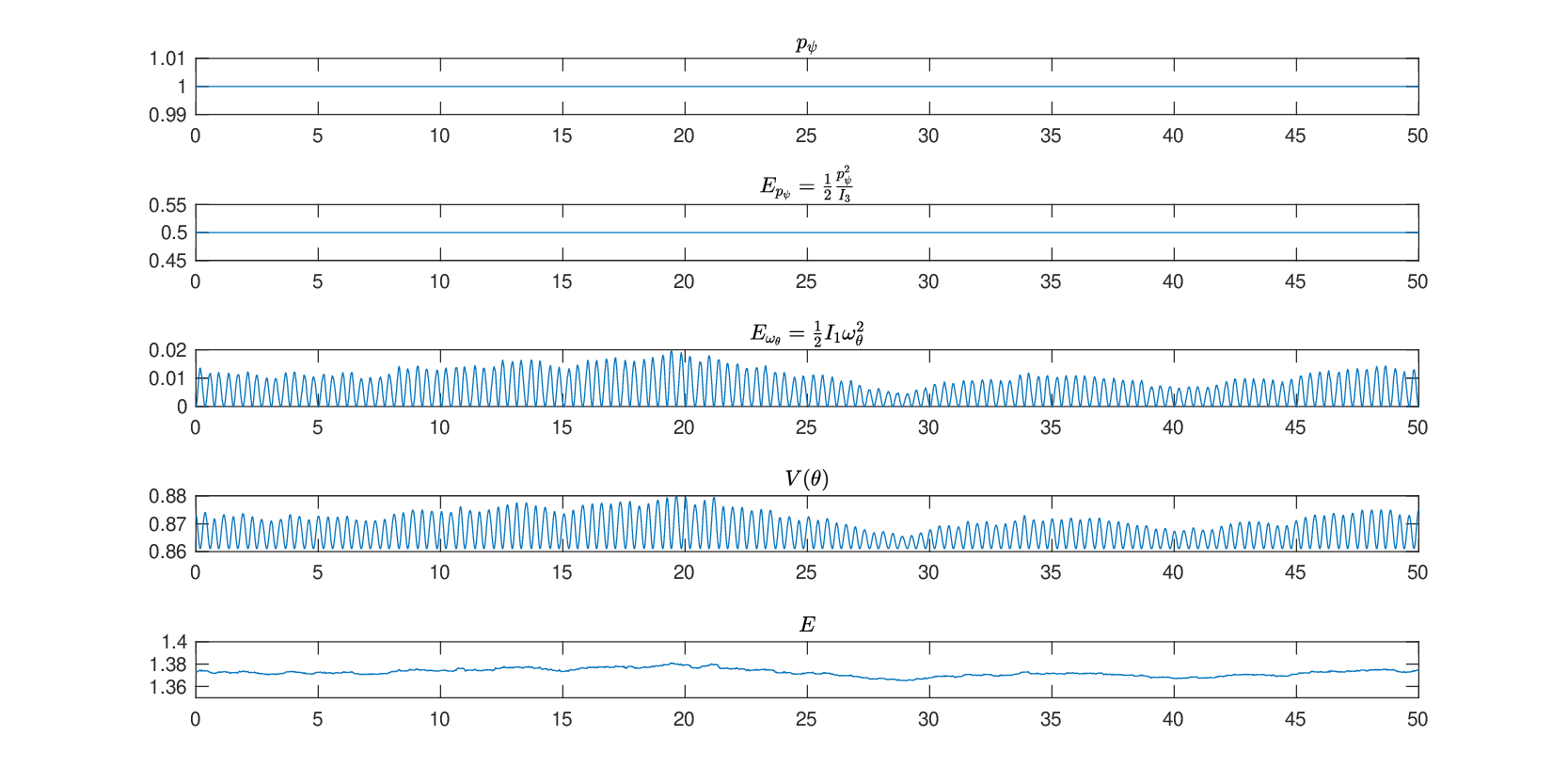}
    \end{subfigure}    
    \caption{$h_1(\Pi,\Gamma) = mg \Gamma_z $}
\end{figure}

\begin{figure}[H]
    \centering
    \includegraphics[width=0.6\textwidth]{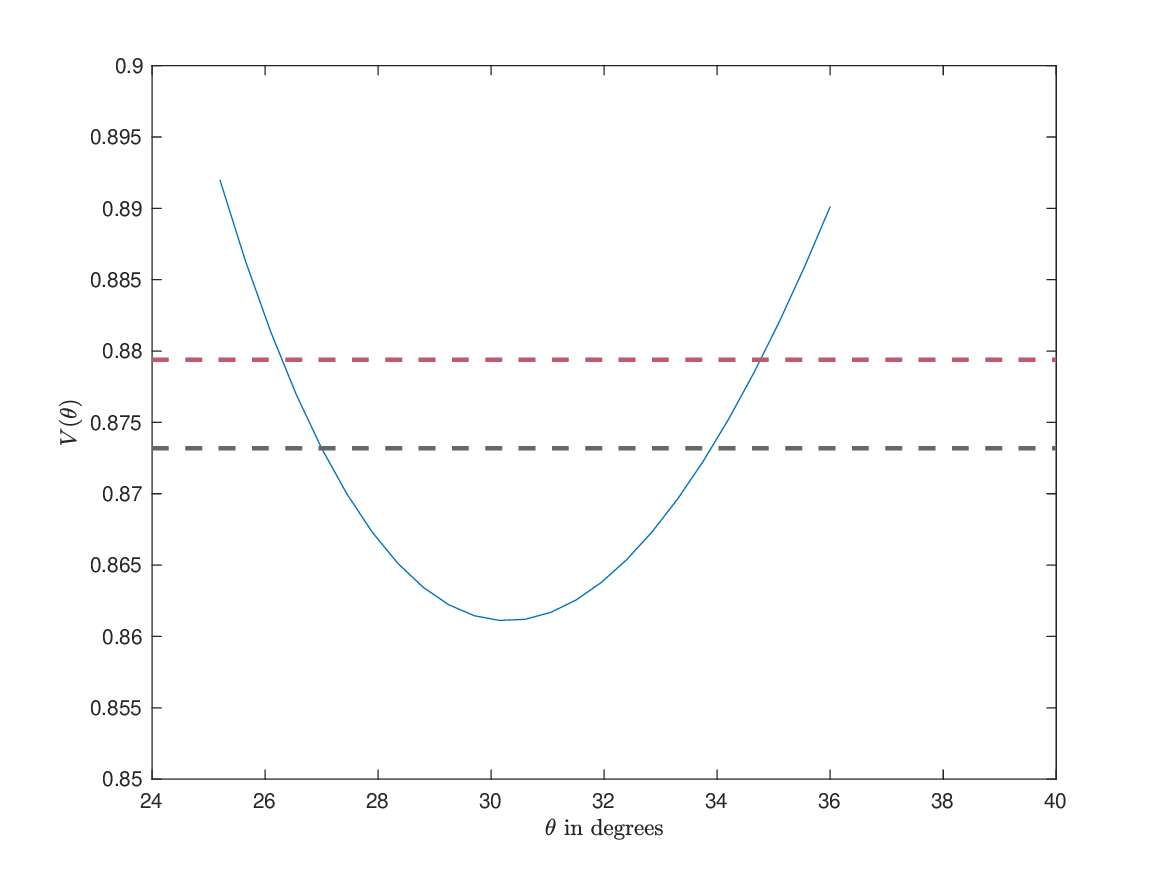}
    \caption{$V(\theta)$ as function of $\theta$, the dashed line marks the level of $E'$ at $t=0$ and at $t=20$, with parameters of the deterministic case discussed below.}
    \label{VthetaGa_z}
\end{figure}

\paragraph{Stochastic case 3: $N=2$, $h_1(\Pi,\Gamma) = mgk_1 \Gamma_x $, $h_2(\Pi,\Gamma) = mgk_2 \Gamma_y$ or $N=2$, $h_1(\Pi,\Gamma) = k_1 \Pi_x $, $h_2(\Pi,\Gamma)= k_2 \Pi_y $.} We run the simulation with the same setting as in the deterministic case except that there are the stochastic Hamiltonian $h_1(\Pi,\Gamma) = mg \Gamma_x $, $h_2(\Pi,\Gamma) = mg\Gamma_y$ or $N=2$. 

In this case, neither $p_\psi$ nor the energy $E$ is conserved. As a result, $\theta_0$ also fluctuates. This time, the shape of $V(\theta,p_{\psi},p_{\phi})$ as function of $\theta$ also changes with time as $p_{\psi}$ changes.

\begin{figure}[H]
    \centering
    
    \begin{subfigure}{0.48\textwidth}
        \centering
        \includegraphics[width=\linewidth]{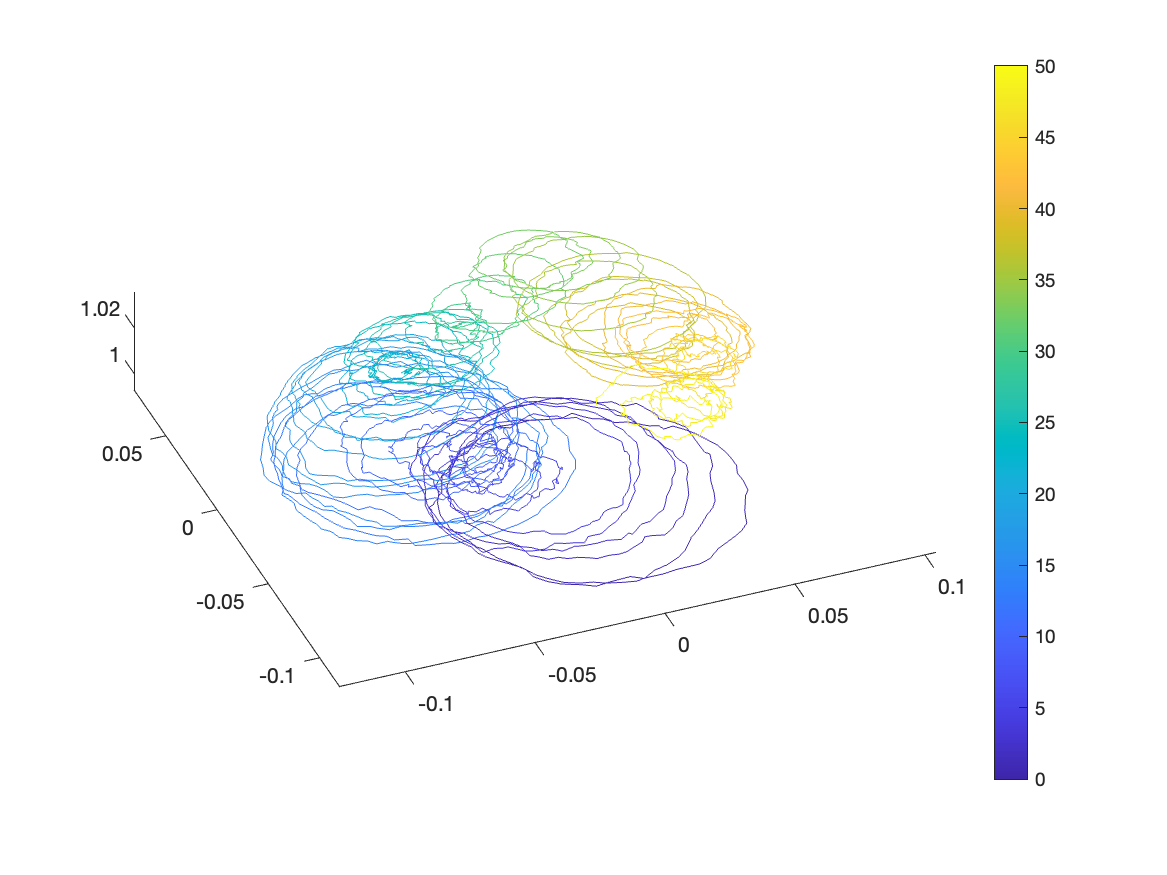}
        \caption{$\Pi$}
    \end{subfigure}
    \hfill
    \begin{subfigure}{0.48\textwidth}
        \centering
        \includegraphics[width=\linewidth]{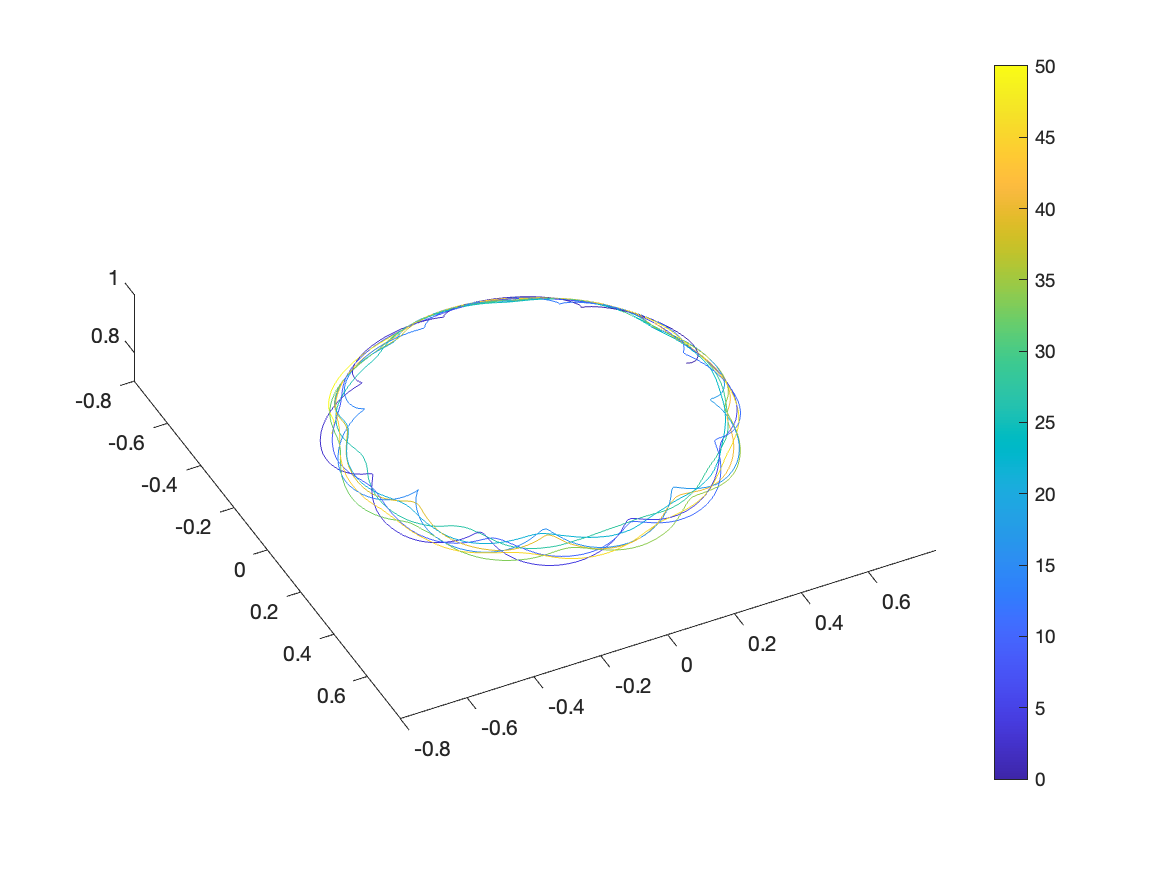}
        \caption{Position of $(0,0,1)$ in the body coordinate}
    \end{subfigure}
    
    \caption{$h_1(\Pi,\Gamma) = mg \Gamma_x $, $h_2(\Pi,\Gamma) = mg\Gamma_y$}
\end{figure}

\begin{figure}[H]
    \centering
    
    \begin{subfigure}{0.48\textwidth}
        \centering
        \includegraphics[width=\linewidth]{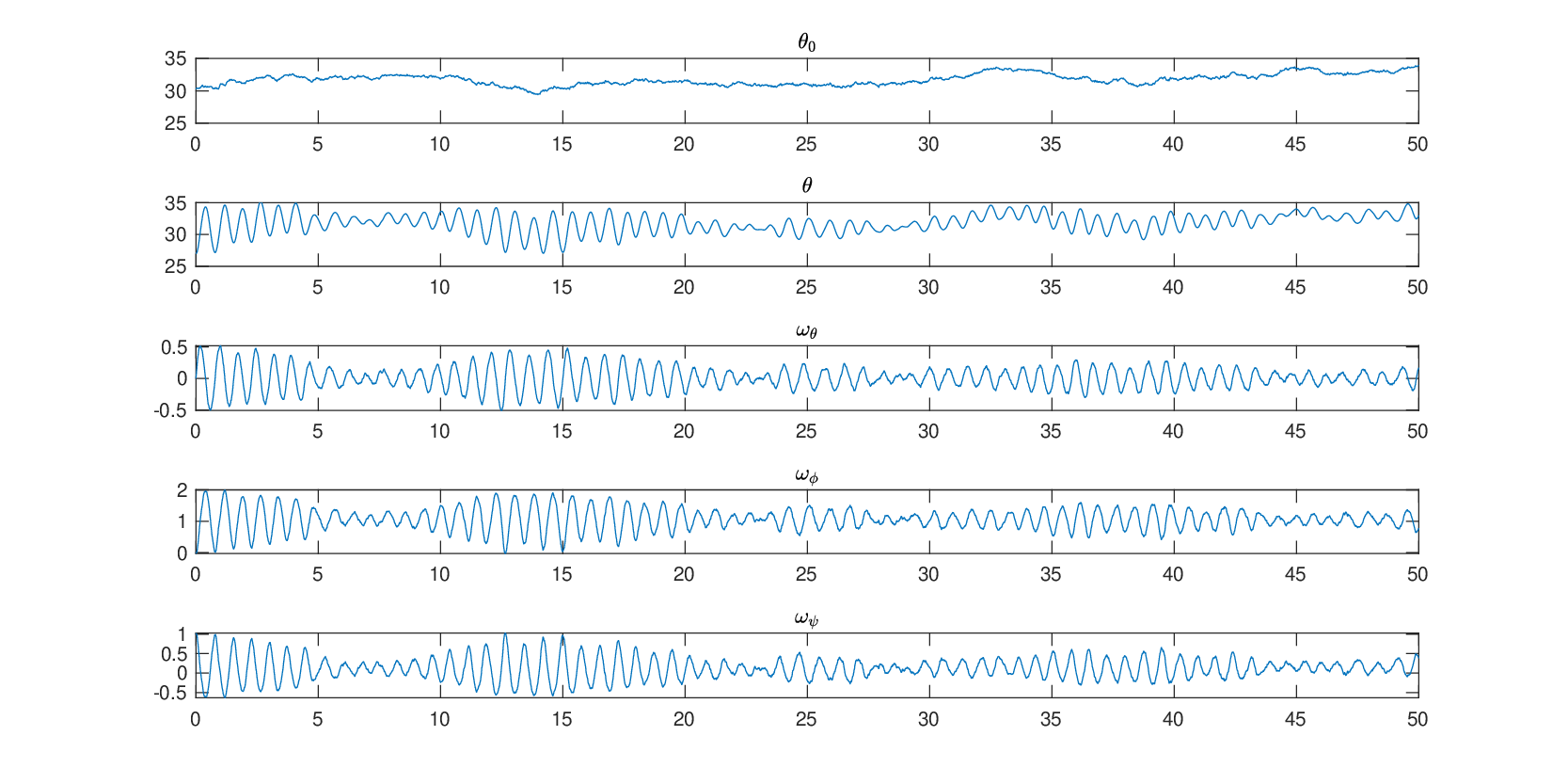}
    \end{subfigure}
    \hfill
    \begin{subfigure}{0.48\textwidth}
        \centering
        \includegraphics[width=\linewidth]{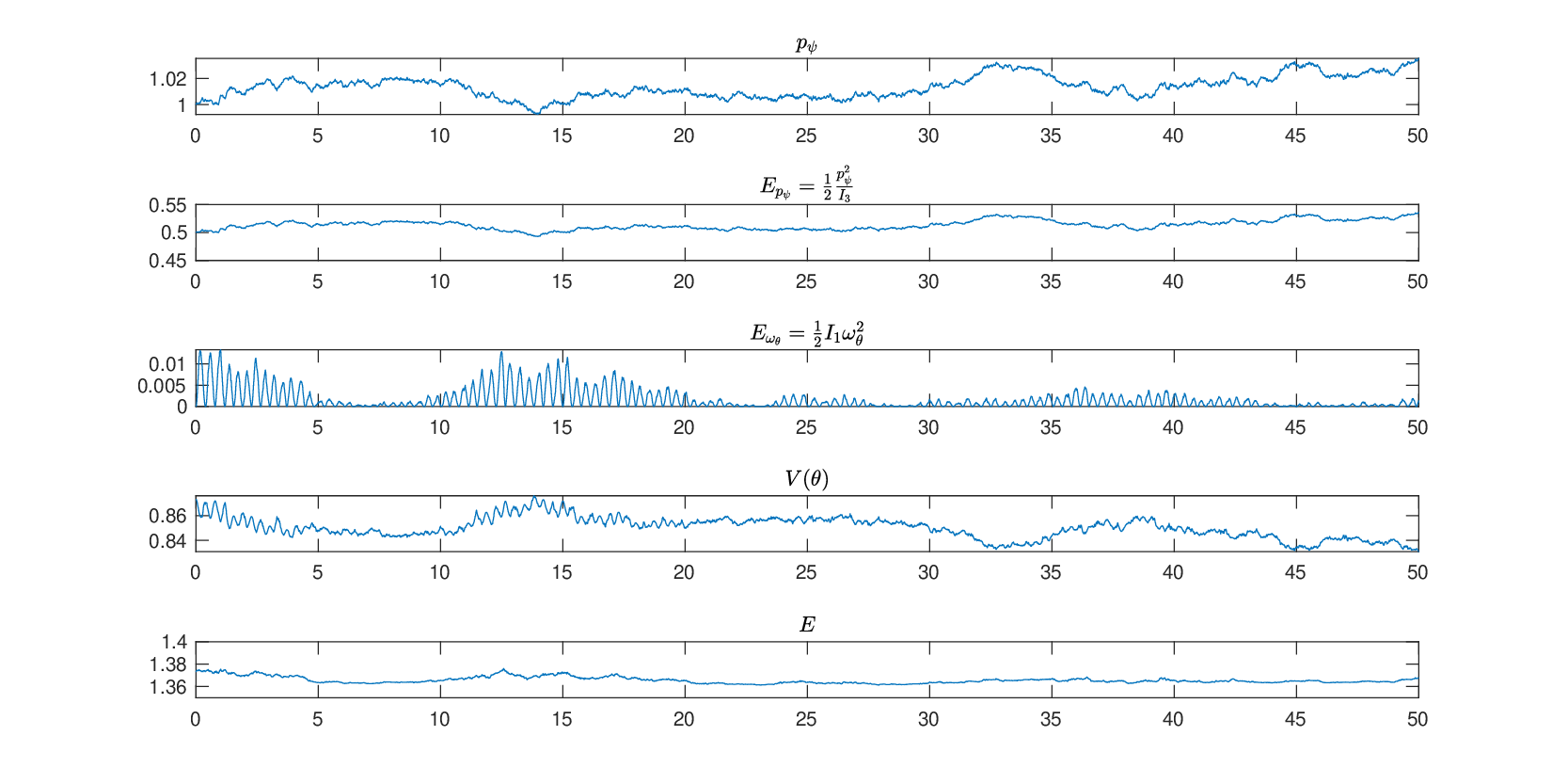}
    \end{subfigure}    
    \caption{$h_1(\Pi,\Gamma) = mg \Gamma_x $, $h_2(\Pi,\Gamma) = mg\Gamma_y$}
\end{figure}

\begin{figure}[H]
    \centering
    \includegraphics[width=0.6\textwidth]{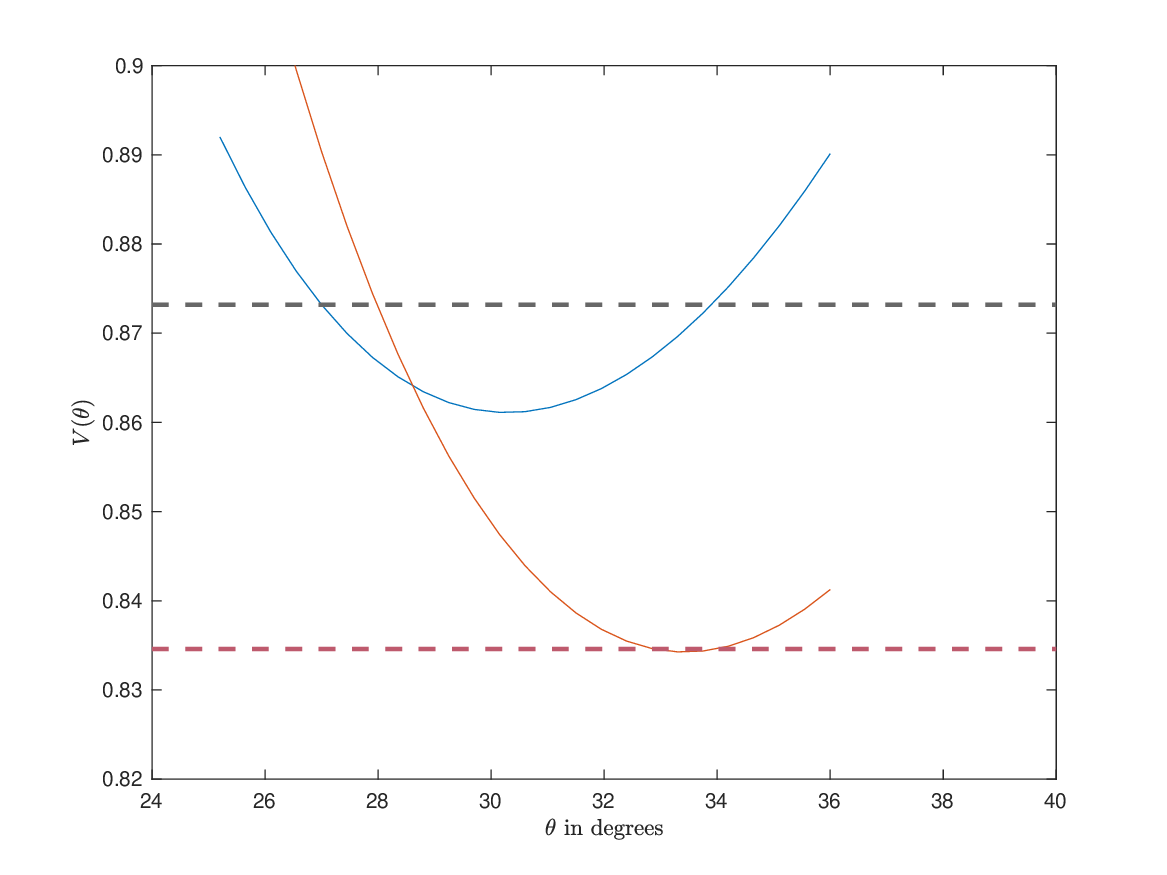}
    \caption{$V(\theta)$ as function of $\theta$, $t=0$ and at $t=45$, the dashed line marks the level of $E'$ at $t=0$ and at $t=45$, with parameters of the deterministic case discussed below.}
    \label{VthetaGa_xy}
\end{figure}

The case $h_1(\Pi,\Gamma) = k_1 \Pi_x $, $h_2(\Pi,\Gamma)= k_2 \Pi_y $ with $k_1 = k_2 = 0.01$ produces a similar scenario.

\section{Conclusion}

In this paper, we have developed a stochastic variational integrator for Hamiltonian systems on Lie groups, which preserves key geometric structures such as symplecticity, Poisson structures, and Noether symmetries. We began by providing a variational derivation of the stochastic midpoint method for canonical Hamiltonian systems on vector spaces, highlighting its natural extension to systems on Lie groups. We then extended this method to general stochastic Hamiltonian systems on Lie groups, proving the symplecticity of the resulting scheme. We also examined the case where the Hamiltonians are invariant under a Lie subgroup and demonstrated that a discrete version of Noether’s theorem holds. When the symmetry group is the full Lie group, we showed that the scheme preserves the Lie-Poisson structure, coadjoint orbits, the Kirillov-Kostant-Souriau symplectic form, and the Casimir functions. Additionally, we considered systems where the symmetry subgroup is the isotropy group of a representation, which has important applications in areas like heavy top dynamics and compressible fluid models. The properties derived in the Lie group case were shown to hold in this more general setting, expressed in terms of a semidirect product Lie group framework. We also provided a full convergence proof for the case of $G
=SO(3)$, applied to the free rigid body model. Finally, we illustrated the practical implications of our results with examples involving the free rigid body and the heavy top, including the stochastic modeling of gyroscopic precession. Overall, the results presented here offer a unified framework for structure-preserving discretization of stochastic Hamiltonian systems, with applications to a variety of dynamical systems.
As a future work we intend to apply these results to stochastic geometric fluid models \cite{Ho2015,GBHo2018,GBHo2020} by exploiting recent progress in variational discretization in fluid dynamics \cite{GaGB2020,GaGB2021}.

%
%
%




\section{Appendix}

\subsection{Proof of the energy conservation for the deterministic midpoint method \eqref{integrator_rigid_body} for the rigid body} \label{proof_energy}

The kinetic energy difference between two time steps can be expressed as : 
\begin{align*}
E(\Pi_k)- E(\Pi_{k-1}) &= \frac{1}{2} \mathbb{I}^{-1}(\Pi_k - \Pi_{k-1}) \cdot \Pi_k + \frac{1}{2} \mathbb{I}^{-1}\Pi_{k-1} \cdot (\Pi_k - \Pi_{k-1}) \\
&= \frac{1}{2} (\Pi_k - \Pi_{k-1}) \cdot (\mathbb{I}^{-1})^T \Pi_k + \frac{1}{2} \mathbb{I}^{-1}\Pi_{k-1} \cdot (\Pi_k - \Pi_{k-1}) \\
&= \frac{1}{2} (\Pi_k - \Pi_{k-1}) \cdot \mathbb{I}^{-1} (\Pi_k + \Pi_{k-1}).
\end{align*}
In the last equality, we used the fact that the body moment of inertia $\mathbb{I}$ is a symmetrical matrix, thus $\mathbb{I}^{-1} = (\mathbb{I}^{-1})^\mathsf{T}$.

Using the expressions of $\Pi_k$ and $\Pi_{k-1}$ in terms of $\tilde \Pi_k^{1}$ and $\tilde \Pi_k^{2}$ given by \eqref{def_Pi}, and the first equation of the scheme \eqref{integrator_rigid_body}, we have that:
\[
\Pi_k - \Pi_{k-1} = - \Delta t \xi_k \times (\tilde \Pi_k^{1}+\tilde \Pi_k^{2}),
\]
and that
\[
 \Pi_k + \Pi_{k-1} = \tilde \Pi_k^{1}+\tilde \Pi_k^{2} + \frac{\Delta t}{2} \xi_k \times (\tilde \Pi_k^{1} -\tilde \Pi_k^{2}) - \frac{\Delta t^2}{4} (\xi_k \cdot (\tilde \Pi_k^{1}+\tilde \Pi_k^{2})) \xi_k. 
\]  
Taking one more time the first equation of \eqref{integrator_rigid_body}: $\tilde \Pi_k^{1} -\tilde \Pi_k^{2} = \frac{\Delta t}{2} \xi_k \times (\tilde \Pi_k^{1}+\tilde \Pi_k^{2}) +  \frac{\Delta t^2}{4} (\xi_k \cdot (\tilde \Pi_k^{1}-\tilde \Pi_k^{2})) \xi_k$, we also get that 
\begin{align*}
\frac{\Delta t}{2} \xi_k \times (\tilde \Pi_k^{1} -\tilde \Pi_k^{2}) &= \frac{\Delta t^2}{4} \xi_k \times \big(  \xi_k \times (\tilde \Pi_k^{1}+\tilde \Pi_k^{2}) \big) + \frac{\Delta t^3}{8} (\xi_k \cdot (\tilde \Pi_k^{1}-\tilde \Pi_k^{2}))  \xi_k \times \xi_k \\
&= \frac{\Delta t^2}{4} (\xi_k \cdot (\tilde \Pi_k^{1}+\tilde \Pi_k^{2})) \xi_k - \frac{\Delta t^2}{4} (\tilde \Pi_k^{1}+\tilde \Pi_k^{2}) \| \xi \|^2.
\end{align*}

Combining the equations above, we have finally: 
\begin{align*}
E(\Pi_k)- E(\Pi_{k-1}) &= -\frac{\Delta t}{2} \big( \xi_k \times (\tilde \Pi_k^{1}+\tilde \Pi_k^{2}) \big) \cdot  \left(1- \frac{\Delta t^2}{4} \| \xi \|^2\right)\mathbb{I}^{-1}(\tilde \Pi_k^{1}+\tilde \Pi_k^{2}) \\
&= \frac{\Delta t}{2} \left(1- \frac{\Delta t^2}{4} \| \xi \|^2\right) \left( \xi \times \mathbb{I}^{-1}(\tilde \Pi_k^{1}+\tilde \Pi_k^{2})\right) \cdot (\tilde \Pi_k^{1}+\tilde \Pi_k^{2}) \\
&=0,
\end{align*}
since in the deterministic case, $\xi_k = \frac{1}{2} \mathbb{I}^{-1}(\tilde \Pi_k^{1}+\tilde \Pi_k^{2})/2$.

\subsection{Proof of the convergence (Theorem \ref{convergence_theorem})}

Within each time step $t_k \leqslant t \leqslant t_{k+1}$ we denote, for simplicity, 
$A_{k+1} :=  \tilde \Pi^1_{k+1}$, $B_{k+1} :=  \tilde \Pi^2_{k+1}$, and
\[
\xi := \xi_{k+1} = \tilde \xi_{k+1} =  \frac{1}{2} \mathbb{I}^{-1} \frac{A_{k+1}+B_{k+1}}{2} + \frac{1}{2} \sum^N_{i=1} \chi_i \frac{\overline{\Delta W^i_{k+1}}}{\Delta t},
\]
where $\overline{\Delta W^i_{k+1}}$ is the truncated increment of the Wiener process $W^i$. See the definition of truncation  \eqref{truncation}. Note that $\overline{\Delta W^i_{k+1}}$ is bounded: $|\overline{\Delta W^i_{k+1}}| \leqslant D_{\Delta t}$. 

The integrator of each step is now: knowing from the previous step $\Pi_k$, we solve implicitly $A_{k+1}$ and $B_{k+1}$ with:
\begin{equation}\label{onestep}
\begin{split}
\left[ {\rm d} _{\Delta t\xi}\tau^{-1}\right]^* A_{k+1} & =  A_{k+1} + \frac{\xi \times A_{k+1}}{2} \Delta t -  \frac{(\xi \cdot A_{k+1}) \xi}{4} \Delta t^2= \Pi_k \\
\left[ {\rm d} _{\Delta t\xi}\tau^{-1}\right]^* B_{k+1}  &=  B_{k+1} + \frac{\xi \times B_{k+1}}{2} \Delta t -  \frac{(\xi \cdot B_{k+1}) \xi}{4} \Delta t^2\\
&= A_{k+1} - \frac{\xi \times A_{k+1}}{2} \Delta t -  \frac{(\xi \cdot A_{k+1}) \xi}{4} \Delta t^2 = \left[ {\rm d} _{-\Delta t\xi}\tau^{-1}\right]^* A_{k+1}
\end{split}
\end{equation}
and then the momentum of the current time step is 
\begin{equation}\label{defKplus1}
    \Pi_{k+1} = \left[ {\rm d} _{-\Delta t\xi}\tau^{-1}\right]^* B_{k+1}  =  B_{k+1} - \frac{\xi \times B_{k+1}}{2} \Delta t -  \frac{(\xi \cdot B_{k+1}) \xi}{4} \Delta t^2.
\end{equation}
We have shown earlier that $||\Pi_k||$ is a constant due to the preservation of coadjoint orbits. 

\vspace{3mm}

We start with $N=1$, the case with only one stochastic Hamiltonian. With no ambiguity, we will write $\Delta W^1_{k+1}$ as $\Delta W_{k+1}$.

\paragraph{Step 1: show that $\tilde \Pi^1_{k+1}$ and $\tilde \Pi^2_{k+1}$ are uniformly bounded.}

For brevity, we will write $\xi_{k+1}$ simply as $\xi$ and $\overline{\Delta W_{k+1}}$ as $\overline{\Delta W}$ in this step. 

By analysing the equations \eqref{onestep}, we can observe that $(A_{k+1},B_{k+1})$ is the fixed point of the function $F(A,B)$ given by: 
\begin{equation} \label{contractingfunction}
\left(
\begin{aligned}
    A\\
    B
\end{aligned}
\right)
\xrightarrow{F}
\left(
\begin{aligned}
   &\Pi_k  - \frac{\xi \times A}{2} \Delta t +  \frac{(\xi \cdot A) \xi}{4} \Delta t^2 \\
    \Pi_k  - ( &\xi \times A +\frac{\xi \times B}{2} ) \Delta t +  \frac{(\xi \cdot B) \xi}{4} \Delta t^2
\end{aligned}
\right).
\end{equation}
We will show that the function $F(A,B)$ is a contracting mapping on $B_{2R} \times B_{2R}$, where $B_{2R}$ is the ball of radius $2R = 2\|\Pi_k\|$ centred at the origin, for small enough $\Delta t$.

First of all, we have the inequalities
\[
\left\|\Pi_k  - \frac{\xi \times A}{2} \Delta t +  \frac{(\xi \cdot A) \xi}{4} \Delta t^2 \right\| \leqslant \| \Pi_k \| + \frac{\Delta t}{2} \|\xi\| \|A\| + \frac{\Delta t^2}{4} \|\xi\|^2 \|A\|,
\]
\[
\left\| \Pi_k  - ( \xi \times A +\frac{\xi \times B}{2} ) \Delta t +  \frac{(\xi \cdot B) \xi}{4} \Delta t^2 \right\| \leqslant \| \Pi_k \| + \Delta t \|\xi\| \bigg( \|A\| + \frac{\|B\|}{2}\bigg)+ \frac{\Delta t^2}{4} \|\xi\|^2 \|B\|,
\]
and
\[
\| \xi \| = \left\| \frac{1}{2} \mathbb{I}^{-1} \frac{A+B}{2} + \frac{1}{2} \chi \frac{\overline{\Delta W}}{\Delta t} \right\| \leqslant \frac{1}{2} \left( \left\| \mathbb{I}^{-1} \right\| \frac{\|A\|+\|B\|}{2}+\| \chi  \| \frac{|\overline{\Delta W}|}{\Delta t} \right).
\]
With $\| \Pi_k \| = R$, $\| A \| \leqslant 2R$ and $\| B \| \leqslant 2R$, we can choose $\Delta t$ small enough so that $\|\mathbb{I}^{-1} \| R \Delta t + \| \chi \| D_{\Delta t} \leqslant \frac{1}{2}$. Under such condition, $F(A,B) \in B_{2R} \times B_{2R}$, so $F(A,B)$ is a function from $B_{2R} \times B_{2R}$ to $B_{2R} \times B_{2R}$. 

On the other hand, we have that for $(A,B), (A',B') \in B_{2R} \times B_{2R}$,
\[
\begin{aligned}
\|F(A,B) - F(A',B')\| &\leqslant \frac{\Delta t}{2} \Bigg[ \| \xi \|\Big( 3 \| A-A' \| + \| B-B' \| \Big) + \| \xi - \xi' \| \bigg( 3 \| A' \| + \|B'\| \bigg) \Bigg]    \\
&\quad + \frac{\Delta t^2}{4} \Bigg[ \| \xi\|^2 \Big( \| A-A' \| + \| B-B' \| \Big) + \| \xi\| \| \xi - \xi' \|  \Big( \| A' \| + \| B' \| \Big) \\
&\hspace{2cm}+ \| \xi'\| \| \xi - \xi' \|  \Big( \| A' \| + \| B' \| \Big) \Bigg],
\end{aligned}
\]
where we denote
\[
\xi':= \frac{1}{2} \mathbb{I}^{-1} \frac{A'+B'}{2} + \frac{1}{2} \chi \frac{\overline{\Delta W}}{\Delta t}.
\]
We have that
\[
\xi - \xi' = \frac{1}{2} \mathbb{I}^{-1} \frac{(A-A')+(B-B')}{2}\quad\text{and}\quad \| \xi - \xi' \| \leqslant \frac{1}{4} \| \mathbb{I}^{-1} \| ( \| A-A' \| + \| B-B' \|).
\]
We have furthermore $ \| A-A' \| \leqslant \|(A,B) - (A',B') \| $ and $\| B-B' \| \leqslant \|(A,B) - (A',B') \|$. Note that, $A$, $B$, $A'$, $B'$, $\xi$, $\xi'$ are all bounded with $(A,B), (A',B') \in B_{2R} \times B_{2R}$. Combining the inequalities above, we can choose small enough $\Delta t$ so that 
\begin{equation}\label{contractmap}
 \|F(A,B) - F(A',B')\| \leqslant \lambda \| (A,B) - (A',B') \|,
\end{equation}
for some $\lambda < 1$.

\begin{remark}\rm
For a given positive $\lambda$, one sufficient condition for the inequality \eqref{contractmap} is that $\frac{1}{2}x^2 + xy + 2x + 4y \leqslant \lambda$, where $x := \Delta t \; \mathrm{max} \{ \| \chi \|, \| \chi' \| \}$ and $y:= \| \mathbb{I}^{-1}\| R \Delta t$. Analyzing this inequality, we know that there exist positive $x_{\lambda}$ and $y_{\lambda}$ such that the inequality is satisfied whenever $x \leqslant x_{\lambda}$ and $y \leqslant y_{\lambda}$. We can thus choose $\Delta t$ that satisfies both $\|\mathbb{I}^{-1} \| R \Delta t + \| \chi \| D_{\Delta t} \leqslant x_{\lambda}$ and $\| \mathbb{I}^{-1}\| R \Delta t \leqslant y_{\lambda}$. 
\end{remark}

We have shown that for small enough $\Delta t$, the function $F(A,B)$ is a contracting mapping. According to the fixed point theorem, there exists a unique fixed point. Thus the implicit equations \eqref{onestep} have unique roots $(A,B)$ in $B_{2R} \times B_{2R}$. Numerically, we can also illustrate the existence and the boundedness of solutions $(A,B)$, as shown in the figures below for different $\overline{\Delta W}$. 

\begin{figure}[h]
    \centering
    \begin{subfigure}[b]{0.3\textwidth}
        \centering
        \includegraphics[width=\textwidth]{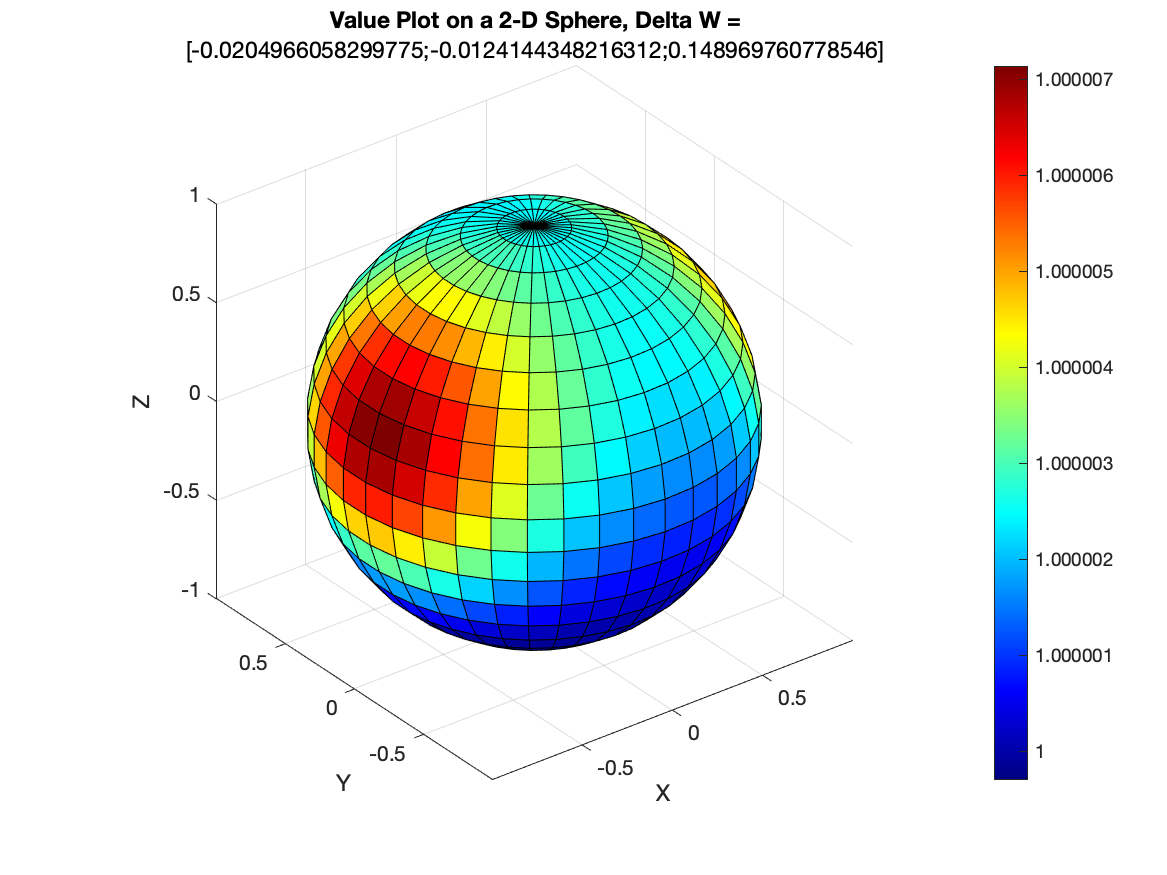}
    \end{subfigure}
    \hfill
    \begin{subfigure}[b]{0.3\textwidth}
        \centering
        \includegraphics[width=\textwidth]{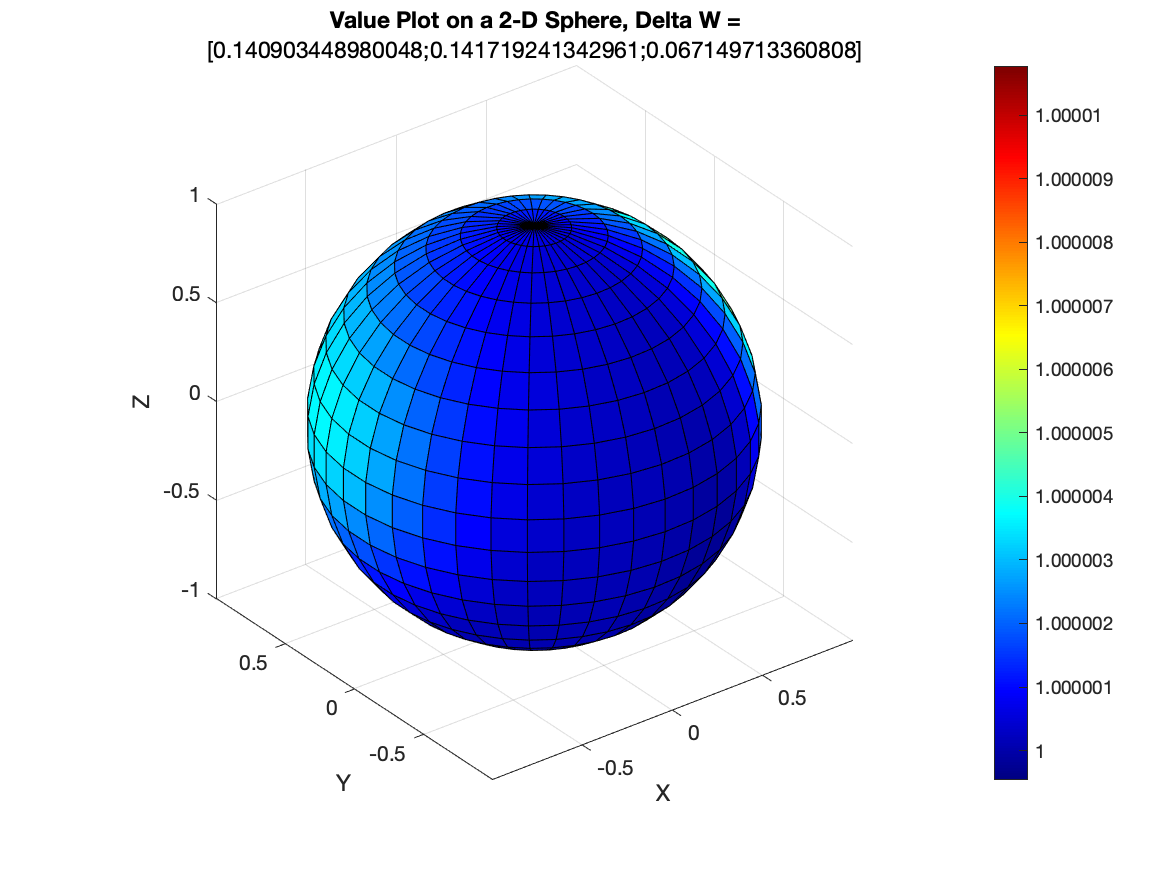}
    \end{subfigure}
    \hfill
    \begin{subfigure}[b]{0.3\textwidth}
        \centering
        \includegraphics[width=\textwidth]{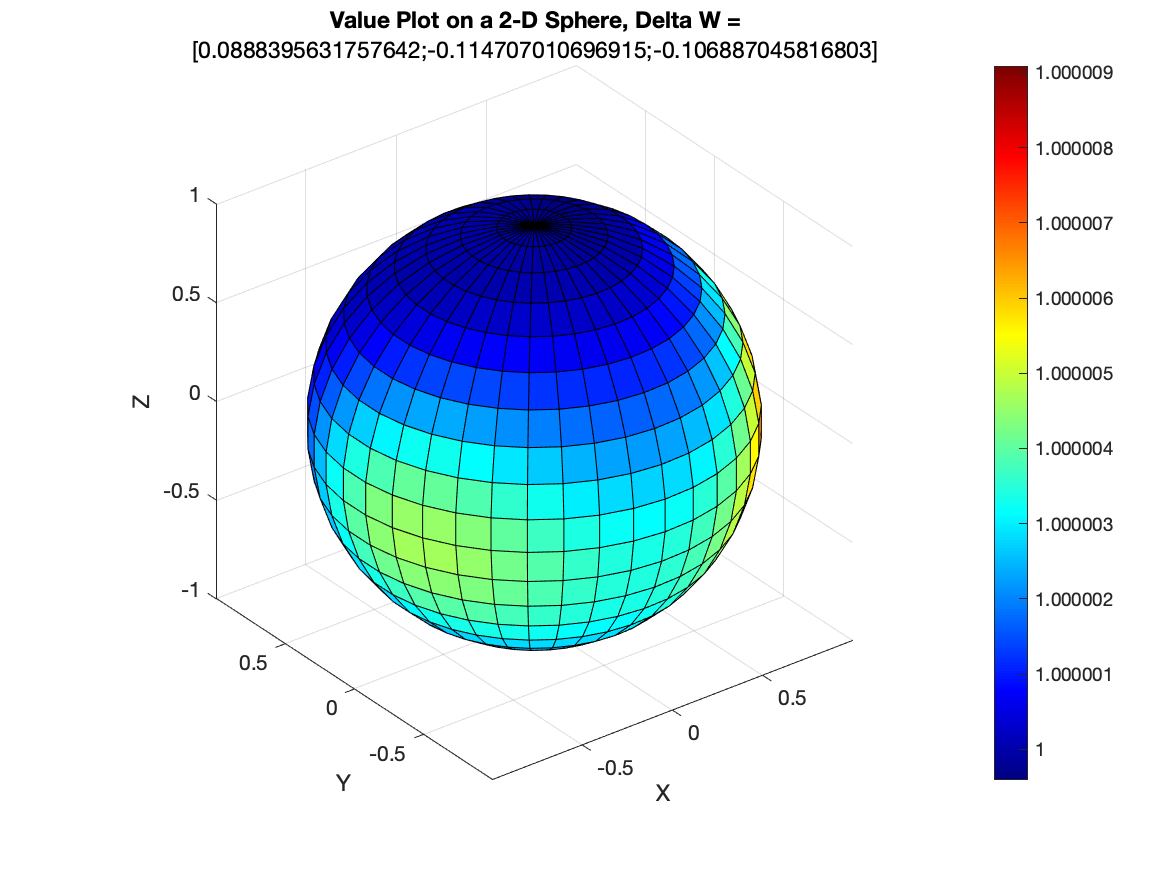}
    \end{subfigure}
    \caption{$||(A,B)||$ as function of $\Pi$ on the coadjoint orbit, $\Delta t = 0.01$, $N=3$, $\overline{\Delta W}$ is indicated in the title of each figure.}
\end{figure}

\paragraph{Step 2: define a continuous stochastic process $\hat \Pi(t)$ from the discrete $\Pi_k$ and estimate their difference.}

Departing from the discrete $\Pi_k$ (check \eqref{defKplus1} for the relationship between $\Pi_k$ and $(A_k, B_k)$), we define the continuous stochastic process $\hat{\Pi}(t)$ as follows:
\begin{equation}\label{deftildePi}
\hat{\Pi}(t) =  \hat{\Pi}(t_k) - \int^t_{t_k}  \mathbb{I}^{-1} \Pi_k \times \Pi_k {\rm d}s - \int^t_{t_k} \chi \times \Pi_k \circ {\rm d} W(s) + \int^t_{t_k} \int^s_{t_k}  \chi \times (\chi \times \Pi_k) \circ {\rm d} W(\tau) \circ {\rm d} W(s), 
\end{equation}
for $t_k < t \leqslant t_{k+1}$.

\begin{remark}\rm
\sloppy
The definition of the continuous $\hat\Pi$ is inspired by the need of estimating $ E\left[ \sup_{0 \leqslant k \leqslant K} \| \Pi_k - \pi(t_k) \|^2  \right] $. $\hat\Pi$ will serve as a middle step and we will estimate $ E\left[ \sup_{0 \leqslant k \leqslant K} \| \Pi_k - \hat \Pi(t_k)) \|^2  \right] $ and $ E\left[ \sup_{0 \leqslant k \leqslant K} \| \hat \Pi(t_k) - \pi(t_k)) \|^2  \right] $ respectively. 
\end{remark}

Note that $\hat{\Pi}(t)$ is an $\mathcal{F}_t$-adapted process. Equivalently, the expression of $\hat{\Pi}(t)$ can be written as: 
\begin{align*}
\hat{\Pi}(t) = \pi_0 &-  \sum^{k-1}_{n=0}  \left(\mathbb{I}^{-1} \Pi_n \times \Pi_n \right) \Delta t - \int^t_{t_k}  \mathbb{I}^{-1} \Pi_k \times \Pi_k {\rm d}s\\
& -\sum^{k-1}_{n=0} \left( \chi \times \Pi_n \right) \Delta W_{n+1} -\int^t_{t_k} \chi \times \Pi_k \circ {\rm d} W(s) \\
&+ \sum^{k-1}_{n=0} \frac{1}{2}  \chi \times (\chi \times \Pi_n) \Delta W_{n+1}^2 + \int^t_{t_k} \int^s_{t_k}  \chi \times (\chi \times \Pi_k) \circ {\rm d} W(\tau) \circ {\rm d} W(s).
\end{align*}
Note that we have used the following formula for the double Stratonovich integral:
\[
\int^{t_{n+1}}_{t_n}\!\! \int^s_{t_n}  \circ \,{\rm d} W(\tau) \circ {\rm d} W(s) = \frac{1}{2} ( W_{n+1} - W_n)^2 =  \frac{1}{2} \Delta W_{n+1}^2.
\]

When evaluated at discrete time $t_k$, $\hat{\Pi}(t)$ had the expression: 
\begin{equation}\label{hatPik}
\hat{\Pi}(t_k) =  \pi_0 -  \sum^{k-1}_{n=0}  \left(\mathbb{I}^{-1} \Pi_n \times \Pi_n \right) \Delta t -\sum^{k-1}_{n=0} \left( \chi \times \Pi_n \right) \Delta W_{n+1} + \sum^{k-1}_{n=0} \frac{1}{2}  \chi \times (\chi \times \Pi_n) \Delta W_{n+1}^2  
\end{equation}

On the other hand, from \eqref{onestep} and \eqref{defKplus1}, $\Pi_k$ has the expression:
\begin{equation}\label{Pik}
\begin{aligned}
\Pi_k &= \pi_0 -  \sum^{k-1}_{n=0}  \left(\mathbb{I}^{-1} \frac{A_{n+1}+B_{n+1}}{2}\times \frac{A_{n+1}+B_{n+1}}{2} \right) \Delta t\\
& \qquad\qquad -\sum^{k-1}_{n=0} \left( \chi \times \frac{A_{n+1}+B_{n+1}}{2} \right) \overline{\Delta W_{n+1}} 
\end{aligned}
\end{equation}
 
Using the two equations above, we would like to calculate the difference between $\hat{\Pi}(t_k)$ and $\Pi_k$. In the following we denote $P_{n+1} := \frac{A_{n+1}+B_{n+1}}{2}$. From \eqref{onestep}, we can express $A_{n+1}$ and $B_{n+1}$ in terms of $\Pi_n$: 
\begin{equation}\label{p_expansion}
\begin{aligned}
A_{n+1} &= \Pi_n  - \frac{1}{2} \mathbb{I}^{-1} P_{n+1} \times \frac{A_{n+1}}{2} \Delta t -  \frac{1}{2} \chi \times  \frac{A_{n+1}}{2} \overline{\Delta W_{n+1}}  + R_{A}  \\
B_{n+1} &= \Pi_n  - \frac{1}{2} \mathbb{I}^{-1} P_{n+1} \times \frac{2A_{n+1}+B_{n+1}}{2} \Delta t -  \frac{1}{2} \chi \times \frac{2A_{n+1}+B_{n+1}}{2} \overline{\Delta W_{n+1}} + R_{B},
\end{aligned}
\end{equation}
where  
\begin{align*}
R_{A}&= \frac{1}{16} (\mathbb{I}^{-1} P_{n+1} \cdot A_{n+1} ) \mathbb{I}^{-1} P_{n+1} \Delta t^2  \\
&\quad+ \frac{1}{8} \Big( (\chi \cdot A_{n+1} ) \mathbb{I}^{-1} P_{n+1} + (\mathbb{I}^{-1} P_{n+1} \cdot A_{n+1} ) \chi \Big)\Delta t \overline{\Delta W_{n+1}}  +\frac{1}{8} (\chi \cdot A_{n+1} ) \chi \overline{\Delta W_{n+1}} ^2 \\
R_{B}&= \frac{1}{16} (\mathbb{I}^{-1} P_{n+1} \cdot B_{n+1} ) \mathbb{I}^{-1} P_{n+1} \Delta t^2 \\
&\quad + \frac{1}{8} \Big( (\chi \cdot B_{n+1} ) \mathbb{I}^{-1} P_{n+1} + (\mathbb{I}^{-1} P_{n+1} \cdot B_{n+1} ) \chi \Big)\Delta t \overline{\Delta W_{n+1}} 
+\frac{1}{8} (\chi \cdot B_{n+1} ) \chi \overline{\Delta W_{n+1}} ^2.
\end{align*}

Using \eqref{p_expansion}, we have
\begin{equation}\label{expansion1}
\left( \mathbb{I}^{-1} \frac{A_{n+1}+B_{n+1}}{2}\times \frac{A_{n+1}+B_{n+1}}{2} \right) \Delta t = \mathbb{I}^{-1} \Pi_n \times \Pi_n \Delta t + R_{n+1}^1  
\end{equation}
and using \eqref{p_expansion} twice, we have
\begin{equation}\label{expansion2}
\left( \chi \times \frac{A_{n+1}+B_{n+1}}{2} \right) \overline{\Delta W_{n+1}} =  \chi \times \Pi_n \overline{\Delta W_{n+1}}  -  \frac{1}{2} \chi \times ( \chi \times \Pi_n) \overline{\Delta W_{n+1}} ^2 + R_{n+1}^2,
\end{equation}
where the remainders $R_{n+1}^1$ and $R_{n+1}^2$ are polynomials of $\Delta t$, $\overline{\Delta W_{n+1}} $ with finite terms, starting from order 1.5 (in the sense of $L^2$ expectation). Let $R_{n+1} := R_{n+1}^1 + R_{n+1}^2 $, then $R_{n+1}$ is also a polynomial with finite terms:
\[
R_{n+1} = p^{(1,1)}_{n+1}(A_{n+1}, B_{n+1}) \Delta t \overline{\Delta W_{n+1}} + p^{(0,3)}_{n+1}(A_{n+1}, B_{n+1}) \overline{\Delta W_{n+1}}^3 + p^{(2,0)}_{n+1}(A_{n+1}, B_{n+1}) \Delta t^2 + ...
\]

The coefficient $p^{(i,j)}_{n+1}$ of each term $\Delta t^i \overline{\Delta W_{n+1}}^j$, on the other hand, is a polynomial function of $A_{n+1}$, $B_{n+1}$, $\mathbb{I}^{-1} $ and $\chi$ with finite terms. We can use \eqref{p_expansion} one more time to replace the variables $A_{n+1}$, $B_{n+1}$ with $\Pi_k$ in the first two terms $p^{(1,1)}_{n+1}$ and $p^{(0,3)}_{n+1}$. The new expression has the form
\begin{equation}\label{Rest}
R_{n+1} = \mathfrak{p}^{(1,1)}_{n+1}(\Pi_n) \Delta t \overline{\Delta W_{n+1}} + \mathfrak{p}^{(0,3)}_{n+1}(\Pi_n) \overline{\Delta W_{n+1}}^3 +  \mathfrak{p}^{(2,0)}_{n+1}(A_{n+1}, B_{n+1}) \Delta t^2 + ...  ,  
\end{equation}
where $\mathfrak{p}^{(1,1)}_{n+1}$ and $ \mathfrak{p}^{(0,3)}_{n+1}$ are polynomial functions of $\Pi_k$, $\mathbb{I}^{-1} $ and $\chi$, while the rest of the coefficients are polynomial functions of $A_{n+1}$, $B_{n+1}$, $\mathbb{I}^{-1} $ and $\chi$. 

Let $I$ be the finite set of the 2-tuples $(i,j)$ needed for the expression of $R_{n+1}$ as in the equation \eqref{Rest}, then the remainder term $R_{n+1}$ can be expressed as 
\[
R_{n+1} = \sum_{(i,j) \in I} \mathfrak{p}^{(i,j)}_{n+1} \Delta t^i \overline{\Delta W_{n+1}}^j,
\]
where $\mathfrak{p}^{(i,j)}_{n+1}$ are polynomials of $\Pi_n$ for $(i,j)=(1,1)$ or $(i,j)=(0,3)$, and polynomials of $A_{n+1}$ and $B_{n+1}$ for $(i,j) \in I / \{(1,1),(0,3)\}$.

Take \eqref{expansion1} and \eqref{expansion2} into \eqref{Pik} and compare it with \eqref{hatPik}, we have that:
\[
\Pi_k - \hat{\Pi}(t_k) = \sum^{k-1}_{n=0} \left[ -\left( \chi \times \Pi_n \right) ( \Delta W_{n+1} - \overline{\Delta W_{n+1}}) +  \frac{1}{2}  \chi \times (\chi \times \Pi_n) (\Delta W_{n+1}^2 - \overline{\Delta W_{n+1}}^2) + R_{n+1} \right].
\]
In order to estimate $E\left[ \sup_{0 \leqslant k \leqslant K} \left(\Pi_k - \hat{\Pi}(t_k) \right)^2 \right] $, we need the following two lemmas. 

\begin{lemma}\label{lemma2}
Taking the truncation value as $D_{\Delta t} = \sqrt{4|{\rm ln} \Delta t| \Delta t}$ in \eqref{truncation}, we have the following estimation: 
\[
E[ ( \Delta W_{n+1} - \overline{\Delta W_{n+1}})^2 ] \leqslant  \Delta t ^3, \; \textrm{and} \; E[ (\Delta W_{n+1}^2 - \overline{\Delta W_{n+1}}^2) ] \leqslant  \Delta t ^3.
\]

\end{lemma}

\begin{remark}\rm
The proof is straightforward calculation. Refer to Lemma 2.1 in \cite{MiReTr2002} for details of a more general result.
\end{remark}

\begin{lemma}\label{lemma1}
For all the index pairs $(i,j) \in I$, we have the following estimation: for any $0 \leqslant k \leqslant K$ 
 \begin{equation}\label{estimateR}
   E\left[ \sup_{0 \leqslant k \leqslant K} \left(\sum^{k-1}_{n=0} \mathfrak{p}^{(i,j)}_{n+1} \Delta t^i \overline{\Delta W_{n+1}}^j \right)^2 \right]   \leqslant C_{i,j} \Delta t^2,
\end{equation}
for some constant $C_{i,j}$ independent of $\Delta t$.
\end{lemma}
\begin{proof}
For $(i,j) = (1,1)$ and $(i,j) = (0,3)$, let $X_k$ be the discrete process
\[
X_k = \sum^{k-1}_{n=0} \mathfrak{p}^{(i,j)}_{n+1} (\Pi_n) \Delta t^i \overline{\Delta W_{n+1}}^j.
\]
Then, $X_k$ is an $\mathcal{F}_{t_k}$-adapted process. It is also a martingale, as can be shown from the calculation:
\[
E[X_{k+1} | \mathcal{F}_{t_k}] = X_{k} + E[\mathfrak{p}^{(i,j)}_{k+1} (\Pi_k) \Delta t^i \overline{\Delta W_{k+1}}^j| \mathcal{F}_{t_k}] = X_{k} + E[\overline{\Delta W_{k+1}}^j] E[\mathfrak{p}^{(i,j)}_{k+1} (\Pi_k) \Delta t^i | \mathcal{F}_{t_k}] = X_{k},
\]
where in the second equality we use the fact that $\mathfrak{p}^{(i,j)}_{k+1} (\Pi_k)$ is  $\mathcal{F}_{t_k}$-measurable and $\overline{\Delta W_{k+1}}^j$ is $\mathcal{F}_{t_k}$-independent, and in the third equality we use the fact that $j$ is an odd integer, thus from \eqref{truncation_moments}, $E[\overline{\Delta W_{k+1}}^j] = 0$. 
 
Doob's inequality states that:
\begin{align*}
 &E \left[  \sup_{0 \leqslant k \leqslant K}  \left(\sum^{k-1}_{n=0} \mathfrak{p}^{(i,j)}_{n+1} \Delta t^i \overline{\Delta W_{n+1}}^j \right)^2 \right] \leqslant 4 E\left[   \left( \sum^{K-1}_{n=0} \mathfrak{p}^{(i,j)}_{n+1} \Delta t^i \overline{\Delta W_{n+1}}^j \right)^2 \right]  \\ &= 4 E\left[ \sum^{K-1}_{n=0} \left( \mathfrak{p}^{(i,j)}_{n+1} \Delta t^i \overline{\Delta W_{n+1}}^j \right)^2 \right] + 4E\left[ \sum_{0\leqslant m<n\leqslant K-1} \left(\mathfrak{p}^{(i,j)}_{m+1} \Delta t^i \overline{\Delta W_{m+1}}^j \right) \left(\mathfrak{p}^{(i,j)}_{n+1} \Delta t^i \overline{\Delta W_{n+1}}^j \right) \right].
\end{align*}

The second term in the equation above is equal to zero, since
\begin{align*}
&E\left[ \left(\mathfrak{p}^{(i,j)}_{m+1} \Delta t^i \overline{\Delta W_{m+1}}^j \right) \left(\mathfrak{p}^{(i,j)}_{n+1} \Delta t^i \overline{\Delta W_{n+1}}^j \right) \right]\\
&= E\left[ \left(\mathfrak{p}^{(i,j)}_{m+1} \Delta t^i \overline{\Delta W_{m+1}}^j  \mathfrak{p}^{(i,j)}_{n+1} \Delta t^i \right)\right] E \left[\overline{\Delta W_{n+1}}^j | \mathcal{F}_{t_n}  \right] = 0,
\end{align*}
the first equality is because $\mathfrak{p}^{(i,j)}_{m+1} \Delta t^i \overline{\Delta W_{m+1}}^j  \mathfrak{p}^{(i,j)}_{n+1} \Delta t^i$ is $\mathcal{F}_{t_n}$-measurable, and the second because $E \left[\overline{\Delta W_{n+1}}^j | \mathcal{F}_{t_n} \right] = 0$, since $\overline{\Delta W_{n+1}}$ is $\mathcal{F}_{t_n}$-independent and $j$ is odd. 

The coefficients $\mathfrak{p}^{(i,j)}_{n+1}$ are also uniformly bounded on $[0,T]$ since $\Pi_n$ is on the coadjoint orbit, the sphere of radius $\pi_0$. Thus, 
\begin{align*}
 E&\left[ \left(\sum^{K-1}_{n=0} \mathfrak{p}^{(i,j)}_{n+1} \Delta t^i \overline{\Delta W_{n+1}}^j \right)^2 \right] =  E\left[ \sum^{K-1}_{n=0} \left(\mathfrak{p}^{(i,j)}_{n+1} \Delta t^i \overline{\Delta W_{n+1}}^j \right)^2 \right]
 \leqslant C_1 \Delta t^{2i} E\left[  \sum^{K-1}_{n=0} \overline{\Delta W_{n+1}}^{2j} \right] \\
 & = C_1 \Delta t^{2i}  \sum^{K-1}_{n=0}  E \left[\overline{\Delta W_{n+1}}^{2j} \right] \leqslant C_2 \Delta t^{2i} K \Delta t^j = C_2 T \Delta t ^{2i + j - 1} = C_2 T \Delta t ^{2},
\end{align*}
where we used the estimation $E \left[\overline{\Delta W_{n+1}}^{2j} \right] \leqslant C \Delta t^j$. Check \eqref{truncation_moments}.

For $(i,j) \in I / \{(1,1),(0,3)\}$, note that the coefficients $\mathfrak{p}^{(i,j)}_{n+1}$ are polynomials of $A_{n+1}$, $B_{n+1}$, $\mathbb{I}^{-1} $ and $\chi$. We know from step 1 that $A_{n+1}$, $B_{n+1}$ are uniformly bounded on $[0,T]$. So are $\mathfrak{p}^{(i,j)}_{n+1}$. We have the following estimation: 
\begin{align*}
 E &\left[  \sup_{0 \leqslant k \leqslant K} \left(\sum^{k-1}_{n=0} \mathfrak{p}^{(i,j)}_{n+1} \Delta t^i \overline{\Delta W_{n+1}}^j \right)^2 \right]  \leqslant   E\left[ \sup_{0 \leqslant k \leqslant K} \sum^{k-1}_{n=0} k \left( \mathfrak{p}^{(i,j)}_{n+1} \Delta t^i \overline{\Delta W_{n+1}}^j \right)^2 \right] \\
 &\leqslant K \sum^{K-1}_{n=0} E\left[ \left( \mathfrak{p}^{(i,j)}_{n+1} \Delta t^i \overline{\Delta W_{n+1}}^j \right)^2 \right] \leqslant C K^2 \Delta t ^{2i + j} = C T^2  \Delta t ^{2i + j -2}.
\end{align*}
Note that for $(i,j) \in I / \{(1,1),(0,3)\}$, $2i + j -2 \geqslant 2$. The proof is complete.
\end{proof}

Now we are ready to prove the next proposition.

\begin{proposition}\label{estidiscrete}
Let $\Pi_k$ be the discrete solution of the  stochastic midpoint integrator of the rigid body, and let $\tilde \Pi(t_k)$
be defined as in \eqref{deftildePi}. We have the following estimation: 
 \begin{equation}\label{step2estim}
 E \left[ \sup_{0 \leqslant k \leqslant K} \left(\Pi_k - \hat{\Pi}(t_k) \right)^2 \right] \leqslant C \Delta t^2.
 \end{equation}
\end{proposition}
\begin{proof}
Comparing the expression of $\hat \Pi (t_k)$ as in \eqref{hatPik} and the expression of $\Pi_k$ as in \eqref{Pik}, we have that 
\begin{align*}
&E \left[ \sup_{0 \leqslant k \leqslant K}  \left(\Pi_k - \hat{\Pi}(t_k) \right)^2 \right] \\
=&E \left[ \sup_{0 \leqslant k \leqslant K} \left(\sum^{k-1}_{n=0}\left[ -\left( \chi \times \Pi_n \right) ( \Delta W_{n+1} - \overline{\Delta W_{n+1}}) +  \frac{1}{2}  \chi \times (\chi \times \Pi_n) (\Delta W_{n+1}^2 - \overline{\Delta W_{n+1}}^2) + R_{n+1} \right] \right)^2 \right] \\
\leqslant &3 E \left[ \sup_{0 \leqslant k \leqslant K} \left( \sum^{k-1}_{n=0} -\left( \chi \times \Pi_n \right) ( \Delta W_{n+1} - \overline{\Delta W_{n+1}})\right)^2\right] \\ &+3 E \left[ \sup_{0 \leqslant k \leqslant K} \left( \sum^{k-1}_{n=0}  \frac{1}{2}  \chi \times (\chi \times \Pi_n) (\Delta W_{n+1}^2 - \overline{\Delta W_{n+1}}^2)\right)^2\right] +3 E \left[ \sup_{0 \leqslant k \leqslant K} \left( \sum^{k-1}_{n=0} R_{n+1}\right)^2\right].
\end{align*}

From Lemma \ref{lemma2}, we have that
\begin{align*}
 &E \left[ \sup_{0 \leqslant k \leqslant K} \left( \sum^{k-1}_{n=0} -\left( \chi \times \Pi_n \right) ( \Delta W_{n+1} - \overline{\Delta W_{n+1}})\right)^2\right] \\
 \leqslant  &E \left[ \sup_{0 \leqslant k \leqslant K} K \|\chi \| \|\pi_0 \| \sum^{k-1}_{n=0} ( \Delta W_{n+1} - \overline{\Delta W_{n+1}})^2\right]  \leqslant C K \sum^{k-1}_{n=0} E\left[ ( \Delta W_{n+1} - \overline{\Delta W_{n+1}})^2\right] \\ \leqslant & C K \Delta t ^3 \leqslant CT \Delta t^2,
\end{align*}
and that 
\begin{align*}
 &E \left[ \sup_{0 \leqslant k \leqslant K} \left( \sum^{k-1}_{n=0}  \frac{1}{2}  \chi \times (\chi \times \Pi_n) (\Delta W_{n+1}^2 - \overline{\Delta W_{n+1}}^2)\right)^2\right] \\ \leqslant & E \left[ \sup_{0 \leqslant k \leqslant K}  \frac{K}{2}  \| \chi \|^2 \| \| \pi_0\| \sum^{k-1}_{n=0} (\Delta W_{n+1}^2 - \overline{\Delta W_{n+1}}^2)^2\right] \leqslant CK \sum^{k-1}_{n=0} E \left[ (\Delta W_{n+1}^2 - \overline{\Delta W_{n+1}}^2)^2\right] \\
\leqslant & CK \Delta t ^6  \leqslant CT \Delta t^5.
\end{align*}

From Lemma \ref{lemma1}, we have that 
\begin{align*}
&E \left[ \sup_{0 \leqslant k \leqslant K} \left( \sum^{k-1}_{n=0} R_{n+1}\right)^2\right]\\
= &E \left[  \sup_{0 \leqslant k \leqslant K} \left(\sum_{(i,j) \in I} \sum^{k-1}_{n=0} \mathfrak{p}^{(i,j)}_{n+1} \Delta t^i \overline{\Delta W_{n+1}}^j \right)^2 \right]   \\
\leqslant &E \left[  \sup_{0 \leqslant k \leqslant K} |I| \sum_{(i,j) \in I} \left( \sum^{k-1}_{n=0} \mathfrak{p}^{(i,j)}_{n+1} \Delta t^i \overline{\Delta W_{n+1}}^j \right)^2 \right] \\
\leqslant  &|I| \sum_{(i,j) \in I} E \left[  \sup_{0 \leqslant k \leqslant K} \left( \sum^{k-1}_{n=0} \mathfrak{p}^{(i,j)}_{n+1} \Delta t^i \overline{\Delta W_{n+1}}^j \right)^2 \right] \\
\leqslant & |I| \sum_{(i,j) \in I} C_{i,j} \Delta t^2 =: C \Delta t^2.
\end{align*}

Summarizing the three inequalities, we have that
\[
E \left[ \sup_{0 \leqslant k \leqslant K}  \left(\Pi_k - \hat{\Pi}(t_k) \right)^2 \right] \leqslant C \Delta t^2.
\]

\end{proof}

\paragraph{Step 3: compare $\hat \Pi$ and $\pi$.} The continuous solution satisfies 
\[
\pi(t)  = \pi(t_k) - \int^{t}_{t_k} \mathbb{I}^{-1} \pi(t) \times \pi (t)  \; {\rm d}s -  \int^{t}_{t_k} \chi \times  \pi(t)  \; \circ {\rm d} W(s).
\]
for $t_k \leqslant t < t_{k+1}$. The Taylor-Stratonovich formula (see \cite{KoPl1992}) gives that: 
\begin{equation}\label{Taylorexpa}
\begin{aligned}
\pi(t)  &= \pi(t_k) - \int^{t}_{t_k} \mathbb{I}^{-1} \pi(t_k) \times \pi (t_k)  \; {\rm d}s -  \int^{t}_{t_k} \chi \times  \pi(t_k)  \; \circ {\rm d} W(s) \\
&\qquad + \int^{t}_{t_k} \int^{s}_{t_k}\chi \times ( \chi \times \pi(t_k) ) \;  \circ {\rm d} W(\tau) \; \circ {\rm d} W(s) + R_k(t_k,t),  
\end{aligned}
\end{equation}
where the remainder term is 
\begin{align*}
&R_k(t_k,t) = \int^{t}_{t_k}  \int^{s}_{t_k} J_{00}(\sigma) \,  {\rm d} \sigma \, {\rm d} s + \int^{t}_{t_k} \int^{s}_{t_k} J_{01}(\sigma)    \circ {\rm d} W(\sigma) \, {\rm d} s + \int^{t}_{t_k}  \int^{s}_{t_k}  J_{10} (\sigma)\, {\rm d} \sigma  \circ {\rm d} W(s)  \\
&  +  \int^{t}_{t_k}  \int^{s}_{t_k} \int^{\sigma}_{t_k}  J_{110} (\tau)\, {\rm d} \tau \circ {\rm d} W(\sigma) \circ {\rm d} W(s) + \int^{t}_{t_k}  \int^{s}_{t_k} \int^{\sigma}_{t_k}  J_{111} (\tau)\, \circ {\rm d} W(\tau) \circ {\rm d} W(\sigma) \circ {\rm d} W(s),
\end{align*}
with
\begin{align*}
J_{00}(\cdot) &= -  \frac{\partial \left( \mathbb{I}^{-1}\pi \times \pi\right)}{\partial \pi} \cdot \left(-\mathbb{I}^{-1}\pi \times \pi\right) = \mathbb{I}^{-1} (\mathbb{I}^{-1}\pi \times \pi) \times \pi + \mathbb{I}^{-1}\pi \times ( \mathbb{I}^{-1}\pi \times \pi)\\
J_{01}(\cdot) &= - \frac{\partial \left( \mathbb{I}^{-1}\pi \times \pi\right)}{\partial \pi} \cdot \left( - \chi \times \pi \right) =  \mathbb{I}^{-1} ( \chi \times \pi) \times \pi + \mathbb{I}^{-1}\pi \times (  \chi \times \pi) \\
J_{10}(\cdot) &= - \frac{\partial \left( \chi \times \pi \right)}{\partial \pi} \cdot \left( - \mathbb{I}^{-1}\pi \times \pi \right) = \chi \times \left( \mathbb{I}^{-1}\pi \times \pi \right) \\
J_{110}(\cdot) &= \frac{\partial \left( \chi \times \left( \chi \times \pi\right) \right)}{\partial \pi} \cdot \left( - \mathbb{I}^{-1}\pi \times \pi \right) = - \chi \times \left( \chi \times \left(  \mathbb{I}^{-1}\pi \times \pi \right)  \right) \\
J_{111}(\cdot)  &= \frac{\partial \left( \chi \times \left( \chi \times \pi\right) \right)}{\partial \pi} \cdot \left( - \chi \times \pi \right) = - \chi \times \left( \chi \times \left(  \chi \times \pi  \right)  \right) .
\end{align*}
Here we omit the argument $(t)$ from the function $\pi(t)$ for brevity.

By comparing the Taylor-Stratonovich expansion of $\pi(t)$ \eqref{Taylorexpa} and the definition of $\hat{\Pi}(t)$ \eqref{deftildePi}, we can express the difference as:
\begin{equation}\label{difference}
\begin{split}
\pi(t) - \hat{\Pi}(t) =& \sum^{k-1}_{n=0} \int^{t_{n+1}}_{t_n} \left(\mathbb{I}^{-1} \Pi_n \times \Pi_n - \mathbb{I}^{-1} \pi(t_n) \times \pi(t_n)  \right) {\rm d} s\\
&+ \sum^{k-1}_{n=0}  \int^t_{t_k}  \left( \mathbb{I}^{-1} \Pi_k \times \Pi_k - \mathbb{I}^{-1} \pi(t_k) \times \pi(t_k) \right) {\rm d}s\\
  & +\sum^{k-1}_{n=0}  \int^{t_{n+1}}_{t_n} \chi \times \left( \Pi_n - \pi(t_n)\right) \circ {\rm d} W(s) + \int^t_{t_k} \chi \times \left( \Pi_k - \pi(t_k) \right) \circ {\rm d} W(s) \\
  & -\sum^{k-1}_{n=0} \int^{t_{n+1}}_{t_n} \int^{t_s}_{t_n}  \chi \times (\chi \times \left( \Pi_n - \pi(t_n)\right) )  \circ {\rm d} W(\tau) \circ {\rm d} W(s) \\
  & - \int^t_{t_k} \int^s_{t_k}  \chi \times (\chi \times \left( \Pi_k - \pi(t_k)\right) ) \circ {\rm d} W(\tau) \circ {\rm d} W(s) \\
  & +\sum^{k-1}_{n=0} R_n (t_n,t_{n+1}) \; + \; R_k (t_k, t),
\end{split}
\end{equation}
for $t_k \leqslant t < t_{k+1}$.

\begin{lemma}\label{multistraestimate}
Let $\alpha = (\alpha_1 \alpha_2 .. \alpha_m)$ be a multi-index comprised of 0 and 1, and define the multi-Stratonovich integral $S_{\alpha} [g(\cdot)] _ {s,\tau} $ of some integrable function $g$  as
\[
S_{\alpha} [g(\cdot)] _ {s,t} = \int^t_{s} .. \int^{\sigma_2}_{s} g(\sigma_1) \circ {\rm d} W^{\alpha_1} (\sigma_1) .. \circ {\rm d} W^{\alpha_m} (\sigma_m),
\]
where with an abuse of notation we have noted ${\rm d} s$ with $\circ \,{\rm d} W^{0} (s)$ and $\circ \,{\rm d} W(s)$ with $\circ \, {\rm d} W^{1} (s)$.  

Then there is the estimation: 
\begin{align*}
&E \left[ \sup_{t \in [0,u]} \left(  \sum^{k-1}_{n=0} S_{\alpha} [g(\cdot)] _ {t_n,t_{n+1}} \; + \; S_{\alpha} [g(\cdot)] _ {t_k,t} \right)^2 \right]\\
&\leqslant 
\begin{cases}
    & C \Delta t ^{2(l(\alpha) -1)} \int^u_0 E \left[  \sup_{t \in [0,s]} \left( \| g(t) \|^2 \right)\right] {\rm d}s  \quad\mathrm{when} \, n(\alpha) = l(\alpha)\\
    & C \Delta t ^{l(\alpha) + n(\alpha) -1} \int^u_0 E \left[ \sup_{t \in [0,s]} \left( \| g(t) \|^2 \right)\right] {\rm d}s \quad \mathrm{when} \, n(\alpha) \neq l(\alpha),
\end{cases}  
\end{align*}
where $l(\alpha)$ is the length of the multi-index $\alpha$ and $n(\alpha)$ is the number of zeros in the multi-index $\alpha$.

\end{lemma}

This lemma is the Stratonovich integral version of Lemma 10.8.1 of \cite{KoPl1992}.  In the case $\alpha = (1)$, this becomes the Burkholder-Davis-Gundys inequality. 

From the formula \eqref{difference}, we can express $\pi(t) - \hat{\Pi}(t) $ as :
\begin{equation} \label{shortexpress}
\begin{aligned}
\pi(t) - \hat{\Pi}(t) &= \sum_{\alpha \in \mathcal{A}} \left( \sum^{k-1}_{n=0} S_{\alpha} \left[G_{\alpha}(\cdot) \right]_{t_n,t_{n+1}} +  S_{\alpha} \left[G_{\alpha}(\cdot) \right]_{t_k,t} \right) \\
&\qquad +  \sum_{\alpha \in \bar{\mathcal{A}}} \left( \sum^{k-1}_{n=0} S_{\alpha} \left[J_{\alpha}(\cdot) \right]_{t_n,t_{n+1}} +  S_{\alpha} \left[J_{\alpha}(\cdot) \right]_{t_k,t} \right),
\end{aligned}
\end{equation}
where $\mathcal{A} = \left\{ (0),(1),(11)\right\}$ and,
\begin{align*}
G_{(0)} (t) &= \mathbb{I}^{-1} \Pi_n \times \Pi_n - \mathbb{I}^{-1} \pi(t_n) \times \pi(t_n) \\
G_{(1)} (t) &= \chi \times (\Pi_n - \pi(t_n)) \\
G_{(11)} (t) &= \chi \times (\chi \times (\Pi_n - \pi(t_n))),
\end{align*}
for $t_n \leqslant t < t_{n+1}$; and $\bar{\mathcal{A}} = \left\{ (00),(01),(10), (110), (111)\right\}$, with $J_{\alpha}$ defined above. 

We are ready to give estimation to $E[\sup_{t \in [0,u]} \| \pi(t) - \hat{\Pi}(t)  \|^2]$. 

\begin{proposition}\label{differencehatpi}
We have the following estimation:
\[
 E\left[ \sup_{t \in [0,u]} \| \pi(t) - \hat{\Pi}(t)  \|^2  \right] \leqslant C \Delta t^2,
\]
for any $0 \leqslant u \leqslant T$ and the coefficient $C$ does not depend on $u$.
\end{proposition}

\begin{proof}
In the following calculation, the constant $C$ may vary from line to line. 
Let
\[
F(u) := E\left[ \sup_{t \in [0,u]} \| \pi(t) - \hat{\Pi}(t)  \|^2  \right]
\]
for $0 \leqslant u \leqslant T$.

From the expression \eqref{shortexpress} and Lemma \ref{multistraestimate}, we have the following estimation:  

\begin{align*}
F(u) \leqslant &C \left \{  \sum_{\alpha \in \mathcal{A}} E \left[ \sup_{t \in [0,u]}  \left( \sum^{k-1}_{n=0} S_{\alpha} \left[G_{\alpha}(\cdot) \right]_{t_n,t_{n+1}} +  S_{\alpha} \left[G_{\alpha}(\cdot) \right]_{t_k,t} \right)^2 \right] \right. \\
 & \; \; \left. + \sum_{\alpha \in \bar{\mathcal{A}}} E \left[ \sup_{t \in [0,u]}  \left( \sum^{k-1}_{n=0} S_{\alpha} \left[J_{\alpha}(\cdot) \right]_{t_n,t_{n+1}} +  S_{\alpha} \left[J_{\alpha}(\cdot) \right]_{t_k,t} \right)^2 \right]  \right\}  \\
 \leqslant &C  \Bigg \{ \int^{u}_0 E \left[  \sup_{t \in [0,s]} \| G_{(0)}(t) \|^2 \right] {\rm d}s +  \int^{u}_0 E \left[  \sup_{t \in [0,s]} \| G_{(1)}(t) \|^2 \right] {\rm d}s + \Delta t \int^{u}_0 E \left[  \sup_{t \in [0,s]} \| G_{(11)}(t) \|^2 \right] {\rm d}s \\
 &\; \;  + \sum_{\substack{\alpha \in \{(00), (01),\\ (10), (111)\}}} \Delta t^2 \int^{u}_0 E \left[  \sup_{t \in [0,s]} \| J_{\alpha}(t) \|^2 \right] {\rm d}s + \Delta t^3 \int^{u}_0 E \left[  \sup_{t \in [0,s]} \| J_{(110))}(t) \|^2 \right] {\rm d}s \Bigg\}.
\end{align*}

First of all, 
\[
G_{(0)} (t) =  \mathbb{I}^{-1} \Pi_n \times ( \Pi_n -\pi(t_n) ) + \mathbb{I}^{-1}(\Pi_n  -\pi(t_n) )\times \pi(t_n),
\]
thus 
\begin{align*}
\|G_{(0)} (t) \|^2 &\leqslant 2 \| \mathbb{I}^{-1} \|^2 \| \Pi_n \|^2 \|\Pi_n  -\pi(t_n) \|^2 + 2 \| \mathbb{I}^{-1} \|^2 \| \pi(t_n) \|^2 \|\Pi_n  -\pi(t_n) \|^2\\
&= 4 \| \mathbb{I}^{-1} \|^2 C_{\pi}^2 \|\Pi_n  -\pi(t_n) \|^2 ,
\end{align*}
where $C_{\pi} = \| \Pi_n \| = \| \pi\|$. 

Similarly it is also easy to see that
\[
\| G_{(1)} (t) \|^2 \leqslant \| \chi \|^2 \|\Pi_n  -\pi(t_n) \|^2, \; \| G_{(11)} (t) \|^2 \leqslant \| \chi \|^4 \|\Pi_n  -\pi(t_n) \|^2. 
\]

Thus for $\alpha \in \mathcal{A}$, applying Proposition \ref{estidiscrete}, we have that:
\begin{equation}\label{estipart1}
 \begin{aligned}
 \int^{u}_0 &E \left[  \sup_{t \in [0,s]} \| G_{\alpha}(t) \|^2 \right] {\rm d}s  \leqslant C \int^{u}_0 E \left[  \sup_{t \in [0,s]} \| \Pi_n  -\pi(t_n)  \|^2 \right] {\rm d}s   \\ 
 & \leqslant 2C \int^{u}_0 E \left[  \sup_{t \in [0,s]} \| \Pi_n  -\tilde \Pi(t_n)  \|^2 \right] {\rm d}s + 2C \int^{u}_0 E \left[  \sup_{t \in [0,s]} \| \pi(t_n)  -\tilde \Pi(t_n)  \|^2 \right] {\rm d}s \\
  & \leqslant  C_1 \Delta t^2 + C_2 \int^{u}_0 E \left[  \sup_{t \in [0,s]} \| \pi(t)  -\tilde \Pi(t)  \|^2 \right] {\rm d}s,
 \end{aligned}
\end{equation}
where $n$ is the integer such that $t_n \leqslant t < t_{n+1}$. 

On the other hand, for $\alpha \in \bar{\mathcal{A}}$, the expressions of $J_{\alpha}(t)$ are polynomials of $\mathbb{I}^{-1}$, $\chi$ and $\pi(t)$, all have fixed norms. Thus $J_{\alpha}(t)$ are bounded and,
\begin{equation}\label{estipart2}
  \int^{u}_0 E \left[  \sup_{t \in [0,s]} \| J_{\alpha}(t) \|^2 \right] {\rm d}s   \leqslant CT.
\end{equation}

Taking \eqref{estipart1} and \eqref{estipart2} into the general estimation of $F(u)$ above, we have that for small enough $\Delta t$,
\begin{equation}
  F(u) \leqslant C_1 \Delta t^2 + C_2 \int^{u}_0 F(s) {\rm d}s.
\end{equation}

Gronwall's inequality then gives that
\[
F(u)  \leqslant C \Delta t^2. 
\]
\end{proof}

\paragraph{Step 4: Compare $\Pi_k$ and $\pi(t_k)$ and conclude.}
From Step 2 and Step 3, it is now easy to see that 
\[
 E\left[ \sup_{0 \leqslant k \leqslant K} \| \Pi_k - \pi(t_k) \|^2  \right] \leqslant 2E\left[ \sup_{t \in [0,T]} \| \pi(t) - \hat{\Pi}(t)  \|^2  \right] + 2E\left[ \sup_{0 \leqslant k \leqslant K} \| \hat{\Pi}(t_k) - \Pi_k  \|^2  \right] \leqslant C \Delta t ^ 2.
\]
We have shown the meaning square convergence of order 1 for the midpoint scheme \eqref{integrator_rigid_body}.

Furthermore, we have that 
\[
E\left[\| \Pi_K - \pi(t_K) \|  \right] \leqslant \left(E\left[\| \Pi_K - \pi(t_K) \|^2  \right]\right)^{1/2} \leqslant C \Delta t .
\]
Thus the midpoint scheme also has a strong convergence of order 1.

\paragraph{Step 5: Convergence order 0.5 when $N > 1$.}
When there are more than one stochastic Hamiltonians, that is when $N > 1$ in \eqref{integrator_rigid_body}, the strong convergence order is 0.5. We will only provide a sketch of the proof below.

The main difference of this case is in step 2, where the definition of $\hat{\Pi}(t)$ is now 
\begin{equation}\label{deftildePiN}
\begin{aligned}
\hat{\Pi}(t) &=  \hat{\Pi}(t_k) -  \int^t_{t_k}  \mathbb{I}^{-1} \Pi_k \times \Pi_k {\rm d}s - \sum^{N}_{i=1}\int^t_{t_k} \chi_i \times \Pi_k \circ {\rm d} W_i(s) \\
&\qquad+ \sum^{N}_{i=1}\sum^{N}_{j=1} \int^t_{t_k} \int^s_{t_k}  \chi_i \times (\chi_j \times \Pi_k) \circ {\rm d} W_j(\tau) \circ {\rm d} W_i(s) ,
\end{aligned}
\end{equation}
for $t_k  \leqslant t < t_{k+1}$. $W_i(t)$ are independent Wiener processes. 

Note that the multiple Stratonovich integral
\[
\int^t_{t_k} \int^s_{t_k} \; \circ {\rm d} W_j(\tau) \circ {\rm d} W_i(s)
\]
cannot be expressed as a simple polynomial function of $\Delta W_i$ and $\Delta W_j$(See \cite{KoPl1992}). In this case, 
\[
\sum^{N}_{i=1}\sum^{N}_{j=1} \int^{t_{k+1}}_{t_k} \int^s_{t_k}  \chi_i \times (\chi_j \times \Pi_k) \circ  {\rm d} W_j(\tau) \circ {\rm d} W_i(s) \neq  \sum^{N}_{i=1}\sum^{N}_{j=1}  \frac{1}{2} \chi_i \times (\chi_j \times \Pi_k) \Delta W_{j,n+1} \Delta W_{i, n+1}(s). 
\]
As a result, in place of Proposition \ref{estidiscrete}, we can only have
\[
E \left[ \sup_{0 \leqslant k \leqslant K} \|\Pi_k - \hat{\Pi}(t_k) \|^2 \right] \leqslant C \Delta t,
\]
and furthermore in place of Proposition \ref{differencehatpi}, we have
\[
E\left[ \sup_{t \in [0,u]} \| \pi(t) - \hat{\Pi}(t)  \|^2  \right] \leqslant C \Delta t.
\]

We can only derive the order 0.5 strong convergence of the midpoint scheme:
\[
E\left[\| \Pi_K - \pi(t_K) \|  \right] \leqslant C \Delta t^{1/2} .
\]

\end{document}